\documentclass[a4paper,reqno,12pt]{amsart}
\usepackage{amsmath}
\usepackage{amssymb}
\usepackage{cases}
\usepackage{caption}
\usepackage{array, longtable}
\usepackage[all]{xy}
\allowdisplaybreaks[4]

\newtheorem{thm}{Theorem}[section]
\newtheorem{prop}[thm]{Proposition}
\newtheorem{prpty}[thm]{Property}
\newtheorem{lem}[thm]{Lemma}

\newtheorem{prob}[thm]{Problem}

\newtheorem{claim}{Claim}

\theoremstyle{definition}
\newtheorem{definition}[thm]{Definition}
\newtheorem{example}[thm]{Example}
\newtheorem{assum}[thm]{Assumption}

\theoremstyle{remark}
\newtheorem{remark}[thm]{Remark}
\numberwithin{equation}{section}

\newcommand{\bQ}{\mathbb{Q}}
\newcommand{\bP}{\mathbb{P}}
\newcommand\tp{{\tilde{P}}}
\newcommand\mB{{B}}
\newcommand{\bC}{{\mathbb C}}
\newcommand\OO{{\mathcal{O}}}
\newcommand{\rounddown}[1]{\left\lfloor{#1}\right\rfloor}
\newcommand{\roundup}[1]{\left\lceil{#1}\right\rceil}

\newcommand\BB{\mathcal{B}}

\newcommand\ZZ{{\mathbb{Z}}}

\newcommand\Supp{{\text{\rm Supp}}}
\newcommand\lcm{{\text{l.c.m.}}}

\newcommand\Mov{{\text{\rm Mov}}}
\newcolumntype{C}{>{$}c<{$}}
\newcolumntype{L}{>{$}l<{$}}
\newcommand\coeff{{\text{coeff}}}

\title{On the anti-canonical geometry of weak $\bQ$-Fano threefolds, II}
\author{Meng Chen and Chen Jiang}
\date{\today}
\address{\rm School of Mathematical Sciences \& Shanghai Centre for Mathematical Sciences, Fudan University, Shanghai 200433, China}
\email{mchen@fudan.edu.cn}
\address{\rm Shanghai Center for Mathematical Sciences, Fudan University, Jiangwan Campus, Shanghai, 200438, China}
\email{chenjiang@fudan.edu.cn}

\thanks{The first author was supported by National Natural Science Foundation of China (\#11571076, \#11231003, \#11421061) and Program of Shanghai Subject Chief Scientist (\#16XD1400400). The second author was supported by JSPS KAKENHI Grant Number JP16K17558 and World Premier International Research Center Initiative (WPI), MEXT, Japan.}

\begin{document}
\maketitle
\pagestyle{myheadings} \markboth{\hfill M. Chen \& C. Jiang
\hfill}{\hfill On the anti-canonical geometry of weak $\bQ$-Fano $3$-folds, II\hfill}

\begin{abstract}
By a canonical (resp. terminal) weak $\bQ$-Fano $3$-fold we mean a normal projective one with at worst canonical (resp. terminal) singularities on which the anti-canonical divisor is $\bQ$-Cartier, nef and big.  For a canonical weak $\bQ$-Fano $3$-fold $V$, we show that there exists a terminal weak $\bQ$-Fano $3$-fold $X$, being birational to $V$, such that the $m$-th anti-canonical map defined by $|-mK_{X}|$ is birational for all $m\geq 52$. As an intermediate result, we show that for any $K$-Mori fiber space $Y$ of a canonical weak $\bQ$-Fano $3$-fold, the $m$-th anti-canonical map defined by $|-mK_{Y}|$ is birational for all $m\geq 52$. 
\end{abstract} 


\section{\bf Introduction}
Throughout this paper, we work over an algebraically closed field $k$ of characteristic 0 (for instance, $k=\bC$). We adopt the standard notation in \cite{KM} and will freely use them.

A normal projective variety $X$ is called a {\it
weak $\bQ$-Fano variety} (resp. {\it $\bQ$-Fano variety}) if the anti-canonical divisor $-K_X$ is nef and big (resp. ample). 
A {\it canonical} (resp. {\it terminal)} weak $\bQ$-Fano variety is a weak $\bQ$-Fano variety with at worst canonical (resp. {terminal)} singularities. 

According to Minimal Model Program, weak $\bQ$-Fano varieties form a fundamental class among research objects of birational geometry. There are a lot of works (for instance, \cite{K, A, KMMT, S04, Prok05, BS07a, BS07b, Prok07, Prok, C, Prok13, PR16, CJ16}) which study the explicit geometry of canonical or terminal (weak) $\bQ$-Fano $3$-folds. 

Given a canonical weak $\bQ$-Fano $n$-fold $X$, the {\it $m$-th anti-canonical map} $\varphi_{-m,X}$ (or simply $\varphi_{-m}$) is the rational map defined by the linear system $|-mK_X|$. Since $-K_X$ is big, $\varphi_{-m,X}$ is a  birational map onto its image when $m$ is sufficiently large. Therefore it is interesting to find such a practical number $m_n$, independent of $X$, which guarantees the stable birationality of $\varphi_{-m_n}$. Such a number $m_3$ exists for canonical weak $\bQ$-Fano $3$-folds due to the boundedness result proved by Koll\'ar, Miyaoka, Mori, and Takagi \cite{KMMT}. In general, $m_n$ exists for canonical weak $\bQ$-Fano $n$-folds due to the recent work of Birkar \cite{Bir16a}.  It is still mysterious to the authors whether $m_n$ can be explicitly computed or effectively estimated for $n\geq 4$. 

In this paper we are interested in the explicit birational geometry of canonical weak $\bQ$-Fano $3$-folds, so it is natural to consider the following problem.

\begin{prob}\label{b P} {}Find the optimal constant $c$ such that
$\varphi_{-m}$ is birational (onto its image) for all $m\geq c$ and for all canonical weak $\bQ$-Fano
$3$-folds. 
\end{prob}

The following example suggests that $c\geq 33$.

\begin{example}[{\cite[List 16.6, No. 95]{Fletcher}}]\label{ex 33}
The general weighted hypersurface $X_{66}\subset\mathbb{P}(1,5,6,22,33)$ is a $\bQ$-factorial terminal $\mathbb{Q}$-Fano $3$-fold of Picard number one. It is clear that $\varphi_{-m}$ is birational for $m\geq 33$, but $\varphi_{-32}$ is non-birational.
\end{example}

In our previous article \cite{CJ16}\footnote{In \cite{CJ16}, a {\it $\bQ$-Fano $3$-fold} means a $\bQ$-factorial terminal $\bQ$-Fano $3$-fold of Picard number one, and a {\it weak $\bQ$-Fano $3$-fold} means a $\bQ$-factorial terminal weak $\bQ$-Fano $3$-fold.} (see also \cite{phd} for related results on generic finiteness), we showed the following theorems. 

\begin{thm}[{\cite[Theorem 1.6]{CJ16}}]\label{birationality1} 
Let $X$ be a $\bQ$-factorial terminal $\bQ$-Fano $3$-fold of Picard number one. Then $\varphi_{-m,X}$ is birational for all $m\geq 39$. 
\end{thm}

\begin{thm}[{\cite[Theorem 1.8, Remark 1.9]{CJ16}}]\label{birationality2} 
Let $X$ be a canonical weak $\bQ$-Fano $3$-fold. Then $\varphi_{-m,X}$ is birational for all $m\geq 97$. 
\end{thm}

One intuitively feels that the numerical bound ``97'' obtained in Theorem \ref{birationality2} might be far from optimal. As being indicated in \cite{C, CJ16}, the birationality problem is closely related to the following problem.

\begin{prob}\label{problem2} Given a canonical weak $\bQ$-Fano $3$-fold X, what is the smallest positive integer $\delta_1=\delta_1(X)$ satisfying $\dim\overline{\varphi_{-\delta_1}(X)}>1$?
\end{prob}

In our previous paper we have shown the following theorems. 

\begin{thm}[{\cite[Theorem 1.4]{CJ16}}]\label{main1}
Let X be a $\bQ$-factorial terminal $\mathbb{Q}$-Fano $3$-fold of Picard number one. Then there exists an integer $n_1\leq 10$ such that $\dim\overline{\varphi_{-n_1}(X)}>1$.
\end{thm}

\begin{thm}[{\cite[Theorem 1.7, Remark 1.9]{CJ16}}]\label{main2}
Let X be a canonical weak $\mathbb{Q}$-Fano $3$-fold. Then $\dim\overline{\varphi_{-n_2}(X)}>1$ for all $n_2\geq 71$.
\end{thm}

Note that Theorems \ref{birationality1} and \ref{main1} are very close to be optimal due to Example \ref{ex 33} and the following example.

\begin{example}[{\cite[List 16.7, No. 85]{Fletcher}}]
Consider the general codimension $2$ weighted complete intersection $X=X_{24,30}\subset\mathbb{P}(1,8,9,10,12,15)$ which is a $\bQ$-factorial terminal $\mathbb{Q}$-Fano $3$-fold of Picard number one. Then $\dim\overline{\varphi_{-9}(X)}>1$ while $\dim\overline{\varphi_{-8}(X)}=1$ since $P_{-8}=2$.  Clearly $\delta_1(X)=9$. 
\end{example}

The aim of this paper is to intensively study Problem \ref{problem2}. It turns out that, modulo birational equivalence, the following theorem considerably improves Theorems \ref{birationality2} and \ref{main2}:

\begin{thm}\label{main thm}
Let $V$ be a canonical weak $\bQ$-Fano $3$-fold. Then there exists a terminal weak $\bQ$-Fano $3$-fold $X$ birational to $V$ such that 
\begin{enumerate}
\item $\dim \overline{\varphi_{-m}(X)}>1$ for all $m\geq 37$;
\item $\varphi_{-m, X}$ is birational for all $m\geq 52$.
\end{enumerate}
\end{thm}

According to Minimal model program, Mori fiber spaces are fundamental classes in birational geometry, so it is also interesting to consider the relation of a $\bQ$-Fano variety and its Mori fiber spaces.
As an intermediate result, we show the following theorem.
\begin{thm}\label{main thm2}
Let $V$ be a canonical weak $\bQ$-Fano $3$-fold. Then 
for any $K$-Mori fiber space $Y$ of $V$ (see Definition \ref{def KMFS}),
\begin{enumerate}
\item $\dim \overline{\varphi_{-m}(Y)}>1$ for all $m\geq 37$;
\item $\varphi_{-m, Y}$ is birational for all $m\geq 52$.
\end{enumerate}
\end{thm}
Note that here $-K_Y$ is big, but not necessarily nef. But we can still consider the map $\varphi_{-m, Y}$ defined by $|-mK_Y|$.

\begin{remark} The birationality of anti-pluricanonical maps of weak $\bQ$-Fano varieties is not necessarily invariant under birational equivalence. This may account for the dilemma while the notion ``birational geometry'' is used to understand the geometry of $\bQ$-Fanos. 
\end{remark}

The main idea in this paper is to show that, after a birational modification, a canonical weak $\bQ$-Fano $3$-fold is birationally equivalent to another terminal weak $\bQ$-Fano $3$-fold which takes over one vertex of an interesting triangle (i.e. so-called {\it Fano--Mori triple}, see Definition \ref{MF3}). Instead of dealing with the given $\bQ$-Fano 3-fold, we treat the Fano--Mori triple which exhibits richer geometry.  More precisely, we prove Theorems \ref{main thm} and  \ref{main thm2} by showing the following two theorems.

\begin{thm}[={Proposition \ref{prop Fano=FanoMori}}]\label{thm1}
Let $V$ be a canonical weak $\bQ$-Fano $3$-fold and $Y$  a $K$-Mori fiber space of $V$. Then there exists a Fano--Mori triple $(X, Y, Z)$ containing $Y$ as the second term.
\end{thm}

\begin{thm}\label{thm2}
Let $(X, Y, Z)$ be a Fano--Mori triple in which $X$ is a terminal weak $\bQ$-Fano $3$-fold. Then
\begin{enumerate}
\item $\dim \overline{\varphi_{-m}(X)}>1$ for all $m\geq 37$;
\item $\varphi_{-m, X}$ is birational for all $m\geq 52$.
\end{enumerate}
\end{thm}

\begin{remark}
In Theorem \ref{thm2}, we can show moreover that $\varphi_{-51}$ is not birational only if $-K_X^3=1/330$ and $X$ has the Reid basket $$\{(1,2),(2,5),(1,3),(2,11)\},$$
that is, $X$ has exactly the same anti-canonical degree and singularities as $X_{66}$ in Example \ref{ex 33}.
\end{remark}
 The key step is to establish an effective method to tell when the anti-pluricanonical system of the terminal weak $\bQ$-Fano $3$-fold in a Fano--Mori triple is not composed with a pencil, which settles Problem \ref{problem2} in this case. Then we can use the birationality criterion established in \cite{CJ16} to prove the effective birationality as stated in the main theorem. 

This paper is organized as follows. In Section \ref{preliminaries}, we recall some basic knowledge. In Section \ref{sec FM}, we introduce the concept of Fano--Mori triples and their basic properties, in particular, we show that any canonical weak $\bQ$-Fano $3$-fold is birational to a Fano--Mori triple containing certain $K$-Mori fiber space. In Section \ref{sec GI}, we prove an important geometric inequality for Fano--Mori triples on which there is an anti-pluricanonical pencil. In Section \ref{sec criterion}, we collect some criteria for $|-mK|$ being not composed with a pencil and for birationality of $|-mK|$.   In Section \ref{sec proof}, we apply those criteria to prove the main theorems and, meanwhile, the structure of weighted baskets revealed in \cite{explicit} is effectively used to classify $\bQ$-Fano 3-folds with small invariants.

\section*{\bf Notation}

For the convenience of readers, we list here the notation that will be frequently used in this paper. Let $X$ be a terminal weak $\bQ$-Fano $3$-fold. 

{\footnotesize
\begin{longtable*}{LL}

\hline
\endfirsthead
\hline 
\endhead
\hline
\endfoot

\hline \hline
\endlastfoot
\varphi_{-m} & \text{the $m$-th anti-canonical map corresponding}\\
& \text{to}\   |-mK_X|\\ 
P_{-m}(X)=h^0(X,\OO_X(-mK_X))&  \text{the $m$-th anti-plurigenus of 
$X$}\\
\sim_{\bQ} & \text{$\bQ$-linear equivalence}\\
B=B_X=\{(b_i,r_i)\}& \text{the Reid basket of orbifold points of $X$}\\
r_X=\lcm\{r_i\mid r_i\in B_X\}& \text{the Gorenstein index of $X$}, \text{i.e., the}\\
&\text{Cartier index of  $K_X$}\\
r_{\max}=\max\{r_i\mid r_i\in B_X\}&\text{the maximal local index}\\
M_X=r_X(-K^3_X)&\text{see Section \ref{sec criterion}}   \\
\lambda(M_X)& \text{see Section \ref{sec criterion}}\\

\sigma(B)=\sum_i b_i&  \text{see Subsection \ref{cc method}}\\

\sigma'(B)=\sum_i\frac{b_i^2}{r_i}&\text{see Subsection \ref{cc method}} \\

\mathbb{B}=(B, \tp_{-1})& \text{a weighted basket}\\

-K^3(\mathbb{B})&\text{the volume of}\ \mathbb{B}\\

\tp_{-m}(\mathbb{B})& \text{the $m$-th anti-plurigenus of $\mathbb{B}$}\\

\{B^{(m)}\}& \text{the canonical sequence of $B$}\\

B^{(m)}=\{n^m_{b,r}\times (b,r)\}& \text{expression of $B^{(m)}$}\\

\epsilon_m(B)& \text{the number of prime packings from}\\
& B^{(m-1)}$ to $B^{(m)}\\

\sigma_5=\sum_{r\geq 5} n_{1,r}^0& \text{see Subsection \ref{cc method}}\\

\epsilon=2\sigma_5-n_{1,5}^0& \text{see Subsection \ref{cc method}}\\

\gamma(B)=\sum_{i} \frac{1}{r_i}-\sum_{i} r_i+24&\text{see \eqref{kwmt}} \\
\hline
\end{longtable*}
}
\section{\bf Preliminaries}\label{preliminaries}

Let $X$ be a canonical weak $\bQ$-Fano $3$-fold. Denote by $r_X$ the Gorenstein index of $X$ ($=$ the Cartier index of $K_X$). 
For any positive integer $m$, the number
$P_{-m}(X)=h^0(X,\OO_X(-mK_X))$ is called the {\it $m$-th anti-plurigenus} of
$X$. Clearly, since $-K_X$ is nef and big, Kawamata--Viehweg
vanishing theorem \cite[Theorem 1-2-5]{KMM} implies that
$$h^i(X, -mK_X)=h^i(X, K_X-(m+1)K_X)=0$$
for all $i>0$ and $m\geq 0$. Hence $\chi(\OO_X)=1$. 

For two linear systems $|A|$ and $|B|$, we write $|A|\preceq |B|$ if there exists an effective divisor $F$ such that $$|B|\supset |A|+F.$$ In particular, if $A\leq B$ as divisors, then $|A|\preceq |B|$.

\subsection{Rational maps defined by Weil divisors}\label{b setting}\

Consider an effective $\bQ$-Cartier Weil divisor $D$ on $X$ with $h^0(X, D)\geq 2$. We study the rational map defined by $|D|$, say 
$$X\overset{\Phi_D}{\dashrightarrow} \bP^{h^0(D)-1}$$ which is
not necessarily well-defined everywhere. By Hironaka's big
theorem, we can take successive blow-ups $\pi: W\rightarrow X$ such
that:
\begin{itemize}
\item [(i)] $W$ is smooth projective;
\item [(ii)] the movable part $|M|$ of the linear system
$|\rounddown{\pi^*(D)}|$ is base point free and, consequently,
the rational map $\gamma=\Phi_D\circ \pi$ is a morphism;
\item [(iii)] the support of the
union of $\pi_*^{-1}(D)$ and the exceptional divisors of $\pi$ is
simple normal crossings.
\end{itemize}
Let $W\overset{f}\longrightarrow \Gamma\overset{s}\longrightarrow Z$
be the Stein factorization of $\gamma$ with $Z=\gamma(W)\subset
\bP^{h^0(D)-1}$. We have the following commutative
diagram:\medskip
$$\xymatrix@=4.5em{
W\ar[d]_\pi \ar[dr]^{\gamma} \ar[r]^f& \Gamma\ar[d]^s\\
X \ar@{-->}[r]^{\Phi_D} & Z}
$$
\medskip

{\bf Case $(f_{\rm{np}})$.} If $\dim(\Gamma)\geq 2$, a general
member $S$ of $|M|$ is a smooth projective surface by
Bertini's theorem. In this case, we say that $|D|$ {\it is not composed with
a pencil of surfaces} (not composed with a pencil, for short).

{\bf Case $(f_{\rm p})$.} If $\dim(\Gamma)=1$,  then $\Gamma\cong
\bP^1$ since $g(\Gamma)\leq q(W)=q(X)=0$. Furthermore, a
general fiber $S$ of $f$ is an irreducible smooth projective surface
by Bertini's theorem. We may write
$$M=\sum_{i=1}^l S_i\sim
lS$$ where $S_i$ is a smooth fiber of $f$ for each $i$ and
$l=h^0(D)-1$. We can write
$$
|D|=|lS'|+E,
$$
where $|S'|=|\pi_*S|$ is an irreducible rational pencil, $|lS'|$ is the movable part and $E$ the fixed part. In this case, 
$|D|$ is said to {\it be composed with a (rational) pencil of surfaces} (composed with a pencil, for short).


\begin{definition} 
For another $\bQ$-Cartier Weil divisor $D'$ on $X$ satisfying $h^0(X,D') \geq 2$, we say that $|D|$ and $|D'|$ are {\it composed with the same pencil (of surfaces)} if $|D|$ and $|D'|$ are composed with pencils and, through the Stein factorization,  they define the same fibration structure $W \to \bP^1$ on some smooth model $W$. 
\end{definition}

\begin{definition} 
In the above setting, $S$ is called a {\it generic irreducible element} of $|M|$.
By abuse of notion, we also say that $S'=\pi_*(S)$ is {\it a generic irreducible element of} $\Mov|D|$ or $|D|$. 
\end{definition}

\subsection{Reid's Riemann--Roch formula}\

A {\it basket} $B$ is a collection of pairs of integers (permitting
weights), say $\{(b_i,r_i)\mid i=1, \cdots, s; b_i\ \text{is coprime
 to}\ r_i\}$. For simplicity, we will alternatively write a basket as a set of pairs with weights, 
 say for example,
$$B=\{(1,2), (1,2), (2,5)\}=\{2\times (1,2), (2,5)\}.$$

Assume $X$ to be a terminal weak $\bQ$-Fano $3$-fold. According to 
Reid \cite{YPG}, there is  a basket of orbifold points (called {\it Reid basket})
$$B_X=\bigg\{(b_i,r_i)\mid i=1,\cdots, s; 0<b_i\leq \frac{r_i}{2};b_i \text{ is coprime to } r_i\bigg\}$$
associated to $X$, where a pair $(b_i,r_i)$ corresponds to an orbifold point $Q_i$ of type $\frac{1}{r_i}(1,-1,b_i)$. 
Moreover, by Reid's Riemann--Roch formula, Kawamata--Viehweg vanishing theorem and Serre duality, one has, for any $n>0$, 
\begin{align*}
P_{-n}(X)={}&-\chi(\OO_X((n+1)K_X))\\
={}&\frac{1}{12}n(n+1)(2n+1)(-K_X^3)+(2n+1)-l(-n)
\end{align*}
where
$l(-n)=l(n+1)=\sum_i\sum_{j=1}^n\frac{\overline{jb_i}(r_i-\overline{jb_i})}{2r_i}$ and the first sum runs over all orbifold points in Reid basket. 

The above formula can be rewritten as:
\begin{align}
{P}_{-1}&{} =
\frac{1}{2}\bigg(-K_X^3+\sum_i \frac{b_i^2}{r_i}\bigg)-\frac{1}{2}\sum_i b_i+3,\label{P1K3}\\
{P}_{-m}-{P}_{-(m-1)}&{}= \frac{m^2}{2}\bigg(-K_X^3+\sum_i
\frac{b_i^2}{r_i}\bigg)-\frac{m}{2}\sum_i b_i+2-\Delta^{m}\notag
\end{align}
where $\Delta^{m}= \sum_i
(\frac{\overline{b_im}(r_i-\overline{b_im})}{2r_i}-
\frac{b_im(r_i-b_im)}{2r_i})$ for any $m\geq 2$.

\subsection{Weighted baskets according to Chen--Chen}\label{cc method}\ 

 All contents of this subsection are mainly from Chen--Chen \cite{CC, explicit}. We list them as follows:
\begin{enumerate}
\item Let $B=\{(b_i,r_i)\mid i=1,\cdots, s; 0<b_i\leq \frac{r_i}{2}; b_i
\text{\ is coprime to\ } r_i\}$ be a basket. We set $\sigma(B)=\sum_i b_i$, $\sigma'(B)=\sum_i\frac{b_i^2}{r_i}$, and $\Delta^n(B)=\sum_i
\big(\frac{\overline{b_in}(r_i-\overline{b_in})}{2r_i}-
\frac{b_in(r_i-b_in)}{2r_i}\big)$ for any integer $n>1$.

\item The new (generalized) basket
$$B'=\{(b_1+b_2, r_1+r_2), (b_3, r_3),\cdots, (b_s,r_s)\}$$ is called a
{\it packing} of $B$, denoted as $B\succeq B'$. 
We call $B\succ B'$ a {\it prime packing} if $b_1r_2-b_2r_1=1$. A composition of finite
packings is also called a packing. So the relation ``$\succeq$'' is a partial ordering on the set of baskets.

\item Note that for a terminal weak $\bQ$-Fano $3$-fold $X$, all the anti-plurigenera $P_{-n}$ can be determined by Reid basket $B_X$ and $P_{-1}(X)$. This leads to the notion of ``weighted basket''. We call a pair $\mathbb{B}=(B, \tilde{P}_{-1})$ a {\it weighted basket} if $B$ is
a basket and $\tilde{P}_{-1}$ is a non-negative integer. We write
$(B, \tilde{P}_{-1})\succeq (B',\tilde{P}_{-1})$ if $B\succeq B'$.

\item Given a weighted basket ${\mathbb B}=(B, \tilde{P}_{-1})$, define $\tilde{P}_{-1}({\mathbb B})=\tilde{P}_{-1}$ and the volume
$$-K^3({\mathbb B})=2\tilde{P}_{-1}+\sigma(B)-\sigma'(B)-6.$$
For all $m\geq 1$, we define the ``anti-plurigenus'' in the following inductive way:
\begin{align*}
{}&\tilde{P}_{-(m+1)}-\tilde{P}_{-m}\\
={}& \frac{1}{2}(m+1)^2(-K^3({\mathbb B})+\sigma'(B))+2-\frac{m+1}{2}\sigma-\Delta^{m+1}(B).
\end{align*}
Note that, if we set ${\mathbb B}=(B_X, P_{-1}(X))$ for a given terminal weak $\bQ$-Fano $3$-fold $X$, then one can verify directly that $-K^3({\mathbb 
B})=-K_X^3$ and $\tilde{P}_{-m}({\mathbb B})=P_{-m}(X)$ for all $m\geq
1$.
\end{enumerate}

\begin{prpty}[{\cite[Section 3]{explicit}}]
Assume that ${\mathbb B}=(B, \tilde{P}_{-1})\succeq \mathbb{B}'=(B',
\tilde{P}_{-1})$. Then
\begin{itemize}
\item[(i)] $\sigma(B)=\sigma(B')$ and
$\sigma'(B)\geq \sigma'(B')$;

\item[(ii)] for all integer $n\geq 1$, $\Delta^n(B)\geq
\Delta^n(B')$;

\item[(iii)] $-K^3({\mathbb B})+\sigma'(B)=-K^3({\mathbb B}')+\sigma'(B')$;

\item[(iv)] $-K^3({\mathbb B})\leq -K^3({\mathbb B}')$;

\item[(v)] $\tilde{P}_{-m}({\mathbb B})\leq \tilde{P}_{-m}({\mathbb B}')$ for all
$m\geq 2$.
\end{itemize}
\end{prpty}

Next we recall the ``canonical'' sequence of a basket $B$. Set
$S^{(0)}=\{\frac{1}{n}\mid n\geq 2\}$,
$S^{(5)}=S^{(0)}\cup\{\frac{2}{5}\}$, and inductively for all $n \ge
5$,
$$S^{(n)}=S^{(n-1)}\cup\bigg\{\frac{b}{n}\mid 0<b<\frac{n}{2},\ b\
\text{is coprime to}\ n\bigg\}.$$
Each set $S^{(n)}$ gives a division of
the interval $(0,\frac{1}{2}]=\underset{i}\bigcup
[\omega_{i+1}^{(n)}, \omega^{(n)}_i]$ with
$\omega_{i}^{(n)},\omega_{i+1}^{(n)} \in S^{(n)}$. Let
$\omega_{i+1}^{(n)}=\frac{q_{i+1}}{p_{i+1}}$ and
$\omega^{(n)}_i=\frac{q_i}{p_i}$ with $\text{g.c.d}(q_l,p_l)=1$ for
$l=i,i+1$. Then it is easy to see that $q_ip_{i+1}-p_iq_{i+1}=1$
for all $n$ and $i$ (cf. \cite[Claim A]{explicit}).

Now given a basket ${B}=\{(b_i, r_i)\mid i=1,\cdots, s\}$, we define new baskets $\BB^{(n)}(B)$, where $\BB^{(n)}(\cdot)$ can be regarded as an operator on the set of baskets. 
For each $(b_i,r_i)
\in B$, if $\frac{b_i}{r_i} \in S^{(n)}$, then we set
$\BB^{(n)}_i=\{(b_i,r_i)\}$. If $\frac{b_i}{r_i}\not\in S^{(n)}$,
then $\omega^{(n)}_{l+1} < \frac{b_i}{r_i} < \omega^{(n)}_{l}$ for
some $l$. We write $\omega^{(n)}_{l}=\frac{q_l}{p_l}$ and
$\omega^{(n)}_{l+1}=\frac{q_{l+1}}{p_{l+1}}$ respectively.
 In this situation, we can unpack $(b_i,r_i)$ to
$\BB^{(n)}_i=\{(r_i q_l-b_ip_l) \times (q_{l+1},p_{l+1}),(-r_i
q_{l+1}+b_i p_{l+1}) \times (q_l,p_l)\}$. Adding up those
$\BB^{(n)}_i$, we get a new basket $\BB^{(n)}(B)$, which is
uniquely defined according to the construction and $\BB^{(n)}(B)
\succeq B$ for all $n$. Note that, by the definition, $B=\BB^{(n)}(B)$ for sufficiently large $n$.

Moreover, we have
$$\BB^{(n-1)}(B)=\BB^{(n-1)}(\BB^{(n)}(B))
\succeq \BB^{(n)}(B)$$
 for all $n\geq 1$ (cf. \cite[Claim B]{explicit}). Therefore we have a chain of baskets
 $$\BB^{(0)}(B)
\succeq \BB^{(5)}(B) \succeq \cdots \succeq \BB^{(n)}(B) \succeq \cdots \succeq B. $$
The step $\BB^{(n-1)}(B) \succeq \BB^{(n)}(B)$ can be achieved by a
number of successive prime packings. Let $\epsilon_n(B)$ be the
number of such prime packings. For any $n>0$, set $B^{(n)}=\BB^{(n)}(B)$. 

The following properties are essential in representing $B^{(n)}$.

\begin{lem}[{\cite[Lemma 2.16]{explicit}}]\label{delta}For the above sequence $\{B^{(n)}\}$, the following statements hold:
\begin{itemize}
\item[(i)]
$\Delta^j(B^{(0)})= \Delta^j(B)$ for $j=3,4$;
\item[(ii)]
$\Delta^j(B^{(n-1)})= \Delta^j(B^{(n)})$ for all $j <n$;
\item[(iii)]
$\Delta^n(B^{(n-1)})= \Delta^n(B^{(n)})+\epsilon_n(B)$.
\end{itemize}
\end{lem}

It follows that $\Delta^j(B^{(n)})=\Delta^j(B) $ for all $j \leq n$ and
$$\epsilon_n(B)=\Delta^n(B^{(n-1)})-\Delta^n(B^{(n)})=\Delta^n(B^{(n-1)})-\Delta^n(B).$$

Moreover, given a weighted basket ${\mathbb B}=(B, \tilde{P}_{-1})$, we
can similarly consider $\BB^{(n)}({\mathbb B})=(B^{(n)},
\tilde{P}_{-1})$. It follows that
$$\tilde{P}_{-j}(\BB^{(n)}({\mathbb B}))=\tilde{P}_{-j}({\mathbb
B})$$
for all $j \leq n$. Therefore we can realize the
canonical sequence of weighted baskets as an approximation of weighted
baskets via anti-plurigenera.

We now recall the relation between weighted baskets and
anti-plurigenera more closely. For a given weighted basket ${\mathbb B}=(B,\tilde{P}_{-1})$, we start by computing the non-negative number $\epsilon_n$ and $B^{(0)}$, $B^{(5)}$ in terms of $\tilde{P}_{-m}$. {}From the definition of $\tilde{P}_{-m}$ we get
\begin{align}
\sigma(B)={}&10-5\tilde{P}_{-1}+\tilde{P}_{-2},\label{105}\\
\Delta^{m+1}={}&(2-5(m+1)+2(m+1)^2)+\frac{1}{2}(m+1)(2-
3m)\tilde{P}_{-1}\notag\\
{}&+\frac{1}{2}m(m+1)\tilde{P}_{-2}+
\tilde{P}_{-m}-\tilde{P}_{-(m+1)}.\notag
\end{align}
In particular, we have
\begin{align*}
\Delta^3={}&5-6\tp_{-1}+4\tp_{-2}-\tp_{-3};\\
\Delta^4={}&14-14\tp_{-1}+6\tp_{-2}+\tp_{-3}-\tp_{-4}.
\end{align*}
Assume that $B^{(0)}=\{n_{1,r}^0\times (1,r)\mid r\geq 2\}$. By Lemma
\ref{delta}, we have
\begin{align*}
\sigma(B)={}&\sigma(B^{(0)})=\sum n_{1,r}^0;\\
\Delta^3(B)={}&\Delta^3(B^{(0)})=n_{1,2}^0;\\
\Delta^4(B)={}&\Delta^4(B^{(0)})=2n_{1,2}^0+n_{1,3}^0.
\end{align*}
Thus we get $B^{(0)}$ as follows: $$\begin{cases}
n_{1,2}^0=5-6\tp_{-1}+4\tp_{-2}-\tp_{-3};\\
n_{1,3}^0=4-2\tp_{-1}-2\tp_{-2}+3\tp_{-3}-\tp_{-4};\\
n_{1,4}^0=1+3\tp_{-1}-\tp_{-2}-2\tp_{-3}+\tp_{-4}-\sigma_5;\\
n_{1,r}^0=n_{1,r}^0, r\geq 5,
\end{cases}$$
where $\sigma_5=\sum_{r\geq 5} n_{1,r}^0$. A computation gives
$$\epsilon_5=2+\tp_{-2}-2\tp_{-4}+\tp_{-5}-\sigma_5.$$
Therefore we get $\mB^{(5)}=\{n_{1,r}^5\times (1,r), n_{2,5}^5\times (2,5)\mid r\geq 2\}$ as follows:
$$\begin{cases}
n_{1,2}^5=3-6\tp_{-1}+3\tp_{-2}-\tp_{-3}+2\tp_{-4}-\tp_{-5}+\sigma_5;\\
n_{2,5}^5=2+\tp_{-2}-2\tp_{-4}+\tp_{-5}-\sigma_5;\\
n_{1,3}^5=2-2\tp_{-1}-3\tp_{-2}+3\tp_{-3}+\tp_{-4}-\tp_{-5}+\sigma_5;\\
n_{1,4}^5=1+3\tp_{-1}-\tp_{-2}-2\tp_{-3}+\tp_{-4}-\sigma_5;\\
n_{1,r}^5=n_{1,r}^0, r\geq 5.
\end{cases}$$
Because $\mB^{(5)}=\mB^{(6)}$, we see $\epsilon_6=0$ and on the
other hand
$$\epsilon_6=3\tp_{-1}+\tp_{-2}-\tp_{-3}-\tp_{-4}-\tp_{-5}+\tp_{-6}-\epsilon=0$$
where $\epsilon=2\sigma_5-n_{1,5}^0 \ge 0$.

Going on a similar calculation, we get
\begin{align*}
\epsilon_7={}&1+\tilde{P}_{-1}+\tilde{P}_{-2}-\tilde{P}_{-5}-\tilde{P}_{-6}
+\tilde{P}_{-7}-2\sigma_5+2n^0_{1,5}+n^0_{1,6};\\
 \epsilon_{8} ={}&
2\tilde{P}_{-1}+\tilde{P}_{-2}+\tilde{P}_{-3}-\tilde{P}_{-4}-\tilde{P}_{-5}
-\tilde{P}_{-7}+\tilde{P}_{-8}\\
{}&-3\sigma_5+3 n^0_{1,5}+2 n^0_{1,6}+n^0_{1,7}.
\end{align*}


A weighted basket ${\mathbb B}=(B,\tilde{P}_{-1})$ is said to be
{\it geometric} if ${\mathbb B}=(B_X, P_{-1}(X))$ for some terminal weak
$\bQ$-Fano $3$-fold $X$. Geometric baskets are subject to some geometric properties. 
By \cite{KMMT}, we have that
$(-K_X\cdot c_2(X))\geq 0$. Therefore \cite[10.3]{YPG} gives the inequality
\begin{align}\label{kwmt}
\gamma(B):=\sum_{i} \frac{1}{r_i}-\sum_{i} r_i+24\geq 0.
\end{align}
For packings, it is easy to see the following lemma.
\begin{lem}\label{31} Given a packing of baskets $B_1\succeq
 B_2$, we have $\gamma(B_1) \geq \gamma(B_2)$. In particular, if
inequality \eqref{kwmt} does not hold for $B_1$, then it does not hold for $B_2$.
\end{lem}

Furthermore, $-K^3({\mathbb B})=-K_X^3>0$ gives the
inequality
\begin{align}\label{vol}
\sigma'(B)<2P_{-1}+\sigma(B)-6. 
\end{align}
{}Finally, by \cite[Lemma 15.6.2]{Kollar}, if $P_{-m}>0$ and
$P_{-n} >0$, then
\begin{align}\label{2m}
 P_{-m-n} \ge P_{-m}+P_{-n}-1. 
\end{align}

\subsection{Effective results on terminal weak $\bQ$-Fano $3$-folds}\

We collect here some known facts about terminal weak $\bQ$-Fano $3$-folds. 

\begin{prop}\label{facts}Let $X$ be a terminal weak $\bQ$-Fano $3$-folds. Then
\begin{enumerate}
\item $P_{-m}>0$ for $m\geq 6$ (\cite[Theorem 1.1]{CC} and \cite[Corollary 5.1]{C});
\item $P_{-8}\geq 2$ (\cite[Theorem 1.1]{CC});
\item $-K_X^3\geq 1/330$ (\cite[Theorem 1.1]{CC});
\item $r_X\leq 660$ or $r_X=840$, moreover, if $r_X=840$, then $r_{\max}=8$ (\cite[Proposition 2.4]{CJ16});
\item $r_{\max}\leq 24$ by inequality \eqref{kwmt}.
\end{enumerate}
\end{prop}

\section{\bf Fano--Mori triples}\label{sec FM}
In this section, we introduce the concept of Fano--Mori triples and prove some basic properties. 

\begin{definition}We say that a normal projective variety $Y$ admits a {\it Mori fiber structure} if the following conditions hold:
\begin{enumerate}
\item $g:Y\to T$ is a surjective morphism onto a normal projective variety $T$;
\item $Y$ is $\mathbb{Q}$-factorial with at worst terminal singularities; 
\item $g_*\OO_Y=\OO_T$; 
\item $-K_Y$ is $g$-ample;
\item $\rho(Y/T)=1$;
\item $\dim Y > \dim T$. 
\end{enumerate} 
\end{definition}

\begin{example}
Note that, in dimension $3$, there are three kinds of Mori fiber structures according to the value of $\dim T$:
\begin{enumerate}
\item $\dim T=0$, $Y$ is a $\mathbb{Q}$-factorial terminal $\mathbb{Q}$-Fano $3$-folds with $\rho(Y)=1$;
\item $\dim T=1$, $Y\rightarrow T$ is a {\it del Pezzo fibration}, of which a general fiber is a smooth del Pezzo surface;
\item $\dim T=2$, $Y\rightarrow T$ is a {\it conic bundle}, of which a general fiber is a smooth rational curve.
\end{enumerate}
\end{example}

\begin{definition}[{cf. \cite[Definition 3.6.1]{BCHM}}]
Let $\phi: X\dashrightarrow Y$ be a proper birational contraction (i.e., a birational map extracting no divisors) between normal
quasi-projective varieties and $D$ a $\bQ$-Cartier divisor on $X$ such that 
$D' =\phi_*D$ is also $\bQ$-Cartier. We say that $\phi$ is {\it a $D$-non-positive contraction} if, for some common resolution $p: W \to X$ and $q : W\to Y$, one may write
$$p^*D = q^*D' + E$$
where $E$ is an effective $q$-exceptional $\bQ$-divisor.
\end{definition}

\begin{lem}\label{auto non-positive}Let $\phi: X\dashrightarrow Y$ be a  birational contraction between normal
projective varieties and $D$ a $\bQ$-Cartier divisor on $X$ such that 
$D' =\phi_*D$ is also $\bQ$-Cartier. If $-D$ is nef, then $\phi$ is $D$-non-positive.
 \end{lem}
\begin{proof}
 Taking a common resolution $p: W \to V'$ and $q : W\to Y$, one may write
$$p^*D = q^*D' + E$$
where $E$ is a $q$-exceptional $\bQ$-divisor. Note that $-E=q^*D'-p^*D$ is $q$-nef since $-D$ is nef, hence $E\geq 0$ by Negativity Lemma (see \cite[Corollary 3.39(1)]{KM}).
\end{proof}

Note that the definition of $D$-non-positive contraction is independent of the choice of common resolutions.  The most standard example of a $D$-non-positive contraction might be a composition of steps of $D$-MMP.

\begin{definition}\label{def KMFS}
Let $V$ be a normal projective variety with at worst canonical singularities and $V'$ be a terminalization of $V$. Then we say  $Y$ is  a {\it $K$-Mori fiber space} of $V$ if the following conditions are satisfied: 
\begin{enumerate}
\item $Y$ has a Mori fiber structure $Y\to T$;
\item there is a birational contraction $\sigma: V'\dashrightarrow Y$;
\item $\sigma$ is a $K_{V'}$-non-positive birational contraction.
\end{enumerate}
\end{definition}
Note that if $-K_V$ is nef (or equivalently, $-K_{V'}$ is nef), then condition (3) automatically holds by Lemma \ref{auto non-positive}.

\begin{remark}\label{remark exist MFS}Note that the definition of $K$-Mori fiber spaces  is independent of the choice of the terminalization. By \cite{BCHM}, if $K_V$ is not pseudo-effective, then $K$-Mori fiber spaces of $V$ always exist by running $K$-MMP on $V'$.
\end{remark}

The following is the key object that we are interested in.

\begin{definition}\label{MF3}
We say that $(X, Y, Z)$ is a {\it Fano--Mori triple} if there is a commutative diagram
$$\xymatrix{
X\ar[dr]_p \ar@{-->}[rr]^{\phi} & & Y \ar@{-->}[dl]^q \\
& Z&}
$$
satisfying the following conditions:
\begin{enumerate}
\item $X$ is a $\bQ$-factorial terminal weak $\bQ$-Fano $3$-fold;
\item $Y$ has a Mori fiber structure $Y\to T$;
\item $Z$ is a $\bQ$-factorial canonical weak $\bQ$-Fano $3$-fold;
\item $\phi: X\dashrightarrow Y$ is a  birational contraction;
\item $q: Y\dashrightarrow Z$ is a  birational contraction which is isomorphic in codimension one;
\item $p: X\to Z$ is a terminalization of $Z$, or equivalently, $K_X=p^*(K_Z)$.
\end{enumerate}
\end{definition}

\begin{remark}\label{rhoY1}
For a Fano--Mori triple $(X, Y, Z)$, if $\rho(Y)=1$ (or equivalently, $\dim T=0$), then $X, Y, Z$ are isomorphic to each other, and they are $\mathbb{Q}$-factorial terminal $\mathbb{Q}$-Fano $3$-folds with $\rho=1$. In fact, in this case, since $\rho(Y)=1$ and $Y$ admits a Mori fiber structure, $Y$ is a $\mathbb{Q}$-factorial terminal $\mathbb{Q}$-Fano $3$-fold of Picard number one. Since $Y$ and $Z$ are isomorphic in codimension one, it turns out that 
$\rho(Z)=1$ and $Y\cong Z$. In particular, $Z$ is terminal and hence $X\cong Z$. 
\end{remark}

By Remark \ref{rhoY1}, the concept of Fano--Mori triples can be viewed as a natural generalization of the concept of $\mathbb{Q}$-factorial terminal $\mathbb{Q}$-Fano $3$-folds with $\rho=1$. 
Moreover, the following proposition suggests that Fano--Mori triples appear naturally in the study of birational geometry of canonical weak $\bQ$-Fano $3$-folds.

\begin{prop}\label{prop Fano=FanoMori}
Let $V$ be a canonical weak $\bQ$-Fano $3$-fold and $Y$ a $K$-Mori fiber space of $V$. Then there exists a Fano--Mori triple $(X, Y, Z)$ containing $Y$ as the second term.
\end{prop}
\begin{proof}
Let $V$ be a canonical weak $\bQ$-Fano $3$-fold and $Y$ a $K$-Mori fiber space of $V$. After replacing $V$ by its terminalization, we may assume that $V$ is $\bQ$-factorial and terminal and there exists a birational 
contraction $\sigma: V\dashrightarrow Y$. 

Note that $-K_V$ is semi-ample by Basepoint-free Theorem (see \cite[Theorem 3.3]{KM}), we may choose an effective $\bQ$-divisor $B\sim_\bQ-K_V$ such that $(V, B)$ is terminal and $\Supp(B)$ does not contain any exceptional divisor on $V$ over $Y$. 
Take $B_Y=\sigma_*B$. Then $K_Y+B_Y\sim_\bQ 0$. 
Taking a common resolution $f:W\to V$ and $g:W\to Y$.  Then $f^*(K_V+B)-g^*(K_Y+B_Y)$ is $g$-exceptional. By Negativity Lemma (see \cite[Corollary 3.39(1)]{KM}), we see that $f^*(K_V+B)= g^*(K_Y+B_Y)$. Hence $(Y,B_Y)$ is canonical by the construction of $(V, B)$. 

Noting that $-K_Y=\sigma_*(-K_V)$ is big, we may write $-K_Y\sim_\bQ A_Y+E_Y$ where $A_Y$ is an ample $\bQ$-divisor and $E_Y$ is an effective $\bQ$-divisor. Now take $t>0$ to be sufficiently small such that $(Y, B_Y+tE_Y)$ is klt. Then the pair $(Y, (1-t)B_Y+tE_Y)$ is also klt and $$-(K_Y+(1-t)B_Y+tE_Y)\sim_\bQ tA_Y$$ is ample. By \cite[Corollary 1.3.2]{BCHM}, $Y$ is a Mori dream space, which means that we can run MMP for any divisor on $Y$.

Running the $(-K)$-MMP on $Y$, we end up with a minimal model $Z$ which is $\bQ$-factorial such that $-K_Z$ is nef and big. 
Since $-K_V$ is semi-ample, the stable base locus of $-K_Y$ does not contain any divisors, hence the $(-K)$-MMP on $Y$ does not contract any divisor, which means that the 
birational contraction $q: Y\dashrightarrow Z$ is an isomorphism in codimension one. Moreover, take $B_Z=q_*B_Y$, then $(Z, B_Z)$ is canonical also by Negativity Lemma since $(Y, B_Y)$ is canonical and $K_Y+B_Y\sim_\bQ 0$, which implies that $Z$ is also canonical. In summery, $Z$ is a $\bQ$-factorial canonical weak $\bQ$-Fano $3$-fold.

Finally, we take $p:X\to Z$ to be a terminalization of $Z$, then $X$ is a $\bQ$-factorial terminal weak $\bQ$-Fano $3$-fold. 
Clearly the induced map $\phi= q^{-1}\circ p: X\dashrightarrow Y$ is a birational contraction.
\end{proof}


\begin{prop}\label{prop P=P=P}
Let $(X, Y, Z)$ be a Fano--Mori triple. Then 
\begin{enumerate}
\item $\phi: X\dashrightarrow Y$ is  $K_X$-non-positive, in particular, $Y$ is a $K$-Mori fiber space of $X$;
\item  $q: Y\dashrightarrow Z$ is $(-K_Y)$-non-positive;
\item for any integer $m\geq 0$, 
$$
h^0(X, -mK_X)=h^0(Y, -mK_Y)=h^0(Z,-mK_Z).
$$
In particular, $\varphi_{-m, X}$,  $\varphi_{-m, Y}$, and $\varphi_{-m, Z}$ factor through each other by the birational maps $\phi$, $p$, and $q$, and therefore share the same birational properties.
\end{enumerate}
\end{prop}
\begin{proof}
Since $-K_X$ is nef, by Lemma \ref{auto non-positive}, $\phi$ is $K_X$-non-positive, which proves statement (1).

Since $q$ is an isomorphism in codimension one, $q^{-1}:Z\dashrightarrow Y$ is also a birational contraction. Since $-K_Z$ is nef, again by Lemma \ref{auto non-positive}, $q^{-1}$ is $K_Z$-non-positive. Since $q$ is an isomorphism in codimension one, it is easy to check by definition that $q: Y\dashrightarrow Z$ is $(-K_Y)$-non-positive, which proves statement (2).

For statement (3), fix an integer $m\geq 0$. Since $\phi$ is $K_X$-non-positive and $q$ is $(-K_Y)$-non-positive, we have
$$
h^0(X, -mK_X)\leq h^0(Y, -mK_Y)=h^0(Z,-mK_Z).
$$
On the other hand, since $K_X=p^*(K_Z)$, we have
$$
h^0(X, -mK_X)=h^0(Z,-mK_Z).
$$
We complete the proof.
\end{proof}


\begin{lem}\label{lem TFT}
Let $V$ be a canonical weak $\bQ$-Fano $3$-fold and $Y$ a $K$-Mori fiber space of $V$. Let $Y\to T$ be the Mori fiber structure, then $T$ is of Fano type,
 that is, there exists an effective $\bQ$-divisor $B_T$ such that $(T, B_T)$ is klt and $-(K_T+B_T)$ is ample. In particular, if $(X, Y, Z)$ is a Fano--Mori triple and $Y\to T$ be the Mori fiber structure, then $T$ is of Fano type.
\end{lem}
\begin{proof}  In the proof of Proposition \ref{prop Fano=FanoMori}, we have seen that
 the pair $(Y, (1-t)B_Y+tE_Y)$ is klt and $$-(K_Y+(1-t)B_Y+tE_Y)\sim_\bQ tA_Y$$ is ample. Hence $Y$ is of Fano type. Then we may apply \cite[Lemma 2.8]{PS09} or \cite[Corollary 3.3]{FG} to conclude that $T$ is of Fano type as well.
 
The last statement follows from Proposition \ref{prop P=P=P}(1).
\end{proof}

\section{\bf A geometric inequality}\label{sec GI}
In our previous paper, we used the following proposition to get a number $m_1>0$ so that
$|-m_1K|$ is not composed with a pencil. 

\begin{prop}[{\cite[Proof of Corollary 4.2]{CJ16}}]\label{prop old}
Let $X$ be a canonical weak $\bQ$-Fano $3$-fold and $m>0$ an integer. Assume that $|-mK_X|$ is composed with a pencil. Keep the notation in Subsection \ref{b setting}. Then
$$
\frac{P_{-m}-1}{m}\leq \frac{-K_X^3}{(\pi^*(K_X)^2\cdot S)}.
$$
\end{prop}

Proposition \ref{prop old} is, however, too weak for applications. The main goal of this section is to present an alternative inequality (see Proposition \ref{prop new}) which turns out to be the key ingredient of this paper.

To start up we recall the setting. Let $(X, Y, Z)$ be a Fano--Mori triple and assume that $|-mK_X|$ is composed with a pencil for an integer $m>0$. By the notation in Subsection \ref{b setting}, we may take $W$ to be a common log resolution of $X$ and $Y$, and $f: W\to \bP^1$ is the induced fibration by $|-mK_X|$. A general fiber of $f$ is denoted by $S$. We have the following commutative diagram:
$$\xymatrix{
 & W\ar[dl]_\pi \ar[dr]^\eta \ar[r]^f& \bP^1 \\
X\ar[dr]_p \ar@{-->}[rr]^{\phi} & & Y \ar@{-->}[dl]^q \ar[r]^g & T\\
& Z& &}
$$


\begin{lem}\label{lemma exist exc} Let $(X, Y, Z)$ be a Fano--Mori triple and $m>0$ an integer. Assume that $|-mK_X|$ is composed with a pencil. Keep the notation in Subsection \ref{b setting}. Assume that there exists an effective $\eta$-exceptional $\bQ$-divisor $F$ on $W$ such that, $(-\pi^*(K_X)|_S\cdot F|_S)>0$ for a general $S$, 
then $$
\frac{P_{-m}-1}{m}\leq 2r_X(\pi^*(K_X)^2\cdot S).
$$
\end{lem}
\begin{proof}
Since $-\pi^*(K_X)|_S$ is nef, there exists a component $F_0$ of $F$ such that $(-\pi^*(K_X)|_S\cdot F_0|_S)>0$, which means that 
$$
(-\pi_Z^*(K_Z)|_S\cdot F_0|_S)=(-\pi^*(K_X)|_S\cdot F_0|_S)>0
$$
where we denote by $\pi_Z$ the composition $p\circ\pi:W\to Z$.
In particular, $F_0|_S$ is not contracted by $\pi_Z.$ Hence there exists a curve $C_Z\subset Z$ such that $$C_Z\subset \pi_Z(F_0\cap S)\subset \pi_Z(F_0).$$ On the other hand, since $Y$ and $Z$ are isomorphic in codimension one and $F_0$ is $\eta$-exceptional, $F_0$ is also $\pi_Z$-exceptional. This implies that $\pi_Z(F_0)$ is a subvariety of codimension at least $2$. Hence $C_Z= \pi_Z(F_0)$. In particular, $C_Z$ is independent of $S$, and for any general $S$, $C_Z= \pi_Z(F_0\cap S)$. Moreover, 
$$
(-K_Z\cdot C_Z)=\frac{1}{n}(\pi_Z^*(-K_Z)\cdot F_0\cdot S)>0$$
by the projection formula, where $\pi_*(F_0|_S)=nC_Z$ as $1$-cycles. Hence $(-K_Z\cdot C_Z)\geq \frac{1}{r_X}$ since $r_XK_Z$ is Cartier.


By assumption, we have
$$
-\pi^*(mK_X)\sim (P_{-m}-1)S+D_m
$$
where $D_m$ is an effective $\bQ$-divisor. Set $w=\frac{P_{-m}-1}{m}$. Then \begin{align}\label{equ 2}
-\frac{1}{w}\pi^*(K_X)-S\sim_\bQ D
\end{align}
for some effective $\bQ$-divisor $D$.
Write
$$
K_W=\pi^*(K_X)+E_\pi,
$$
where $E_\pi\geq 0$ is $\pi$-exceptional. Pick two general fibers $S_1$ and $S_2$ of $f$ and consider the pair
$$
(W, -E_\pi+2D+S_1+S_2),
$$
which can be assumed to have simple normal crossing support modulo a further  birational modification of $W$.
Note that 
\begin{align*}
{}& -(K_W-E_\pi+2D+S_1+S_2)\\
\sim_\bQ{}& -\Big(1-\frac{2}{w}\Big)\pi^*(K_X)\\
\sim_\bQ{}& -\Big(1-\frac{2}{w}\Big)\pi_Z^*(K_Z)
\end{align*}
is $\pi_Z$-nef and ${\pi_Z}_*(-E_\pi+2D+S_1+S_2)\geq 0$ since 
$E_\pi$ is $\pi_Z$-exceptional.
Denote by $G$ the support of the effective part of $\rounddown{-E_\pi+2D}+S_1+S_2$.
By Connectedness Lemma (see \cite[Theorem 5.48]{KM}), 
$$
G\cap \pi_Z^{-1}(z)
$$
is connected for any point $z\in Z$.

We claim that there exists a prime divisor $F_1$ on $W$ such that $2D\geq F_1$ and $\pi_Z(F_1\cap S)$ contains $C_Z$ for a general $S$.
Consider a point $z\in C_Z\subset Z$, then $S\cap \pi_Z^{-1}(z) \neq \emptyset$ for a general $S$ since $C_Z=\pi_Z(F_0\cap S)\subset \pi_Z(S)$. In particular, $S_1\cap \pi_Z^{-1}(z)\neq \emptyset$ and $S_2\cap \pi_Z^{-1}(z)\neq \emptyset$, which are two disconnected sets in $G\cap \pi_Z^{-1}(z)$. Since $G\cap \pi_Z^{-1}(z)$ is connected, there exists a curve $B_z\subset G\cap \pi_Z^{-1}(z)$ such that $B_z\cap S_1\cap \pi_Z^{-1}(z)\neq \emptyset$ and $B_z\not\subset S_1\cap \pi_Z^{-1}(z)$. Moving $z$ in $C_Z$, $B_z$ deforms to a prime divisor $F_1\subset G$, to be more precise, there exists a prime divisor $F_1\subset G$ such that $B_z\subset F_1\cap \pi_Z^{-1}(z)$ for infinitely many $z\in C_Z$. Hence 
$$z\in \pi_Z( B_z\cap S_1\cap \pi_Z^{-1}(z))\subset \pi_Z(F_1\cap S_1\cap \pi_Z^{-1}(z))\subset \pi_Z(F_1\cap S_1)$$
for infinitely many $z\in C_Z$. This means that $C_Z\subset \pi_Z(F_1\cap S_1).$
By the construction of $F_1$, $F_1\subset G$ and it is clear that $F_1$ is different from $S_1$ and $S_2$, hence $\coeff_{F_1}(-E_\pi+2D)\geq 1$ and in particular, $2D\geq F_1$. 
By the generality of $S_1$, $C_Z\subset \pi_Z(F_1\cap S)$ for general $S$.

By equation \eqref{equ 2},
\begin{align*}
\frac{2}{w}(\pi^*(K_X)^2\cdot S){}&=\big(-\pi^*(K_X)|_S\cdot (2S+2D)|_S\big)\\
{}&= \big( \pi^*(-K_X)|_S\cdot 2D|_S\big)\geq  \big( \pi^*(-K_X)|_S\cdot F_1|_S\big)\\
{}&= \big( \pi_{Z}^*(-K_Z)|_S\cdot F_1|_S\big)=\big( -K_Z\cdot {\pi_{Z}}_*( F_1|_S)\big)\\
{}&\geq \big( -K_Z\cdot C_Z\big)\geq \frac{1}{r_X}.
\end{align*}
In other words, we have 
$$
\frac{P_{-m}-1}{m}=w\leq 2r_X(\pi^*(K_X)^2\cdot S).
$$
\end{proof}

\begin{lem}\label{lemma Y P1}
Let $(X, Y, Z)$ be a Fano--Mori triple and $m>0$ an integer. Assume that $|-mK_X|$ is composed with a pencil. Keep the notation in Subsection \ref{b setting}.
Assume that $f:W\to \bP^1$ factors through $Y$. Then 
$$
\frac{P_{-m}-1}{m}\leq \max\{-K_X^3, 2r_X(\pi^*(K_X)^2\cdot S)\}.
$$
\end{lem}
\begin{proof}
Let $S_Y$ be a general fiber of $Y \to \bP^1$.
Write $$\pi^*(K_X)=\eta^*(K_Y)+E$$ where $E$ is an effective $\eta$-exceptional $\bQ$-divisor by Proposition \ref{prop P=P=P}(1). Restricting to a general fiber $S$ of $f$, we have
\begin{align}\label{equ 1}
\pi^*(K_X)|_S=\eta^*(K_{Y})|_S+E|_S=\eta_S^*(K_{S_Y})+E|_S,
\end{align}
where $\eta_S:S\to S_Y$ is the restriction of $\eta$ on $S$.

If $E|_S=0$, then 
$$
(\pi^*(K_X)^2\cdot S)=(\pi^*(K_X)|_S)^2=\eta_S^*(K_{S_Y})^2
$$
is an integer since $S_Y$ is smooth. Since $-\pi^*(K_X)|_S$ is nef and big,  $(\pi^*(K_X)^2\cdot S)\geq 1$ and, by Proposition \ref{prop old},
$$
\frac{P_{-m}-1}{m}\leq -K_X^3.
$$

Now we may assume that $E|_S\neq 0$. 
Since $E$ is exceptional over $Y$ and $S$ is general, $E|_S$ is an effective $\bQ$-divisor on $S$ exceptional over $S_Y$. By Hodge Index Theorem, $(E|_S)^2<0.$
By \eqref{equ 1},
\begin{align*}
\big(\pi^*(K_X)|_S\cdot E|_S\big)=\big(\eta_S^*(K_{S_Y})\cdot E|_S\big)+(E|_S)^2=(E|_S)^2<0.
\end{align*}
Hence we may apply Lemma \ref{lemma exist exc} to get
$$
\frac{P_{-m}-1}{m}\leq 2r_X(\pi^*(K_X)^2\cdot S).
$$
So we have completed the proof.
\end{proof}

\begin{lem}\label{dP3}
Let $S$ be a smooth del Pezzo surface and $D$  a  non-zero integral effective divisor on $S$ such that
$-K_S-aD$ is $\bQ$-linearly equivalent to an effective $\bQ$-divisor
for some positive rational number $a$. Then $a\leq 3$.
\end{lem}
\begin{proof}
By assumption, there is an effective $\bQ$-divisor $B$ such that 
$$
aD+B\sim_\bQ -K_S.
$$
Since $(S, D+\frac{1}{a}B)$ is not klt, the log canonical threshold $\text{lct}(S; aD+B)\leq \frac{1}{a}$.  Recall that the global log-canonical threshold is defined as:
$$\text{lct}(S)=\text{inf}\{\text{lct}(S;L)\mid L\sim_{\bQ} -K_S,\ \text{$L$ is an effective $\bQ$-divisor}\}.$$
Clearly we have $\text{lct}(S) \leq \frac{1}{a}$. On the other hand, $\text{lct}(S) \geq \frac{1}{3}$ by \cite[Theorem 1.7]{Cheltsov}, which implies that $a\leq 3$. 
\end{proof}

\begin{prop}\label{prop new}
Let $(X, Y, Z)$ be a Fano--Mori triple such that $\rho(Y)>1$ and $m>0$ an integer. Assume that $|-mK_X|$ is composed with a pencil. Keep the same notation as in Subsection \ref{b setting}. Then
$$
\frac{P_{-m}-1}{m}\leq \max\{3, -K_X^3, 2r_X(\pi^*(K_X)^2\cdot S)\}.
$$
\end{prop}

\begin{proof}
Note that we have 
$$
-\pi^*(mK_X)\sim (P_{-m}-1)S+D_m,
$$
where $D_m$ is an effective $\bQ$-divisor (the fixed part). Pushing forward to $Y$, one has
\begin{align}\label{equ3}
-mK_Y=-\eta_*\pi^*(mK_X)\sim (P_{-m}-1)\eta_*S+\eta_*D_m.
\end{align}

Consider the Mori fiber structure $g:Y\to T$. Since $\rho(Y)>1$, $\dim T>0$. We argue by discussing the value of $\dim T$.

If $\dim T=1$, then $T\simeq \bP^1$ since $g(T)\leq q(Y)=q(X)=0$. Take $S_Y$ to be a general fiber of $g$.  If $\eta_*S|_{S_Y} =0$, then $|\eta_*S|$ coincides with the pencil $|S_Y|$, and hence $f$ factors through $g$. By Lemma \ref{lemma Y P1} we have
$$
\frac{P_{-m}-1}{m}\leq \max\{-K_X^3, 2r_X(\pi^*(K_X)^2\cdot S)\}.
$$
 If $\eta_*S|_{S_Y}\neq 0$, then $\eta_*S|_{S_Y}=D$ is an effective non-zero integral divisor on $S_Y$ and by equation \eqref{equ3},
$$
-mK_{S_Y}=-mK_Y|_{S_Y} \sim (P_{-m}-1)\eta_*S|_{S_Y}+\eta_*D_m|_{S_Y}\geq (P_{-m}-1)D,
$$
which means that $\frac{P_{-m}-1}{m}\leq 3$ by Lemma \ref{dP3} since $S_Y$ is a smooth del Pezzo surface.

Now we consider the case when $\dim T=2$. Note that for a general fiber $C\simeq \bP^1$ of $g$, $(\eta_*S\cdot C)$ is a non-negative integer.
If $(\eta_*S\cdot C)\geq 1$ for a general fiber $C$ of $g$, then by equation \eqref{equ3} intersecting with $C$,
$$
2m=-m(K_Y\cdot C)\geq (P_{-m}-1)(\eta_*S\cdot C)\geq P_{-m}-1,
$$
which means that
$$
\frac{P_{-m}-1}{m}\leq 2.
$$

Hence we may assume that $(\eta_*S\cdot C)=0$ for a general fiber $C$ of $g$. By Cone Theorem (see \cite[Theorem 3.7]{KM}), there exists a $\bQ$-divisor $H$ on $T$ such that $\eta_*S\sim_\bQ g^*H$. In particular, $(\eta_*S)^2\equiv(H^2)C$ as $1$-cycles. Note that, if we fix a general $S$, then we may just take $H=g(\eta_*S)$. Since $|\eta_*S|$ is a movable linear system, $H$ is nef and $H^2\geq 0$. Moreover, $H$ is semi-ample by Basepoint-free Theorem and Lemma \ref{lem TFT}.

First we consider the case $H^2=0$. In this case $|nH|$ defines a contraction $T\to \Lambda$ to a curve $\Lambda$ (after taking Stein factorization) for a sufficiently large $n$. Moreover, it is easy to see that the induced morphism $W\to Y\to T\to \Lambda$ contracts $S$, which means that this map coincides with $f:W\to \bP^1$. Hence $f: W\to \bP^1$ factors through $Y$ and we may apply Lemma \ref{lemma Y P1} to get
$$
\frac{P_{-m}-1}{m}\leq \max\{-K_X^3, 2r_X\pi^*(K_X)^2\cdot S\}.
$$

Next we consider the case $H^2>0$. We have the following claim.

\begin{claim}\label{claim} For any $\eta$-exceptional prime divisor $E_0$ on $W$, $$(E_0\cdot \eta^*\eta_*S\cdot S)=0.$$\end{claim}
\begin{proof}
By projection formula, we have
$$
(E_0\cdot \eta^*\eta_*S\cdot S)=(E_0|_S\cdot \eta^*\eta_*S)=(\eta_*(E_0|_S)\cdot \eta_*S).
$$
We may assume that $\eta_*(E_0|_S)\neq 0$ for a general $S$ since, otherwise, there is nothing to prove. Then there is a curve $G\subset \eta(E_0\cap S)\subset \eta(E_0)$. Note that $E_0$ is an $\eta$-exceptional prime divisor, we have $G=\eta(E_0\cap S)= \eta(E_0)$. In particular, $G$ does not depends on $S$ and $G=\eta(E_0\cap S)\subset \eta_*S$ for general $S$. Recall that $\eta_*S\sim_\bQ g^*H$, which means that the intersection of two general $\eta_*S$ lies in fibers of $g$, hence $G$ lies in a fiber of $g$. In particular, since $\text{Supp}(\eta_*(E_0|_S))=G$, 
$$
(\eta_*(E_0|_S)\cdot \eta_*S)=(\eta_*(E_0|_S)\cdot g^*H)=0.
$$
We have proved the claim.
\end{proof}
Write 
\begin{align*}
\pi^*(K_X)=\eta^*(K_Y)+E
\end{align*} 
where $E\geq 0$ is an $\eta$-exceptional $\bQ$-divisor. 
By Claim \ref{claim}, 
\begin{align*}
(-\pi^*(K_X)\cdot \eta^*\eta_*S\cdot S)=&(-\eta^*(K_Y)\cdot \eta^*\eta_*S\cdot S)\\
=&(-K_Y\cdot \eta_*S\cdot \eta_*S)\\
={}&(-K_Y\cdot (H^2)C)=2(H^2)>0.
\end{align*}
On the other hand, we may write 
$$
\eta^*\eta_*S=S+F
$$
where $F\geq 0$ is an $\eta$-exceptional effective $\bQ$-divisor on $W$, and we have
$$
0<(-\pi^*(K_X)\cdot \eta^*\eta_*S\cdot S)=(-\pi^*(K_X)\cdot (S+F)\cdot S)=(-\pi^*(K_X)|_S\cdot F|_S).
$$
Hence we may apply Lemma \ref{lemma exist exc} to get
$$
\frac{P_{-m}-1}{m}\leq 2r_X(\pi^*(K_X)^2\cdot S).
$$

Combining the above cases, we complete the proof. 
\end{proof}


\section{\bf Criteria for birationality}\label{sec criterion}

Based on Section \ref{sec GI}, firstly we give criteria for $|-mK|$ being not composed with pencils (Propositions \ref{criterion 1} and \ref{criterion 2}) and then recall criteria for birationality of $|-mK|$ established in our previous paper \cite{CJ16} (Theorems \ref{criterion b} and \ref{criterion b2}). 



We always assume that $X$ is a terminal weak $\bQ$-Fano $3$-fold. 

Set $M_X=r_X(-K_X^3)$, which is a positive integer. For any positive integer $N$, define the following functions
\begin{align*}
\theta(M_X, N)={}&\min\{M_X/N, \max\{3,M_X/r_X, 2N\}\},\\
\theta(M_X)={}&\max_{N\in \ZZ_{>0}}\theta(M_X, N),
\end{align*}
and
$$
\lambda(M_X)=\begin{cases} M_X, & \text{if\ } M_X\leq 3;\\
\max\left\{3, M_X/r_X, 2\rounddown{\sqrt{M_X/2}}, \frac{M_X}{\roundup{\sqrt{M_X/2}}}\right\}& \text{if\ } M_X\geq 4.
\end{cases}
$$

We have the following relation between $\lambda(M_X)$ and $\theta(M_X)$.
\begin{lem}
$\lambda(M_X) \geq \theta(M_X)$.
\end{lem}
\begin{proof}
It suffices to show that $\lambda(M_X) \geq \theta(M_X, N)$ holds for any positive integer $N$.

If $M_X\leq 3$, then
$$
\lambda(M_X)=M_X\geq M_X/N\geq \theta(M_X, N).
$$

If $M_X\geq 4$ and $N\geq \roundup{\sqrt{M_X/2}}$, then
$$
\lambda(M_X)\geq \frac{M_X}{\roundup{\sqrt{M_X/2}}}\geq M_X/N\geq \theta(M_X, N).
$$

If $M_X\geq 4$ and $N\leq \rounddown{\sqrt{M_X/2}}$, then
\begin{align*}
\lambda(M_X){}&=\max\left\{3, M_X/r_X, 2\rounddown{\sqrt{M_X/2}}, \frac{M_X}{\roundup{\sqrt{M_X/2}}}\right\}\\
{}&\geq\max\left\{3, M_X/r_X, 2\rounddown{\sqrt{M_X/2}}\right\}\\
{}&\geq\max\{3, M_X/r_X, 2N\}\\
{}&\geq  \theta(M_X, N).
\end{align*}
\end{proof}

\begin{prop}\label{criterion 1}
Let $(X, Y, Z)$ be a Fano--Mori triple with $\rho(Y)>1$  and $m>0$ an integer. If
$$
P_{-m}> \lambda(M_X)m+1,
$$
then $|-mK_X|$ is not composed with a pencil.
\end{prop}
\begin{proof} {}First the assumption implies that $P_{-m}\geq 2$. Hence $\varphi_{-m,X}$ is non-trivial. 
Assume that $|-mK_X|$ is composed with a pencil. Keep the notation in Subsection \ref{b setting}. Take $N_0=r_X(\pi^*(K_X)^2\cdot S)$, which is a positive integer (cf. \cite[Lemma 2.2]{C}). Then, by Propositions \ref{prop old} and \ref{prop new},
$$
\frac{P_{-m}-1}{m}\leq \frac{M_X}{N_0}
$$
and 
$$
\frac{P_{-m}-1}{m}\leq \max\{3,M_X/r_X, 2N_0\}.
$$
That is, 
\begin{align*}
\frac{P_{-m}-1}{m}{}&\leq\min\{ M_X/N_0, \max\{3,M_X/r_X, 2N_0\}\}\\
{}&=\theta(M_X, N_0)\leq \lambda(M_X),
\end{align*}
a contradiction.
\end{proof}

In practice, we need to know the lower bound of $P_{-m}$ which is fairly computable by virtue of Reid's Riemann--Roch formula. 
Let us recall the following proposition from \cite{CJ16}, which we will use to estimate the anti-plurigenus.

\begin{prop}[cf.  {\cite[Proof of Proposition 4.3]{CJ16}}]\label{rr inequality}
Let $X$ be a terminal weak $\bQ$-Fano $3$-fold and $t>0$ a real number. Then, for any integer $n\geq t$ and $n\geq r_{\max} t/3$, one has
$$
P_{-n}\geq \frac{1}{12}n(n+1)(2n+1)(-K_X^3)+1-\frac{2n}{t}.
$$
\end{prop}
\begin{proof} From \cite[Proof of Proposition 4.3 (page 90)]{CJ16}, one has 
 $$
l(-n)\leq \frac{2(t+1)n}{t}
$$
provided that $n\geq t$ and $n\geq r_{\max} t/3$.
It follows from Reid's Riemann--Roch formula that
\begin{align*}
P_{-n}={}&\frac{1}{12}n(n+1)(2n+1)(-K_X^3)+2n+1-l(-n)\\
\geq {}& \frac{1}{12}n(n+1)(2n+1)(-K_X^3)+1-\frac{2n}{t}.
\end{align*}
\end{proof}


\begin{prop}\label{criterion 2}
Let $(X, Y, Z)$ be a Fano--Mori triple such that $\rho(Y)>1$. Let $t>0$ be a real number and $m$ an integer. If $m\geq t$, $m\geq \frac{r_{\max} t}{3}$, and
$$
m>-\frac{3}{4}+\sqrt{\frac{12}{t\cdot(-K_X^3)}+\frac{6\lambda(M_X)}{-K_X^3}+\frac{1}{16}},
$$
then $|-mK_X|$ is not composed with a pencil.
\end{prop}
\begin{proof} Under the assumptions, one has 
$$
\frac{1}{12}(m+1)(2m+1)> \frac{2}{t\cdot(-K_X^3)}+\frac{\lambda(M_X)}{-K_X^3},
$$
and, by Proposition \ref{rr inequality},
$$
P_{-m}>\lambda(M_X)m+1.
$$
Hence $|-mK_X|$ is not composed with a pencil by Proposition \ref{criterion 1}.
\end{proof}

To apply Proposition \ref{criterion 2}, we always use the following lemma to estimate the value of $\lambda(M_X)/{(-K_X^3)}$.

\begin{lem}\label{lem l/d}
The following inequalities hold:
\begin{enumerate}
\item $$
\frac{\lambda(M_X)}{-K_X^3}\leq \max\left\{1,\frac{3}{-K_X^3}, \sqrt{\frac{2r_X}{-K_X^3}}\right\};
$$
\item $\lambda(M_X)\leq M_X$ and ${\lambda(M_X)}/{(-K_X^3)}\leq r_X$.
\end{enumerate}
\end{lem}
\begin{proof} 
If $M_X\leq 3$, then it is clear. We may assume that $M_X\geq 4$.

Note that $2\rounddown{\sqrt{M_X/2}}\leq \sqrt{2M_X}$ and $\frac{M_X}{\roundup{\sqrt{M_X/2}}}\leq \sqrt{2M_X}$.
Hence \begin{align*}
\frac{\lambda(M_X)}{-K_X^3}={}&\frac{1}{-K_X^3}\max\left\{3, M_X/r_X, 2\rounddown{\sqrt{M_X/2}}, \frac{M_X}{\roundup{\sqrt{M_X/2}}}\right\}\\
\leq {}&\max\left\{\frac{3}{-K_X^3}, 1, \frac{\sqrt{2M_X}}{-K_X^3}\right\}\\
={}& \max\left\{\frac{3}{-K_X^3}, 1, \sqrt{\frac{2r_X}{-K_X^3}}\right\},
\end{align*}
which implies (1). 

Similarly, since $2\rounddown{\sqrt{M_X/2}}\leq M_X$ and $\frac{M_X}{\roundup{\sqrt{M_X/2}}}\leq M_X$, it is clear that $\lambda(M_X)\leq M_X$ which implies (2). 
\end{proof}


As the last part of this section, we recall the established birationality criterion in \cite{CJ16} along with a newly developed criterion.

\begin{assum}\label{assum 1} Let $X$ be a terminal weak $\bQ$-Fano $3$-fold and $m_0>0$ an integer with  $P_{-m_0}\geq 2$. Let $m_1\geq m_0$ be another integer with $P_{-m_1}\geq 2$ such that $|-m_1K_X|$ and $|-m_0K_X|$ are not composed with the same pencil. 
Pick a generic irreducible element $S$ of $|M_{-m_0}|=\text{Mov}|\rounddown{\pi^*(-m_0K_X)}|$. Define the real number 
$$\mu_0=\text{inf}\{t\in \bQ^+ \mid t\pi^*(-K_X)-S\sim_{\bQ} \text{effective}\ \bQ\text{-divisor}\},$$
and in practice we may choose a suitable rational number $\mu'_0\geq \mu_0$ such that 
$$\mu'_0\pi^*(-K_X)-S\sim_{\bQ} \text{effective}\ \bQ\text{-divisor}.$$
Let $\nu_0$ be an integer such that $h^0(-\nu_0K_X)>0$. 
Set $$a(m_0)= \begin{cases}
6, & {\rm if } \ m_0\geq 2;\\
1, &{\rm if } \ m_0=1,
\end{cases}$$ 
and the positive integer $N_0=r_X(\pi^*(K_X)^2\cdot S)$.
\end{assum}

\begin{thm}[{\cite[Theorem 5.11]{CJ16}}]\label{criterion b}  
Let $X$ be a terminal weak $\bQ$-Fano $3$-fold and keep Assumption \ref{assum 1}. 
Then $\varphi_{-m,X}$ is birational if one of the following holds:
\begin{enumerate}
\item $m\geq \max\{m_0+m_1+a(m_0), \rounddown{3\mu_0}+3m_1\}$;

\item $m\geq \max\{m_0+m_1+a(m_0), \rounddown{\frac{5}{3}\mu_0+ \frac{5}{3}m_1}, \rounddown{\mu_0}+m_1+2r_{\max}\}$;

\item $m\geq \max\{m_0+m_1+a(m_0), \rounddown{\mu_0}+m_1+2\nu_0r_{\max}\}.$
\end{enumerate} 
\end{thm}

\begin{remark}[cf. {\cite[Remark 5.3]{CJ16}}]\label{5.3}
Most of the time in practice, we will take $m_0$ to be the minimal positive integer with $P_{-m_0}\geq 2$. 
By definition one can always choose $\mu'_0=m_0$ and have $\mu_0\leq  m_0$. If $|-m_0K_X|$ is composed with a pencil, then we can choose  $\mu'_0= \frac{m_0}{P_{-m_0}-1}$ and $\mu_0\leq \frac{m_0}{P_{-m_0}-1}$. 
Moreover, if for some integer $k$, $|-kK_X|$ and $|-m_0K_X|$ are composed with the same pencil, then we can choose  $\mu'_0= \frac{k}{P_{-k}-1}$ and $\mu_0\leq \frac{k}{P_{-k}-1}$. \end{remark}

When both $r_X$ and $r_{\max}$ are small, we feel that the following new criterion is very helpful.

\begin{thm}\label{criterion b2}  
Let $X$ be a terminal weak $\bQ$-Fano $3$-fold and keep Assumption \ref{assum 1}. 
Then $\varphi_{-m,X}$ is birational if 
 $$m\geq \max\left\{m_0+a(m_0), \roundup{\mu'_0}+4\nu_0r_{\max}-1, \rounddown{\mu'_0+\sqrt{8r_X/N_0}}\right\}.$$
\end{thm}

\begin{lem}[{see \cite[Proposition 4]{Mas99} or \cite[Lemma 2.6]{Chen14}}]\label{lem mas}
 Let $S$ be a smooth projective surface. Let $L$ be a nef and big $\bQ$-divisor on $S$ satisfying the following conditions:
 \begin{enumerate}
\item $L^2 > 8$;
\item $(L \cdot C_P) \geq 4$ for all irreducible curves $C_P$ passing through any
very general point $P \in S$.
\end{enumerate}
Then the linear system $|K_S + \roundup{L}|$ separates two distinct points in very
general positions. Consequently, $|K_S + \roundup{L}|$ gives a birational map.
\end{lem}
\begin{proof}[Proof of Theorem 5.9]
By \cite[Lemma 5.2]{CJ16}, the birationality of $\varphi_{-m,X}$ is equivalent to the birationality of $\Phi_{\Lambda_m}$ defined by the linear system $\Lambda_m=|K_W+\roundup{(m+1)\pi^*(-K_X)}|$ on $W$.

By \cite[Proposition 5.8]{CJ16}, $\Lambda_m$ can distinguish different generic irreducible elements of $|M_{-m_0}|$ since $m\geq m_0+a(m_0)$. Hence the usual birationality principle (see, for instance,  \cite[2.7]{CC2}) reduces the birationality of ${\Phi_{\Lambda_m}}$ to that of
${\Phi_{\Lambda_m}}|_S$ for a generic irreducible element $S$ of $|M_{-m_0}|$.

Now we will show that ${\Phi_{\Lambda_m}}|_S$ is birational under the assumption of the theorem.


By the definition of $\mu'_0$, we have $$\mu'_0\pi^*(-K_X)\sim_{\bQ} S+E$$ for an effective $\bQ$-divisor $E$. 
Re-modify the resolution $\pi$ in Subsection \ref{b setting} so that $E$ has simple normal crossing support.

For the given integer $m>0$, we have
\begin{align}\label{b subsys1}
|K_Y+\roundup{(m+1)\pi^*(-K_X)-E}|\preceq
|K_Y+\roundup{(m+1)\pi^*(-K_X)}|.
\end{align}
Since by assumption $m\geq \mu'_0$, the $\bQ$-divisor
$$(m+1)\pi^*(-K_X)-E-S\equiv (m+1-\mu'_0)\pi^*(-K_X)$$
is nef and big and thus $$H^1(Y,
K_Y+\roundup{(m+1)\pi^*(-K_X)-E}-S)=0$$ by
Kawamata--Viehweg vanishing theorem. Hence we have
the surjective map
\begin{align}\label{b surj1}
H^0(Y,K_Y+\roundup{(m+1)\pi^*(-K_X)-E})\longrightarrow
H^0(S, K_S+L_{m}) 
\end{align} 
where
\begin{align}\label{b subsys2}
L_{m}:=(\roundup{(m+1)\pi^*(-K_X)-E}-S)|_S\geq
\roundup{\mathcal{L}_{m}}
\end{align}
and ${\mathcal
L}_{m}:=((m+1)\pi^*(-K_X)-E-S)|_S$ which is a nef and big $\bQ$-divisor on $S$.

By relations \eqref{b subsys1}--\eqref{b subsys2}, to show that ${\Phi_{\Lambda_m}}|_S$ is birational, it suffices to show that $|K_S+\roundup{\mathcal{L}_{m}}
|$ gives a birational map. Note that 
$$
\mathcal{L}_{m}^2=(m+1-\mu'_0)^2(\pi^*(K_X)|_S)^2=\frac{(m+1-\mu'_0)^2N_0}{r_X}>8
$$
since $m\geq \rounddown{\mu'_0+\sqrt{8r_X/N_0}}$.
Now consider an irreducible curve $C_P$ passing through a
very general point $P \in S$. Recall that $-\nu_0 K_X\sim D$ for some effective Weil divisor $D$. 
Take $q:Y\to X$ to be the resolution of isolated singularities and  we may assume that $W$ dominates $Y$ by $p:W\to Y$.
Then we write
$$
q^*D=D_Y+\sum\frac{a_i}{r_i}E_i.
$$
Here $D_Y$ is the strict transform of $D$ on $Y$ and
$E_i$ is the exceptional divisor over an isolated singular point of index $r_i$ for some $r_i\in B_X$ and $a_i$ is a positive integer.
Then 
$$
\pi^*D=p^*D_Y+\sum\frac{a_i}{r_i}p^*E_i.
$$
Since $P$ is a very general point,  $(p^*D_Y\cdot C_P)$ and $(p^*E_i\cdot C_P)$ are non-negative integers, and at least one of them is positive since $(\pi^*D\cdot C_P)>0$ by the fact that $D$ is nef and big. Hence $(\pi^*D\cdot C_P)\geq \frac{1}{r_{\max}}$. Consequently, 
$$
(\mathcal{L}_{m}\cdot C_P)=(m+1-\mu'_0)(\pi^*(-K_X)\cdot C_P)\geq \frac{m+1-\mu'_0}{\nu_0r_{\max}}\geq 4.
$$
Hence, by Lemma \ref{lem mas}, $|K_S+\roundup{\mathcal{L}_{m}}
|$ gives a birational map and so does $\varphi_{-m,X}$.
\end{proof}

Finally we explain the strategy to assert the birationality. Given a certain set of terminal weak $\bQ$-Fano $3$-folds, or a set of geometric weighted baskets, we can find a number $m_0>0$ such that $P_{-m_0}\geq 2$. Then we can find another number $m_1>0$, by using either the explicit computation of Reid's Riemann--Roch formula or Proposition \ref{criterion 2}, so that $|-m_1K|$ is not composed with a pencil. Meanwhile we have the estimate of $\mu_0$ according to Remark \ref{5.3}. With all the information, 
we may apply either Theorem \ref{criterion b} or \ref{criterion b2} to obtain the anti-pluricanonical birationality except the fact that to classify weighted baskets with pre-assigned invariants might be a tedious work. 

\section{\bf Proof of the main theorems}\label{sec proof}
In this section, we prove the main theorems.
\subsection{Proof of Theorems \ref{main thm} and \ref{main thm2}}\ \ 

Let $V$ be a canonical weak $\bQ$-Fano $3$-fold and $Y$ a $K$-Mori fiber space of $V$ (by Remark \ref{remark exist MFS}, such $Y$ always exists). By Theorem \ref{thm1}, $V$ is birational to a Fano--Mori triple $(X, Y, Z)$. 

By Theorem \ref{thm2}, the terminal weak $\bQ$-Fano $3$-fold $X$ in the Fano--Mori triple $(X, Y, Z)$ satisfies the properties we require in the Theorem \ref{main thm}. 

By Proposition \ref{prop P=P=P}(3), $\varphi_{-m, X}$ and $\varphi_{-m, Y}$ share the same birational properties. Hence 
Theorem \ref{main thm2} also follows from Theorem \ref{thm2}. 
\qed

\subsection{Proof of Theorem \ref{thm2}}\ \ 

Let  $(X, Y, Z)$ be a Fano--Mori triple. 

If $\rho(Y)=1$, then by Remark \ref{rhoY1}, $X\cong Y\cong Z$ is a $\bQ$-factorial terminal $\bQ$-Fano $3$-fold of Picard number one. Then statement (1) follows from Theorem \ref{main1} and Proposition \ref{facts}(1). Statement (2) follows from Theorem
\ref{birationality1}.

From now on we may assume that $\rho(Y)>1$. Consider the terminal weak $\bQ$-Fano $3$-fold $X$. If  $P_{-2}=0$ or $r_X=840$, then the theorem follows from Theorem \ref{thm P2=0} or \ref{thm r840}. Hence we may assume that $P_{-2}>0$ and $r_X\leq 660$ by  Proposition \ref{facts}.

For statement (1), if $-K_X^3\geq 1/30$, it follows from Lemma \ref{lem 30 1}; if  $-K_X^3< 1/30$, it follows from Theorem \ref{thm <1/30}(1).


For statement (2), if $-K_X^3\geq 0.21$, it follows from Theorem \ref{thm >=0.21}; if $P_{-1}=0$ and $-K_X^3< 0.21$, it follows from Theorem \ref{thm P1=0 <0.21}; if $P_{-1}>0$ and $1/30\leq -K_X^3< 0.21$, it follows from Theorem \ref{thm P1>0 1/30<};
 if $P_{-1}>0$ and $-K_X^3< 1/30$, it follows from Theorem \ref{thm <1/30}.

In summary, the theorem is proved. 

Finally, we remark that from the proof of statement (2), $\varphi_{-51}$ is not birational only if when $-K_X^3=1/330$ and $X$ has Reid basket $\{(1,2),(2,5),(1,3),(2,11)\},$ which appears in the proof of Theorem \ref{thm <1/30}.
\qed





\subsection{The case $P_{-8}=2$ and two lemmas}
\begin{thm}\label{thm P8=2}
Let $(X, Y, Z)$ be a Fano--Mori triple such that $\rho(Y)>1$. Assume that $P_{-8}=2$.
Then
$\varphi_{-m}$ is birational for all $m\geq 51$.
\end{thm}
\begin{proof} Note that $P_{-2}\leq 1$ since $P_{-8}\geq 4P_{-2}-3$. 

If $P_{-2}=0$, then By \cite[Theorem 3.5]{CC},
any geometric basket of a terminal weak $\bQ$-Fano $3$-fold with $P_{-1}=P_{-2}=0$ and $P_{-8}=2$ is
among the following list:

{\scriptsize
\begin{longtable}{LLCCCCCCC}
\caption{}\label{tab18}\\
\hline
B_X & -K^3 & M_X & \lambda(M_X) & n_1 & m_0 & r_{\max} & n_2 \\
\hline
\endfirsthead
\multicolumn{3}{l}{{ {\bf \tablename\ \thetable{}} \textrm{-- continued from previous page}}}
 \\
\hline 
B_X & -K^3 & M_X & \lambda(M_X) & n_1 & m_0 & r_{\max} & n_2 \\ \hline 
\endhead

\hline \multicolumn{3}{r}{{\textrm{Continued on next page}}} \\ \hline
\endfoot

\hline \hline
\endlastfoot
 \{2 \times (1,2), 3 \times (2,5),(1,3),(1,4)\} & 1/60 & 1 & 1 & 20 & 8 & 5 &46\\
 \{5 \times (1,2), 2 \times (1,3),(2,7),(1,4)\} & 1/84 & 1 & 1 & 22 & 8 & 7& 50\\
 \{5 \times (1,2), 2 \times (1,3),(3,11)\} & 1/66 & 1 & 1 & 20 & 8 & 11 &50\\
 \{5 \times (1,2), (1,3),(3,10),(1,4)\} & 1/60 & 1& 1 & 20 & 8 & 10 &48
\end{longtable} 
}

In Table \ref{tab18},
for each basket $B_X$, we can compute $M_X$ and $\lambda(M_X)$, then find $n_1$ such that $P_{-n_1}>\lambda(M_X)n_1+1$ where $P_{-n_1}$ is computed by Reid's Riemann--Roch formula. Hence, by Proposition \ref{criterion 1},
 $\dim \overline{\varphi_{-n_1}(X)}>1$. Again by Reid's Riemann--Roch formula, we may find $m_0$ such that $P_{-m_0}\geq 2$. Then, by Proposition \ref{criterion b}(2), take $m_1=n_1$ and $\mu_0\leq m_0$, we get the integer $n_2$ such that $\varphi_{-m}$ is birational for all $m\geq n_2$.

We consider the case $P_{-2}=1$ from now on. We follow the argument of \cite[Proof of Theorem 3.10]{CC} to classify all possible geometric baskets with $P_{-2}=1$ and $P_{-8}=2$. Note that in this case $P_{-1}\leq 1$ and
$P_{-4}=1$.

If $P_{-1}=0$, then by \cite[Proof of Theorem 3.10, Case 1]{CC}, we have $P_{-3}=0$ and $\sigma_5=1$. Since we know the values of $P_{-n}$ for $1\leq n\leq 4$ and $\sigma_5$, $B^{(0)}=\{9\times (1,2), (1,3), (1,s)\}$ for some $s\geq 5$. If $s\geq 7$, then $P_{-8}\geq \tp_{-8}(B^{(0)})\geq 4$, a contradiction.
If $s<7$, then $-K^3(B^{(0)})\leq 0$. It is easy to see that $B_X$ is dominated by either $B'=\{8\times (1,2), (2,5), (1,6)\}$ or $\{7\times (1,2), (3,7), (1,5)\}$. Note that for both cases, we have $P_{-8}\geq \tp_{-8}(B')\geq 3$, a contradiction.

If $P_{-1}=1$, then $P_{-3}=P_{-4}=1$ and 
$$1\leq P_{-5}\leq P_{-6}\leq P_{-7}\leq 2.$$ 
We have $(n^0_{1,2}, n^0_{1,3}, n^0_{1,4})=(2,2,2-\sigma_5)$. Hence $\sigma_5\leq 2$. On the other hand, $\epsilon_6=0$ gives $\epsilon=P_{-6}-P_{-5}+2\geq 2$, which implies that $\sigma_5>0$.

If $(\sigma_5, P_{-5})=(2,2)$, then $\epsilon_5=1$, $P_{-6}=2$, and $\epsilon=2$, which implies that $n^0_{1,5}=2$. Hence $B^{(5)}=\{ (1,2), (2,5), (1,3),2\times (1,5)\}$. But all the packings have $-K^3<0$, which is absurd.

If $(\sigma_5, P_{-5})=(2,1)$, then $\epsilon_5=0$ and $$n^0_{1,5}=2\sigma_5-\epsilon=2-P_{-6}+P_{-5}>0.$$ Hence $B^{(5)}=\{ 2\times (1,2), 2\times (1,3), (1,5), (1,s)\}$ for some $s\geq 5$. If $s\leq 7$, then all the further packings  have $-K^3<0$, which is absurd. Hence $s\geq 8$ and $B_X=B^{(5)}$. Note that in this case $s=8,9,10$ by $\gamma\geq 0$. It is easy to check that $P_{-8}=3$, a contradiction.

If $(\sigma_5, P_{-5})=(1,2)$, then $\epsilon_5=2$. Hence $B^{(5)}=\{ 2\times (2,5), (1,4), (1,s)\}$ for some $s\geq 5$.  Then 
$\gamma\geq 0$ implies that $s\leq 10$. If $s\leq 6$, then all the further packings  have $-K^3<0$, which is absurd.  Hence $s\geq 7$ and $B_X=B^{(5)}$.  It is easy to check that $P_{-8}\geq 4$, a contradiction.

If $(\sigma_5, P_{-5})=(1,1)$, then $\epsilon=P_{-6}-P_{-5}+2$ implies that $P_{-6}=1-n^0_{1,5}$, which means that $P_{-6}=1$ and $n^0_{1,5}=0$. Hence $B^{(5)}=\{ (1,2), (2,5), (1,3), (1,4), (1,s)\}$ for some $s\geq 6$. Note that 
$s=6, 7$ is impossible since all further packings have $-K^3<0$. Hence we have $8\leq s\leq 11$ by $\gamma\geq 0$. Also we have $\epsilon_7=P_{-7}-1$ and $\epsilon_8=1-P_{-7}$, which means that $\epsilon_7=\epsilon_8=0$. Hence $B_X=B^{(8)}=B^{(5)}=\{ (1,2), (2,5), (1,3), (1,4), (1,s)\}$ for $8\leq s \leq 11$. Note that $s=8$ is absurd since in this case $-K^3<0$. 

In summary, all possible geometric baskets with $P_{-2}=1$ and $P_{-8}=2$ are $$B_X=\{ (1,2), (2,5), (1,3), (1,4), (1,s)\}$$ for $9\leq s \leq 11$.

If $B_X=\{ (1,2), (2,5), (1,3), (1,4), (1,9)\}$, then $-K_X^3=1/180$, $M_X=\lambda(M_X)=1$. Then $P_{-30}=33>30\lambda(M_X)+1$ where $P_{-30}$ is computed by Reid's Riemann--Roch formula. Hence by Proposition \ref{criterion 1}, $\dim \overline{\varphi_{-m}(X)}>1$ for all $m\geq 30$ since $P_{-1}>0$. Recall that $P_{-8}=2$. Take $m_0=8$ and $\nu_0=1$. Note that 
$P_{-19}=11$ by Reid's Riemann--Roch formula. If  $|-19K_X|$ and $|-8K_X|$ are not composed with the same pencil, then we may take $\mu_0\leq 8$ and $m_1=19$, and by Proposition \ref{criterion b}(3), $\varphi_{-m}$ is birational for $m\geq 45$. If $|-19K_X|$ and $|-8K_X|$ are composed with the same pencil, then we may take $\mu_0\leq  \frac{19}{P_{-19}-1}=\frac{19}{10}$ by Remark \ref{5.3} and $m_1=30$, then by Proposition \ref{criterion b}(3), $\varphi_{-m}$ is birational for $m\geq 49$.

If $B_X=\{(1,2), (2,5), (1,3), (1,4), (1,10)\}$, then $-K_X^3=1/60$, $M_X=\lambda(M_X)=1$. Then $P_{-19}=22>19\lambda(M_X)+1$ where $P_{-19}$ is computed by Reid's Riemann--Roch formula. Hence by Proposition \ref{criterion 1}, $\dim \overline{\varphi_{-m}(X)}>1$ for all $m\geq 19$ since $P_{-1}>0$. Recall that $P_{-8}=2$. Then by Proposition \ref{criterion b}(3), take $m_1=19$, $\mu_0\leq m_0=8$, and $\nu_0=1$, we get the integer $n_2$ such that $\varphi_{-m}$ is birational for all $m\geq 47$. 

If $B_X=\{ (1,2), (2,5), (1,3), (1,4), (1,11)\}$, then $-K_X^3=17/660$.
Recall that $P_{-8}=2$. We may take $m_0=8$ and $\nu_0=1$. Note that $P_{-21}=43$. If $|-21K_X|$ and $|-8K_X|$ are not composed with the same pencil, then we may take $\mu_0\leq 8$ and $m_1=21$, and by Proposition \ref{criterion b}(3), $\varphi_{-m}$ is birational for $m\geq 51$. Hence we may assume that  $|-21K_X|$ and $|-8K_X|$ are composed with the same pencil, and we may take $\mu_0\leq \mu'_0= \frac{21}{P_{-21}-1}=\frac{1}{2}$ by Remark \ref{5.3}. If $|-26K_X|$ is not composed with a pencil, then we may take $m_1=26$ and by Proposition \ref{criterion b}(3), $\varphi_{-m}$ is birational for $m\geq 48$. If $|-26K_X|$ is composed with a pencil, by Proposition \ref{prop new} and keep the same notation as in Subsection \ref{b setting}, one has
$$
3<\frac{80}{26}=\frac{P_{-26}-1}{26}\leq \max\{3, -K_X^3, 2N_0\},
$$
where $N_0=r_X(\pi^*(K_X)^2\cdot S)$. This implies that $N_0\geq 2$, and by Proposition \ref{criterion b2}, $\varphi_{-m}$ is birational for $m\geq 51$.

Combining all above discussions, the proof is completed.
\end{proof}

\begin{lem}\label{lem 12 pencil}
Let $(X, Y, Z)$ be a Fano--Mori triple such that $\rho(Y)>1$. 
Assume that $r_X\leq 165$. Then $\dim \overline{\varphi_{-m}(X)}>1$ for all $m\geq 37$;
\end{lem}
\begin{proof}Note that in this case $-K_X^3\geq 1/165$. 
By Lemma \ref{lem l/d},
$$\frac{\lambda(M_X)}{-K_X^3}\leq r_X\leq 165.$$
Take $t=37/8$, then
$$-\frac{3}{4}+\sqrt{\frac{12}{t(-K_X^3)}+\frac{6\lambda(M_X)}{-K_X^3}+\frac{1}{16}}< 37.$$
By Proposition \ref{criterion 2}, $|-mK_X|$ is not composed with a pencil for $m\geq 37.$
\end{proof}

\begin{lem}\label{lem 12}
Let $(X, Y, Z)$ be a Fano--Mori triple such that $\rho(Y)>1$. 
Assume that one of the following conditions holds:
\begin{enumerate}
\item $r_X\leq 69$ and $r_{\max}\leq 12$; or
\item $P_{-1}>0$, $r_X\leq 287$, and $r_{\max}\leq 12$.
\end{enumerate}
Then $\varphi_{-m,X}$ is birational for all $m\geq 51$.
\end{lem}
\begin{proof}
Recall that $P_{-8}\geq 2$ by Proposition \ref{facts}. If $P_{-8}=2$, then we are done by Theorem \ref{thm P8=2}. Hence we may assume that $P_{-8}\geq 3$.

(1) Note that in this case $-K_X^3\geq 1/69$. 
By Lemma \ref{lem l/d},
$$\frac{\lambda(M_X)}{-K_X^3}\leq r_X\leq 69.$$
Take $t=23/4$, then
$$-\frac{3}{4}+\sqrt{\frac{12}{t(-K_X^3)}+\frac{6\lambda(M_X)}{-K_X^3}+\frac{1}{16}}< 23.$$
By Proposition \ref{criterion 2}, $|-mK_X|$ is not composed with a pencil for $m\geq 23.$

If $|-8K_X|$ is not composed with a pencil, then we may take $\mu_0\leq 8$ and $m_1=8$, and by Proposition \ref{criterion b}(2), $\varphi_{-m}$ is birational for $m\geq 40$. If $|-8K_X|$ is composed with a pencil, then we may take $\mu_0\leq  \frac{8}{P_{-8}-1}\leq 4$ by Remark \ref{5.3} and $m_1=23$, then by Proposition \ref{criterion b}(2), $\varphi_{-m}$ is birational for $m\geq 51$.

(2) In this case, we may take $\nu_0=1$ and $m_0=8$. If $|-8K_X|$ is not composed with a pencil, then we may take $\mu_0\leq 8$ and $m_1=8$, and by Proposition \ref{criterion b}(3), $\varphi_{-m}$ is birational for $m\geq 40$. If $|-8K_X|$ is composed with a pencil, then we may take $\mu'_0= \frac{8}{P_{-8}-1}\leq 4$ by Remark \ref{5.3} and $N_0\geq 1$, and by Proposition \ref{criterion b2}, $\varphi_{-m}$ is birational for $m\geq 51$. 
\end{proof}

\subsection{The case $P_{-2}=0$}

\begin{thm}\label{thm P2=0}
Let $(X, Y, Z)$ be a Fano--Mori triple such that $\rho(Y)>1$. Assume that $P_{-2}=0$.
Then 
\begin{enumerate}
\item $\dim \overline{\varphi_{-m}(X)}>1$ for all $m\geq 37$;
\item $\varphi_{-m,X}$ is birational for all $m\geq 51$.
\end{enumerate}
\end{thm}
\begin{proof} Since $P_{-2}=0$, $P_{-1}=0$. If $r_X\leq 69$ and $r_{\max}\leq 12$, then we are done by Lemmas \ref{lem 12 pencil} and \ref{lem 12}(1). By \cite[Theorem 3.5]{CC},
any geometric basket of a terminal weak $\bQ$-Fano $3$-fold with $P_{-1}=P_{-2}=0$ and with either $r_X>69$ or $r_{\max}>12$ is
among the following list:

{\scriptsize
\begin{longtable}{LLCCCCCCC}
\caption{}\label{tab1}\\
\hline
B_X & -K^3 & M_X & \lambda(M_X) & n_1 & m_0 & r_{\max} & n_2 \\
\hline
\endfirsthead
\multicolumn{3}{l}{{ {\bf \tablename\ \thetable{}} \textrm{-- continued from previous page}}}
 \\
\hline 
B_X & -K^3 & M_X & \lambda(M_X) & n_1 & m_0 & r_{\max} & n_2 \\ \hline 
\endhead

\hline \multicolumn{3}{r}{{\textrm{Continued on next page}}} \\ \hline
\endfoot

\hline \hline
\endlastfoot
 \{5 \times (1,2), 2 \times (1,3),(2,7),(1,4)\} & 1/84 & 1 & 1 & 22 & 8 & 7& 50\\
 \{3 \times (1,2),(5,14), 2 \times (1,3)\} & 1/21 & 2 & 2 & 16 & 6 & 14 &50\\
 \{ (1,2),(3,7), (2,5),4 \times (1,3)\} & 17/210 & 17 & 17/3 & 20 & 5 & 7 &41\\
 \{2 \times (1,2), (2,5),(3,8),3 \times (1,3)\} & 3/40 & 9 & 4 & 18 & 5 & 8&39 \\
 \{2 \times (1,2),(5,13),3 \times (1,3)\} & 1/13 & 6 & 3 & 15 & 5 & 13 &46
\end{longtable} 
}
In Table \ref{tab1},
for each basket $B_X$, we can compute $M_X$ and $\lambda(M_X)$, then find $n_1$ such that $P_{-n}>\lambda(M_X)n+1$ for $n_1\leq n\leq n_1+5$ where $P_{-n}$ is computed by Reid's Riemann--Roch formula. Hence by Proposition \ref{criterion 1},
 $\dim \overline{\varphi_{-n}(X)}>1$ for all $n_1\leq n\leq n_1+5$. Since $P_{-m}>0$ for all $m\geq 6$ by Proposition \ref{facts},
this implies that $\dim \overline{\varphi_{-m}(X)}>1$ for all $m\geq n_1$. Again by Reid's Riemann--Roch formula, we may find $m_0$ such that $P_{-m_0}\geq 2$. Then, by Proposition \ref{criterion b}(2), take $m_1=n_1$ and $\mu_0\leq m_0$, we get the integer $n_2$ such that $\varphi_{-m}$ is birational for all $m\geq n_2$.
\end{proof}

\subsection{The case $r_X=840$}
\begin{lem}\label{lem r840}
Let $X$ be a terminal weak $\bQ$-Fano $3$-fold with $r_X=840$. Then
\begin{enumerate}
\item $P_{-1}\geq 1$;
\item $-K_X^3\geq \frac{47}{840}$;
\item $P_{-2}\geq 2$ unless $B_X$ is in Table \ref{tab16}.

{\scriptsize
\begin{longtable}{CL}
\caption{}\label{tab16}\\
\hline
\text{\rm No.} &  B_X\ \\
\hline
\endfirsthead
1&\{(1,2), (1,3),(2, 5), (1,7), (1,8)\}\\

2&\{(1,2), (1,3),(1, 5), (2,7), (1,8)\}\\

3&\{(1,3),(1, 5), (1,7), (3,8)\}\\

4&\{(1,3),(1, 5), (3,7), (1,8)\}\\

5&\{(1,3),(2, 5), (2,7), (1,8)\}\\
\hline\hline
\end{longtable} 
}
\end{enumerate}
\end{lem}
\begin{proof}
By \cite[Proof of Proposition 2.4 (page 67)]{CJ16}, we know that the set of local indices $\{r_i\}$ of $X$ is either $\{2,3,5,7,8\}$ or $\{3,5,7,8\}$, that is, $B_X=\{(1,2), (1,3),(a, 5), (b,7), (c,8)\}$ or $\{(1,3),(a, 5), (b,7), (c,8)\}$ for some $a\in \{1,2\}$, $b\in \{1,2,3\}$, $c\in \{1,3\}$.
To unify the notation, we write $B_X=\{k\times (1,2), (1,3),(a, 5), (b,7), (c,8)\}$, where $k\in \{0,1\}$.

By inequality \eqref{vol},
\begin{align*}
2P_{-1}{}&>\sigma'(B_X)-\sigma(B_X)+6\\
{}&= -\frac{k}{2}-\frac{2}{3}-\frac{a(5-a)}{5}-\frac{b(7-b)}{7}-\frac{c(8-c)}{8}+6\\
{}&\geq -\frac{1}{2} -\frac{2}{3}-\frac{6}{5}-\frac{12}{7}-\frac{15}{8}+6>0.
\end{align*}
Hence $P_{-1}\geq 1$, which proves statement (1).

If $P_{-1}\geq 2$, then
\begin{align*}
-K_X^3{}&=2P_{-1}{}-\sigma'(B_X)+\sigma(B_X)-6\\
{}&\geq 4 +\frac{k}{2}+\frac{2}{3}+\frac{a(5-a)}{5}+\frac{b(7-b)}{7}+\frac{c(8-c)}{8}-6\\
{}&\geq \frac{2}{3}+\frac{4}{5}+\frac{6}{7}+\frac{7}{8}-2>1.
\end{align*}

If $P_{-1}=1$, then  by equality \eqref{105},
$$
k+1+a+b+c=\sigma(B_X)=10-5P_{-1}+P_{-2}\geq 6.
$$
When $k=1$, then at least one of $a,b,c$ is greater than $1$.  In particular, if $a\geq 2$, then
\begin{align*}
-K_X^3{}&=2P_{-1}{}-\sigma'(B_X)+\sigma(B_X)-6\\
{}&\geq 2 +\frac{k}{2}+\frac{2}{3}+\frac{a(5-a)}{5}+\frac{b(7-b)}{7}+\frac{c(8-c)}{8}-6\\
{}&\geq \frac{1}{2}+ \frac{2}{3}+\frac{6}{5}+\frac{6}{7}+\frac{7}{8}-4=\frac{83}{840}.
\end{align*}
Similarly we see $-K^3\geq \frac{227}{840}$ 
when $b\geq 2$ or $c= 3$. When $k=0$, we see $-K_X^3\geq \frac{47}{840}$ in the similar way if $b=3$ or $c=3$. The remaining cases are $(a,b,c)=(1,2,1)$, $(2,1,1)$, $(2,2,1)$ and we get either $-K_X^3=\frac{143}{840}$ or $-K_X^3<0$ (impossible cases). This asserts statement (2).

Finally, we consider the case $P_{-2}<2$ which forces $P_{-1}=P_{-2}=1$. By equality \eqref{105},
$$
k+1+a+b+c=\sigma(B_X)=10-5P_{-1}+P_{-2}=6.
$$
Hence $B_X$ can only be in Table \ref{tab16}.
\end{proof}

\begin{thm}\label{thm r840}
Let $(X, Y, Z)$ be a Fano--Mori triple such that $\rho(Y)>1$. Assume that $r_X=840$.
Then
\begin{enumerate}
\item $\dim \overline{\varphi_{-m}(X)}>1$ for all $m\geq 32$;
\item $\varphi_{-m,X}$ is birational for all $m\geq 50$.
\end{enumerate}
\end{thm}
\begin{proof}
Recall that in this case $r_{\max}=8$. By Lemma \ref{lem r840}, $P_{-1}\geq 1$ and $-K_X^3\geq  47/840$.

By Lemma \ref{lem l/d},
$$\frac{\lambda(M_X)}{-K_X^3}\leq \max\left\{1, \frac{3}{47/840}, \sqrt{\frac{2\times 840}{47/840}}\right\}<174.$$
Take $t=12$, then
$$-\frac{3}{4}+\sqrt{\frac{12}{t(-K_X^3)}+\frac{6\lambda(M_X)}{-K_X^3}+\frac{1}{16}}< 32.$$
By Proposition \ref{criterion 2}, $|-mK_X|$ is not composed with a pencil for $m\geq 32.$

{}For statement (2), first we consider the case $P_{-2}\geq 2$. In this case, we may take $\mu_0\leq m_0=2$, $m_1= 32$, and $\nu_0=1$.  Then by Proposition \ref{criterion b}(3), $\varphi_{-m}$ is birational for $m\geq 50$.

Then we consider the case $P_{-2}<2$, that is, $P_{-1}=P_{-2}=1$. In this case we are dealing with baskets listed in Table \ref{tab16}. In the following Table \ref{tab17},
for each basket $B_X$ in Table \ref{tab16}, 
we can compute $M_X$ and $\lambda(M_X)$, then find $n_1$ such that $P_{-n_1}>\lambda(M_X)n_1+1$ where $P_{-n_1}$ is computed by Reid's Riemann--Roch formula. Hence by Proposition \ref{criterion 1}, $\dim \overline{\varphi_{-m}(X)}>1$ for all $m\geq n_1$ since $P_{-1}>0$. Again by Reid's Riemann--Roch formula, we may find $m_0$ such that $P_{-m_0}\geq 2$. Then by Proposition \ref{criterion b}(3), take $m_1=n_1$, $\mu_0\leq m_0$, and $\nu_0=1$, we get the integer $n_2$ such that $\varphi_{-m}$ is birational for all $m\geq n_2$. 

{\scriptsize
\begin{longtable}{cLLCCCCCCC}
\caption{}\label{tab17}\\
\hline
No. &B_X & -K^3 & M_X & \lambda(M_X) & n_1 & m_0 & r_{\max} & n_2 \\
\hline



1&\{(1,2), (1,3),(2, 5), (1,7), (1,8)\} & 83/840 & 83 & 12 & 27 & 5 & 8 & 48\\
2&\{(1,2), (1,3),(1, 5), (2,7), (1,8)\} & 227/840 &227 & 227/11 &21 &3 &8 & 40 \\
3&\{(1,3),(1, 5), (1,7), (3,8)\} & 167/840 & 167 & 18 & 23 &3 & 8 & 42 \\
4&\{(1,3),(1, 5), (3,7), (1,8)\}& 47/840 & 47& 47/5 & 31 & 5& 8& 52 \\
5&\{(1,3),(2, 5), (2,7), (1,8)\}& 143/840 &143 &16 & 23 &3 &8 & 42\\
\hline\hline
\end{longtable} 
}

Note that, for case No. 4, we need further discussion. 
We know $P_{-5}=2$ and $P_{-11}=16$. Take $m_0=5$ and $\nu_0=1$. If $|-11K_X|$ and $|-5K_X|$ are not composed with the same pencil, then we may take $m_1=11$ and $\mu_0\leq 5$, and by Proposition \ref{criterion b}(3), $\varphi_{-m}$ is birational for all $m\geq 32$; if $|-11K_X|$ and $|-5K_X|$ are composed with the same pencil, then we may take $m_1=31$ and $\mu_0\leq \frac{11}{15}$ by Remark \ref{5.3}, and by Proposition \ref{criterion b}(3), $\varphi_{-m}$ is birational for all $m\geq 47$.

 Combining all above cases, the proof is completed.\end{proof}

\subsection{The case $-K_X^3\geq 0.21$}

\begin{thm}\label{thm >=0.21}
Let $(X, Y, Z)$ be a Fano--Mori triple such that $\rho(Y)>1$. Assume that $r_X\leq 660$. 
Assume that one of the following holds
\begin{enumerate}
\item $r_{\max}\leq 13$ and $-K_X^3\geq 0.12$; or
\item $r_{\max}\geq 14$ and $-K_X^3\geq 0.21$.
\end{enumerate}
Then
$\varphi_{-m,X}$ is birational for all $m\geq 51$.
\end{thm}
\begin{proof}
(1) If $r_{\max}\leq 13$ and $-K_X^3\geq 0.12$, recall that $r_X\leq 660$ by assumption. 
By Lemma \ref{lem l/d},
$$\frac{\lambda(M_X)}{-K_X^3}\leq \max\left\{1, \frac{3}{0.12}, \sqrt{\frac{2\times 660}{0.12}}\right\}<105.$$
Take $t=75/13$, then
$$-\frac{3}{4}+\sqrt{\frac{12}{t(-K_X^3)}+\frac{6\lambda(M_X)}{-K_X^3}+\frac{1}{16}}< 25.$$
By Proposition \ref{criterion 2}, $|-mK_X|$ is not composed with a pencil for $m\geq 25.$

Now by Proposition \ref{rr inequality}, taking $t=30/13$ and using
 the fact that $(-K_X)^3\geq 0.12$, $r_X\leq 660$, and $r_{\max} \leq  13$, we have $P_{-10}\geq 16$. 
 We may take $m_0=8$ since $P_{-8}\geq 2$. If $|-8K_X|$ and $|-10K_X|$ are not composed with the same pencil, take $m_1=10$ and $\mu_0\leq 8$; if $|-8K_X|$ and $|-10K_X|$ are composed with the same pencil, take $m_1=25$ and $\mu_0\leq \frac{10}{15}$ by Remark \ref{5.3}. Then by Proposition \ref{criterion b}(2), $\varphi_{-m}$ is birational for $m\geq 51$.

(2) If $r_{\max}\geq  14$ and $-K_X^3\geq 0.21$, then by \cite[Proof of Theorem 1.8, Case II (page 105)]{CJ16}, $r_X\leq 240$.
By Lemma \ref{lem l/d},
$$\frac{\lambda(M_X)}{-K_X^3}\leq \max\left\{1, \frac{3}{0.21}, \sqrt{\frac{2\times 240}{0.21}}\right\}<48.$$
Note that $r_{\max}\leq 24$. Take $t=17/8$, then
$$-\frac{3}{4}+\sqrt{\frac{12}{t(-K_X^3)}+\frac{6\lambda(M_X)}{-K_X^3}+\frac{1}{16}}< 17.$$
By Proposition \ref{criterion 2}, $|-mK_X|$ is not composed with a pencil for $m\geq 17.$

Now by Proposition \ref{rr inequality}, taking $t=1$ and using
 the fact that $(-K_X)^3\geq 0.21$, $r_X\leq 240$, and $r_{\max} \leq  24$, we have $P_{-8}\geq 7$. Again by Proposition \ref{rr inequality}, taking $t=11/8$, we have $P_{-11}\geq 39$. We may take $m_0=8$. If $|-8K_X|$ is not composed with a pencil, take $m_1=8$ and $\mu_0\leq 8$; if $|-8K_X|$ is composed with a pencil, but $|-8K_X|$ and $|-11K_X|$ are not composed with the same pencil, take $m_1=11$ and $\mu_0\leq \frac{8}{6}$ by Remark \ref{5.3}; if $|-8K_X|$ and $|-11K_X|$ are composed with the same pencil, take $m_1=17$ and $\mu_0\leq \frac{11}{38}$ by Remark \ref{5.3}. Then by Proposition \ref{criterion b}(1), $\varphi_{-m,X}$ is birational for $m\geq 51$.
\end{proof}

\subsection{The case $P_{-1}=0$, $P_{-2}>0$ and $-K_X^3< 0.21$} \ 

We mainly apply Chen--Chen's method to classify all possible baskets.

\begin{thm}\label{thm P1=0 <0.21}
Let $(X, Y, Z)$ be a Fano--Mori triple such that $\rho(Y)>1$. Assume that $P_{-1}=0$, $P_{-2}>0$, and $-K_X^3< 0.21$.
Then
$\varphi_{-m,X}$ is birational for all $m\geq 51$.
\end{thm}
\begin{proof}
By \cite[Inequality (4.1)]{CC},
\begin{align*}
0.21>-K_X^3\geq \frac{1}{12}(-1-P_{-1}-P_{-2}+P_{-4}),
\end{align*}
which means that
\begin{align}\label{426}
P_{-4}\leq P_{-2}+3.
\end{align}
Since $P_{-1}=0$, the basket $B^{(0)}$ has datum
 \begin{numcases}{}
n_{1,2}^0=5 +4P_{-2}-P_{-3};\notag\\
n_{1,3}^0=4 -2P_{-2}+3P_{-3}-P_{-4};\notag\\
n_{1,4}^0=1 -P_{-2}-2P_{-3}+P_{-4}-\sigma_5.\notag
\end{numcases}
By Lemma \ref{31}, $B^{(0)}$ satisfies inequality \eqref{kwmt} and thus
\begin{align*}
0{}&\leq \gamma(B^{(0)})= \sum_{r \ge 2}(\frac{1}{r}-r)
n^0_{1,r}+24\\
{}&\leq\sum_{r=2,3,4}(\frac{1}{r}-r)
n^0_{1,r}-\frac{24}{5}\sigma_5+24\\
{}&=\frac{25}{12}+\frac{37}{12}P_{-2}+P_{-3}-\frac{13}{12}P_{-4}-\frac{21}{20}\sigma_5.
\end{align*}
Hence, by $n^0_{1,3}\geq 0$ and $n^0_{1,4}\geq 0$, we have
\begin{numcases}{}
\frac{25}{12}+\frac{37}{12}P_{-2}+P_{-3}-\frac{13}{12}P_{-4}-\frac{21}{20}\sigma_5\geq 0;\label{1}\\
4 -2P_{-2}+3P_{-3}-P_{-4} \geq 0;\label{2}\\
1 -P_{-2}-2P_{-3}+P_{-4}-\sigma_5 \geq 0.\label{3}
\end{numcases}
Considering the inequality ``\eqref{1}+\eqref{2}+$2\times$\eqref{3}'':
\begin{align}\label{24}
\frac{97}{12}-\frac{11}{12}P_{-2}-\frac{1}{12}P_{-4}-\frac{61}{20}\sigma_5\geq 0,
\end{align}
we obtain
$\sigma_5\leq 2.$
\medskip

{\bf Case 1.} $\sigma_5=0.$

In this case, $B^{(0)}=\{ n^0_{1,2}\times (1,2), n^0_{1,3}\times (1,3), n^0_{1,4}\times (1,4)\}$. By \cite[Proof of Theorem 3.12, Subcase II-1]{CJ16}, $P_{-3}\geq 1$. By inequalities \eqref{3} and \eqref{426},
$$
1+P_{-4}\geq P_{-2}+2P_{-3}\geq P_{-4}-3+2P_{-3},
$$
which means that $P_{-3}\leq 2$. By inequality \eqref{2} and the fact that $P_{-4}\geq 2P_{-2}-1$, 
$$
10\geq 2P_{-2}+P_{-4}\geq 4P_{-2}-1,
$$
which means that $P_{-2}\leq 2$.
\medskip

{\bf Subcase 1-i. $P_{-2}=2$.} 

In this subcase, inequalities \eqref{2} and \eqref{3} imply that $3P_{-3}\geq P_{-4}\geq 2P_{-3}+1$. Combining with the fact that $1\leq P_{-3}\leq 2$ and inequality \eqref{426}, $(P_{-3}, P_{-4})$ can only be $(1,3)$ or $(2,5)$. 
Hence we may compute corresponding values of $(n^0_{1,2},n^0_{1,3},n^0_{1,4})$, which is $(12, 0, 0)$ or
$(11,1,0)$.


Hence we may get all possible baskets $B_X$ as packings of corresponding $B^{(0)}$ with $\gamma\geq 0$, which are listed in Table \ref{tab6}. In Table \ref{tab6},
for each basket $B_X$, 
if $r_X\leq 69$ and $r_{\max}\leq 12$, then we apply Lemma \ref{lem 12}(1) (such baskets are marked with {\checkmark} in the last column), otherwise we can compute $M_X$ and $\lambda(M_X)$, then find $n_1$ such that $P_{-n_1}>\lambda(M_X)n_1+1$ where $P_{-n_1}$ is computed by Reid's Riemann--Roch formula. Hence by Proposition \ref{criterion 1}, $\dim \overline{\varphi_{-n_1}(X)}>1$. By assumption, $P_{-2}= 2$, we may take $m_0=2$. Then by Proposition \ref{criterion b}(1), take $m_1=n_1$ and $\mu_0\leq m_0$, we get the integer $n_2$ such that $\varphi_{-m}$ is birational for all $m\geq n_2$. 

{\scriptsize
\begin{longtable}{LLCCCCCCC}
\caption{}\label{tab6}\\
\hline
B_X & -K^3 & M_X & \lambda(M_X) & n_1 & m_0 & r_{\max} & n_2 \\
\hline
\endfirsthead
\multicolumn{3}{l}{{ {\bf \tablename\ \thetable{}} \textrm{-- continued from previous page}}}
 \\
\hline 
B_X & -K^3 & M_X & \lambda(M_X) & n_1 & m_0 & r_{\max} & n_2 \\ \hline 
\endhead

\hline \multicolumn{3}{r}{{\textrm{Continued on next page}}} \\ \hline
\endfoot

\hline \hline
\endlastfoot
\{12\times(1,2)\}& 0& & & & & & \\
\hline
\{11\times(1,2),(1,3)\}& &  &  &  &  &  & \checkmark \\
\{10\times(1,2),(2,5)\}& &  &  &  &  &  & \checkmark \\
\{9\times(1,2),(3,7)\}& &  &  &  &  &  & \checkmark \\
\{8\times(1,2),(4,9)\}& &  &  &  &  &  & \checkmark \\
\{7\times(1,2),(5,11)\}& &  &  &  &  &  & \checkmark \\
\{6\times(1,2),(6,13)\}& 3/13 & 6 & 3 & 9 & 2 & 13& 33 \\
\{5\times(1,2),(7,15)\}& 7/30 & 7 & 7/2 & 9 & 2 & 15 & 33 \\
\{4\times(1,2),(8,17)\}& 4/17 & 8 & 4 & 10 & 2 & 17 & 36 \\
\{3\times(1,2),(9,19)\}& 9/38 & 9 & 4 & 10 & 2 & 19 & 36 \\
\{2\times(1,2),(10,21)\}& 5/21 & 10 & 4 & 10 & 2 & 21 & 36 
\end{longtable} 
}
\medskip 

{\bf Subcase 1-ii. $P_{-2}=1$.} 

In this subcase, inequalities \eqref{2} and \eqref{3} imply that $3P_{-3}+2\geq P_{-4}\geq 2P_{-3}$. Combining with the fact that $1\leq P_{-3}\leq 2$ and inequality \eqref{426}, $(P_{-3}, P_{-4})$ can only take the following values: $(1,2)$, $(1,3)$, $(1,4)$, 
$(2,4)$.
Hence we may compute corresponding values of $B^{(0)}$ with $\gamma(B^{(0)})\geq 0$ 
and list those possible  $B^{(0)}$ in the following table.

{\scriptsize
\begin{longtable}{L}
\caption{}\label{tab7}\\
B^{(0)} \\
\hline
\endfirsthead
\{8\times(1,2),3\times(1,3)\}\\
\{8\times(1,2),2\times(1,3),(1,4)\}\\
\{8\times(1,2),(1,3),2\times(1,4)\}\\
\{7\times(1,2),4\times(1,3)\}\\
\end{longtable} 
}

Hence we may get all possible baskets $B_X$ as packings of corresponding $B^{(0)}$ with $\gamma\geq 0$, which are listed in Table \ref{tab8}. In Table \ref{tab8},
for each basket $B_X$, 
if $r_X\leq 69$ and $r_{\max}\leq 12$, then we apply Lemma \ref{lem 12}(1) (such baskets are marked with {\checkmark} in the last column), otherwise we can compute $M_X$ and $\lambda(M_X)$, then find $n_1$ such that $P_{-n_1}>\lambda(M_X)n_1+1$ where $P_{-n_1}$ is computed by Reid's Riemann--Roch formula. Hence by Proposition \ref{criterion 1}, $\dim \overline{\varphi_{-n_1}(X)}>1$. Again by Reid's Riemann--Roch formula, we may find $m_0$ such that $P_{-m_0}\geq 2$. Then by Proposition \ref{criterion b}(2), take $m_1=n_1$ and $\mu_0\leq m_0$, we get the integer $n_2$ such that $\varphi_{-m}$ is birational for all $m\geq n_2$. (For the values of $n_2$ with the mark ``$*$'', we apply Proposition \ref{criterion b}(1) instead).

{\scriptsize
\begin{longtable}{LLCCCCCCC}
\caption{}\label{tab8}\\
\hline
B_X & -K^3 & M_X & \lambda(M_X) & n_1 & m_0 & r_{\max} & n_2 \\
\hline
\endfirsthead
\multicolumn{3}{l}{{ {\bf \tablename\ \thetable{}} \textrm{-- continued from previous page}}}
 \\
\hline 
B_X & -K^3 & M_X & \lambda(M_X) & n_1 & m_0 & r_{\max} & n_2 \\ \hline 
\endhead

\hline \multicolumn{3}{r}{{\textrm{Continued on next page}}} \\ \hline
\endfoot

\hline \hline
\endlastfoot
\{8\times(1,2),3\times(1,3)\}& 0 & & & & & & \\
\{7\times(1,2),(2,5), 2\times(1,3)\}& &  &  &  &  &  & \checkmark \\
\{6\times(1,2),(3,7), 2\times(1,3)\}& &  &  &  &  &  & \checkmark \\
\{5\times(1,2),(4,9), 2\times(1,3)\}& &  &  &  &  &  & \checkmark \\
\{4\times(1,2),(5,11), 2\times(1,3)\}& &  &  &  &  &  & \checkmark \\
\{3\times(1,2),(6,13), 2\times(1,3)\}& 5/78 & 5 & 3 & 16 & 4 & 13 & 46 \\
\{2\times(1,2),(7,15), 2\times(1,3)\}& 1/15 & 2 & 2 & 13 & 4 & 15 & 47 \\
\{(1,2),(8,17), 2\times(1,3)\}& 7/102 & 7 & 7/2 & 17 & 4 & 17 & 55? \\
\{7\times(1,2),(3,8), (1,3)\}& &  &  &  &  &  & \checkmark \\
\{7\times(1,2),(4,11)\}& &  &  &  &  &  & \checkmark \\
\{6\times(1,2),2\times(2,5),(1,3)\}& &  &  &  &  &  & \checkmark \\
\{5\times(1,2),(3,7),(2,5),(1,3)\}& 17/210 & 17 & 17/3 & 20 & 4 & 7 & 40 \\
\{4\times(1,2),(4,9),(2,5),(1,3)\}& 4/45 & 8 & 4 & 16 & 4 & 9& 38 \\
\{3\times(1,2),(5,11),(2,5),(1,3)\}& 31/330& 31 & 31/4 & 22 & 4 & 11 & 48 \\
\{2\times(1,2),(6,13),(2,5),(1,3)\}& 19/195 & 38 & 8 & 22 & 4& 13 & 52? \\
\{(1,2),(7,15),(2,5),(1,3)\}& 1/10 & 3 & 3 & 13 & 4 & 15 & 47 \\
\{5\times(1,2),(5,12),(1,3)\}& &  &  &  &  &  & \checkmark \\
\{6\times(1,2),(2,5),(3,8)\}& &  &  &  &  &  & \checkmark \\
\{5\times(1,2),(3,7),(3,8)\}& &  &  &  &  &  & \checkmark \\
\{4\times(1,2),(4,9),(3,8)\}& 7/72 & 7 & 7/2 & 14 & 4 & 9 & 36 \\
\{3\times(1,2),(5,11),(3,8)\}& 9/88& 9 & 4 & 15 & 4 & 11 & 41 \\
\{2\times(1,2),(6,13),(3,8)\}& 11/104 & 11 & 4 & 15 & 4 & 13& 45 \\
\{4\times(1,2),2\times(3,7),(1,3)\}& &  &  &  &  &  & \checkmark \\
\{3\times(1,2),(4,9),(3,7),(1,3)\}& 13/126& 13 & 13/3 & 16 & 4 & 9 & 38 \\
\{2\times(1,2),(5,11),(3,7),(1,3)\}& 25/231 & 50 & 10 & 23 & 4 & 11 & 49 \\
\{(1,2),(6,13),(3,7),(1,3)\}& 61/546 & 61 & 61/6 & 23 & 4 & 13 & 53? \\
\{3\times(1,2),(7,16),(1,3)\}& 5/48 & 5 & 3 & 13 & 4 & 16 & 49 \\
\{2\times(1,2),2\times(4,9),(1,3)\}& &  &  &  &  &  & \checkmark \\
\{(1,2),(5,11),(4,9),(1,3)\}& 23/198 & 23 & 6 & 18 & 4 & 11 & 44 \\
\{6\times(1,2),(5,13)\}& 1/13 & 2 & 2 & 12 & 4 & 13 & 42 \\
\{5\times(1,2),3\times(2,5)\}& &  &  &  &  &  & \checkmark \\
\{4\times(1,2),(3,7), 2\times(2,5)\}& 4/35 & 8 & 4 & 14 & 4 & 7 & 32 \\
\{3\times(1,2),(4,9), 2\times(2,5)\}& 11/90 & 11 & 4 & 14 & 4 & 9 & 36 \\
\{2\times(1,2),(5,11), 2\times(2,5)\}& 7/55 & 14 & 14/3 & 15 & 4 & 11& 41 \\
\{4\times(1,2),(5,12), (2,5)\}& &  &  &  &  &  & \checkmark \\
\{4\times(1,2),(7,17)\}& 2/17 & 4 & 3 & 12 & 4 & 17 & 48* \\
\{3\times(1,2),2\times(3,7), (2,5)\}& 9/70 & 9 & 4 & 14 & 4 & 7 & 32 \\
\{2\times(1,2),(4,9),(3,7), (2,5)\}& 43/315 & 86 & 86/7 & 23 & 4 & 9 & 45 \\
\{2\times(1,2),(7,16), (2,5)\}& 11/80 & 11 & 4 & 13 & 4 & 16 & 49 \\
\{2\times(1,2),(4,9),(5,12)\}& &  &  &  &  &  & \checkmark \\
\{3\times(1,2),(3,7),(5,12)\}& 11/84 & 11 & 4 & 13 & 4 & 12 & 41 \\
\{3\times(1,2),(8,19)\}& 5/38 & 5 & 3 & 12 & 4 & 19 & 48* \\
\{2\times(1,2),3\times(3,7)\}& &  &  &  &  &  & \checkmark \\
\hline
\{8\times(1,2),2\times(1,3),(1,4)\}& &  &  &  &  &  & \checkmark \\
\{7\times(1,2),(2,5),(1,3),(1,4)\}& &  &  &  &  &  & \checkmark \\
\{7\times(1,2),(3,8),(1,4)\}& &  &  &  &  &  & \checkmark \\
\{6\times(1,2),(3,7),(1,3),(1,4)\}& 11/84 & 11 & 4 & 13 & 4 & 7 & 31 \\
\{5\times(1,2),(4,9),(1,3),(1,4)\}& &  &  &  &  &  & \checkmark \\
\{4\times(1,2),(5,11),(1,3),(1,4)\}& 19/132& 19 & 6 & 16 & 4 & 11 & 42 \\
\{3\times(1,2),(6,13),(1,3),(1,4)\}& 23/156 & 23 & 6 & 15 & 4 & 13 & 45 \\
\{8\times(1,2),(1,3),(2,7)\}& &  &  &  &  &  & \checkmark \\
\{7\times(1,2),(2,5),(2,7)\}& 9/70 & 9 & 4 & 13 & 4 & 7 & 31 \\
\{6\times(1,2),(3,7),(2,7)\}& &  &  &  &  &  & \checkmark \\
\{5\times(1,2),(4,9),(2,7)\}& 19/126 & 19 & 6 & 15 & 4 & 9 & 37 \\
\{4\times(1,2),(5,11),(2,7)\}&12/77 & 24 & 6 & 15 & 4 & 11 & 41 \\
\{6\times(1,2),2\times(2,5),(1,4)\}& &  &  &  &  &  & \checkmark \\
\{5\times(1,2),(3,7),(2,5),(1,4)\}& 23/140& 23 & 6 & 15 & 4 & 7 & 33 \\
\{5\times(1,2),(5,12),(1,4)\}& &  &  &  &  &  & \checkmark \\
\{4\times(1,2),(4,9),(2,5),(1,4)\}& 31/180 & 31 & 31/4 & 16 & 4 & 9 & 38 \\
\{3\times(1,2),(5,11),(2,5),(1,4)\}& 39/220 & 39 & 8 & 16 & 4 & 11 & 42 \\
\{4\times(1,2),2\times(3,7),(1,4)\}& &  &  &  &  &  & \checkmark \\
\{3\times(1,2),(4,9),(3,7),(1,4)\}& 47/252 & 47 & 47/5 & 17 & 4 & 9 & 39 \\
\{8\times(1,2),(3,10)\}& &  &  &  &  &  & \checkmark \\
\hline
\{8\times(1,2),(1,3),2\times(1,4)\}& &  &  &  &  &  & \checkmark \\
\{7\times(1,2),(2,5),2\times(1,4)\}& &  &  &  &  &  & \checkmark \\
\{6\times(1,2),(3,7),2\times(1,4)\}& &  &  &  &  &  & \checkmark \\
\{5\times(1,2),(4,9),2\times(1,4)\}& &  &  &  &  &  & \checkmark \\
\{8\times(1,2),(2,7),(1,4)\}& &  &  &  &  &  & \checkmark \\
\{8\times(1,2),(3,11)\}& &  &  &  &  &  & \checkmark \\
\hline
\{7\times(1,2),4\times(1,3)\}& &  &  &  &  &  & \checkmark \\
\{6\times(1,2),(2,5),3\times(1,3)\}& &  &  &  &  &  & \checkmark \\
\{5\times(1,2),(3,7),3\times(1,3)\}& &  &  &  &  &  & \checkmark \\
\{4\times(1,2),(4,9),3\times(1,3)\}& &  &  &  &  &  & \checkmark \\
\{3\times(1,2),(5,11),3\times(1,3)\}& &  &  &  &  &  & \checkmark \\
\{2\times(1,2),(6,13),3\times(1,3)\}&3/13 & 18 & 6 & 12 & 3 & 13& 41 \\
\{6\times(1,2),(3,8),2\times(1,3)\}& &  &  &  &  &  & \checkmark \\
\{6\times(1,2),(4,11),(1,3)\}& &  &  &  &  &  & \checkmark \\
\{6\times(1,2),(5,14)\}& 3/14 & 3 & 3 & 9 & 3 & 14 & 36* \\
\{5\times(1,2),2\times(2,5), 2\times(1,3)\}& &  &  &  &  &  & \checkmark \\
\{4\times(1,2),(3,7),(2,5), 2\times(1,3)\}& 26/105 & 52 & 10 & 15 & 3 & 7& 32 \\
\{3\times(1,2),(4,9),(2,5), 2\times(1,3)\}& 23/90 & 23 & 6 & 12 & 3 & 9& 33 \\
\{3\times(1,2),(4,9),(3,8), (1,3)\}& 19/72& 19 & 6 & 12 & 3 & 9 & 33 \\
\{4\times(1,2),(5,12), 2\times(1,3)\}& &  &  &  &  &  & \checkmark \\
\{4\times(1,2),(3,7),(3,8),(1,3)\}& 43/168 & 43 & 43/5 & 14 & 3 & 8& 33 \\
\{4\times(1,2),(3,7),(4,11)\}& 20/77 & 40 & 8 & 13 & 3 & 11& 38 \\
\{3\times(1,2),2\times(3,7), 2\times(1,3)\}& &  &  &  &  &  & \checkmark \\
\{5\times(1,2),(2,5), (3,8),(1,3)\}& 29/120 & 29 & 29/4 & 13 & 3 & 8 & 32 \\
\{5\times(1,2),(5,13),(1,3)\}&19/78 & 19 & 6 & 12 & 3 & 13 & 41 \\
\{5\times(1,2),(2,5), (4,11)\}& 27/110 & 27 & 27/4 & 13 & 3 & 11 & 38 \\
\{5\times(1,2),2\times(3,8)\}& &  &  &  &  &  & \checkmark \\
\{4\times(1,2),3\times(2,5), (1,3)\}& &  &  &  &  &  & \checkmark \\
\{3\times(1,2),(3,7),2\times(2,5), (1,3)\}&59/210 & 59 & 10 & 14 & 3 & 7& 31 \\
\{3\times(1,2),(5,12),(2,5), (1,3)\}& &  &  &  &  &  & \checkmark \\
\{4\times(1,2),2\times(2,5), (3,8)\}& &  &  &  &  &  & \checkmark \\
\{4\times(1,2),(2,5), (5,13)\}&18/65 & 36 & 8 & 13 & 3 & 13& 42 \\
\{4\times(1,2),(7,18)\}&5/18 & 5 & 3 & 8 & 3 & 18& 33* \\
\{3\times(1,2),4\times(2,5)\}& &  &  &  &  &  & \checkmark \\
\end{longtable} 
}
Note that there are 3 cases where the values in the $n_2$ column are marked 
with ``?", of which $n_2=55, 52, 53$ respectively.   We shall discuss them in more details to prove our statement.

If $B_X=\{(1,2), (8,17),2\times(1,3)\}$, we know that $\dim \overline{\varphi_{-m}(X)}>1$ for all $m\geq 17$ from the list. Note that $P_{-10}=13$. Take $m_0=4$. If $|-10K_X|$ and $|-4K_X|$ are not composed with the same pencil, then we may take $m_1=10$ and $\mu_0\leq 4$, and by Proposition \ref{criterion b}(1), $\varphi_{-m}$ is birational for all $m\geq 42$; if $|-10K_X|$ and $|-4K_X|$ are composed with the same pencil, then we may take $m_1=17$ and $\mu_0\leq \frac{10}{12}$ by Remark \ref{5.3}, and by Proposition \ref{criterion b}(2), $\varphi_{-m}$ is birational for all $m\geq 51$.

If $B_X=\{2\times(1,2),(6,13), (2,5),(1,3)\}$, we know that $\dim \overline{\varphi_{-m}(X)}>1$ for all $m\geq 22$ from the list. Note that $P_{-8}=10$. Take $m_0=4$. If $|- 8K_X|$ and $|-4K_X|$ are not composed with the same pencil, then we may take $m_1= 8$ and $\mu_0\leq 4$, and by Proposition \ref{criterion b}(1), $\varphi_{-m,X}$ is birational for all $m\geq 36$; if $|- 8K_X|$ and $|-4K_X|$ are composed with the same pencil, then we may take $m_1=22$ and $\mu_0\leq \frac{8}{9}$ by Remark \ref{5.3}, and by Proposition \ref{criterion b}(2), $\varphi_{-m}$ is birational for all $m\geq 48$.

If $B_X=\{(1,2),(6,13), (3,7),(1,3)\}$, we know that $\dim \overline{\varphi_{-m}(X)}>1$ for all $m\geq 23$ from the list. Note that $P_{-8}=11$. Take $m_0=4$. If $|- 8K_X|$ and $|-4K_X|$ are not composed with the same pencil, then we may take $m_1= 8$ and $\mu_0\leq 4$, and by Proposition \ref{criterion b}(1), $\varphi_{-m,X}$ is birational for all $m\geq 36$; if $|- 8K_X|$ and $|-4K_X|$ are composed with the same pencil, then we may take $m_1=23$ and $\mu_0\leq \frac{8}{10}$ by Remark \ref{5.3}, and by Proposition \ref{criterion b}(2), $\varphi_{-m}$ is birational for all $m\geq 49$.
\medskip

{\bf Case 2.} $\sigma_5=2.$

In this case, by \cite[Proof of Theorem 3.12, Subcase II-2]{CJ16}, we have $(P_{-2}, P_{-3}, P_{-4})=(1,0,2)$ and $$B_X=B^{(0)}=\{9\times (1,2), 2\times (1,5)\}.$$
We may apply Lemma \ref{lem 12}(1).
\medskip

{\bf Case 3.} $\sigma_5=1.$

In this case, all possible baskets $B_X$ are classified in \cite[Proof of Theorem 3.12, Subcase II-3]{CJ16} under the strict inequality $\gamma(B^{(0)})>0$. The output is however the same even under the weaker condition $\gamma(B^{(0)})\geq 0$, which is concluded after checking the original proof word by word.  
So we simply list all possible baskets in Table \ref{tab9}. 
In Table \ref{tab9},
for each basket $B_X$,  
if $r_X\leq 69$ and $r_{\max}\leq 12$, then we may directly apply Lemma \ref{lem 12}(1) (such baskets are marked with {\checkmark} in the last column); if the condition $0<-K_X^3<0.21$ is not satisfied, we mark it with {$\times$} in the last column (such baskets are irrelevant to the proof); otherwise we can compute $M_X$ and $\lambda(M_X)$, then we may find $n_1$ such that $P_{-n_1}>\lambda(M_X)n_1+1$ where $P_{-n_1}$ is computed by Reid's Riemann--Roch formula. Hence by Proposition \ref{criterion 1}, $\dim \overline{\varphi_{-n_1}(X)}>1$. Again by Reid's Riemann--Roch formula, we may find $m_0$ such that $P_{-m_0}\geq 2$. Then by Proposition \ref{criterion b}(2), take $m_1=n_1$ and $\mu_0\leq m_0$, we get the integer $n_2$ such that $\varphi_{-m}$ is birational for all $m\geq n_2$. 

{\scriptsize
\begin{longtable}{LLCCCCCCC}
\caption{}\label{tab9}\\
\hline
B_X & -K^3 & M_X & \lambda(M_X) & n_1 & m_0 & r_{\max} & n_2 \\
\hline
\endfirsthead
\multicolumn{3}{l}{{ {\bf \tablename\ \thetable{}} \textrm{-- continued from previous page}}}
 \\
\hline 
B_X & -K^3 & M_X & \lambda(M_X) & n_1 & m_0 & r_{\max} & n_2 \\ \hline 
\endhead

\hline \multicolumn{3}{r}{{\textrm{Continued on next page}}} \\ \hline
\endfoot

\hline \hline
\endlastfoot
\{11\times (1,2),(1,3),(1,5)\}  & &  &  &  &  &  & \checkmark \\
\{12\times (1,2),(1,5)\}  & &  &  &  &  &  & \checkmark \\
\{12\times (1,2),(1,6)\}  & &  &  &  &  &  & \checkmark \\
\hline
\{9\times (1,2),2\times(1,3),(1,5)\}  & &  &  &  &  &  & \checkmark \\
\hline
 \{10\times (1,2),(1,4),(1,5)\}  & &  &  &  &  &  & \checkmark \\
 \{10\times (1,2),(2,9)\}  & &  &  &  &  &  & \checkmark \\
\hline
 \{10\times (1,2),(1,3),(1,6)\}  & &  &  &  &  &  & \checkmark \\
 \hline
 \{10\times (1,2),(1,3),(1,5)\}  & &  &  &  &  &  & \checkmark \\
 \{9\times (1,2),(2,5),(1,5)\}  & &  &  &  &  &  & \checkmark \\
 \{8\times (1,2),(3,7),(1,5)\} & 18/35& & & & & & \times \\
\hline
\{11\times (1,2),(1,5)\}  & &  &  &  &  &  & \checkmark \\
\hline
\{11\times (1,2),(1,6)\}  & &  &  &  &  &  & \checkmark \\
\hline
\{11\times (1,2),(1,7)\}  & &  &  &  &  &  & \checkmark \\
\hline
\{7\times (1,2),3\times(1,3),(1,5)\}  & &  &  &  &  &  & \checkmark \\
\{6\times (1,2),(2,5),2\times(1,3),(1,5)\}  & &  &  &  &  &  & \checkmark \\
\hline
\{8\times (1,2),(1,3),(1,4),(1,5)\}  & &  &  &  &  &  & \checkmark \\
\{8\times (1,2),(2,7),(1,5)\} & 8/35 &  &  &  &  &  & \times \\
\{8\times (1,2),(1,3),(2,9)\}  & &  &  &  &  &  & \checkmark \\
\{7\times (1,2),(2,5),(1,4),(1,5)\}  & &  &  &  &  &  & \checkmark \\
\hline
\{8\times (1,2),2\times(1,3),(1,6)\}  & &  &  &  &  &  & \checkmark \\
\{7\times (1,2),(2,5),(1,3),(1,6)\}  & &  &  &  &  &  & \checkmark \\
\hline
\{8\times (1,2),2\times(1,3),(1,5)\}  & &  &  &  &  &  & \checkmark \\
\{7\times (1,2),(2,5),(1,3),(1,5)\}  & &  &  &  &  &  & \checkmark \\
\{6\times (1,2),2\times(2,5),(1,5)\}  & &  &  &  &  &  & \checkmark \\
\{6\times (1,2),(3,7),(1,3),(1,5)\} & 19/105& 38 & 8 & 16 & 4 & 7 & 34 \\
\{5\times (1,2),(4,9),(1,3),(1,5)\} & 17/90 & 17 & 17/3 & 13 & 4 & 9 & 35 \\
\{7\times (1,2),(3,8),(1,5)\}  & &  &  &  &  &  & \checkmark \\
\{5\times (1,2),(3,7),(2,5),(1,5)\} & 3/14 &  &  &  &  &  &\times\\
\hline
\{9\times (1,2),(1,4),(1,6)\}  & &  &  &  &  &  & \checkmark \\
\hline
\{9\times (1,2),(1,4),(1,5)\}  & &  &  &  &  &  & \checkmark \\
\{9\times (1,2),(2,9)\}  & &  &  &  &  &  & \checkmark \\
\hline
\{9\times (1,2),(1,3),(1,7)\}  & &  &  &  &  &  & \checkmark \\
\{8\times (1,2),(2,5),(1,7)\} & 2/35 & 4 & 3 & 18 & 6 & 7 & 40 \\
\hline
\{8\times (1,2),(2,5),(1,6)\}  & &  &  &  &  &  & \checkmark \\
\{7\times (1,2),(3,7),(1,6)\}  & &  &  &  &  &  & \checkmark \\
\{6\times (1,2),(4,9),(1,6)\}  & &  &  &  &  &  & \checkmark \\
\hline
\{7\times (1,2),(3,7),(1,5)\} & 1/70 & 1 & 1 & 20 & 6 & 7& 43 \\
\{6\times (1,2),(4,9),(1,5)\} & 1/45 & 2 & 2 & 23 & 6 & 9 & 48 \\
\{5\times (1,2),(5,11),(1,5)\} & 3/110 & 3 & 3 & 25 & 6 & 11 & 53? \\
\{4\times (1,2),(6,13),(1,5)\} & 2/65 & 4 & 3 & 24 & 6 & 13 & 56?
\end{longtable} 
}
Note that there are 2 cases with ``?" marked in the $n_2$ column of Table \ref{tab9}, for which we discuss them in more details as follows.

If $B_X=\{5\times (1,2),(5,11),(1,5)\}$, we know that $\dim \overline{\varphi_{-m}(X)}>1$ for all $m\geq 25$ from the list. Note that $P_{-15}=18$. Take $m_0=6$. If $|-15K_X|$ and $|-6K_X|$ are not composed with the same pencil, then we may take $m_1=15$ and $\mu_0\leq 6$, and by Proposition \ref{criterion b}(2), $\varphi_{-m}$ is birational for all $m\geq 43$; if $|-15K_X|$ and $|-6K_X|$ are composed with the same pencil, then we may take $m_1=25$ and $\mu_0\leq \frac{15}{17}$ by Remark \ref{5.3}, and by Proposition \ref{criterion b}(2), $\varphi_{-m}$ is birational for all $m\geq 47$.

If $B_X=\{4\times (1,2),(6,13),(1,5)\}$, we know that $\dim \overline{\varphi_{-m}(X)}>1$ for all $m\geq 24$ from the list. Note that $P_{-14}=16$. Take $m_0=6$. If $|- 14K_X|$ and $|-6K_X|$ are not composed with the same pencil, then we may take $m_1= 14$ and $\mu_0\leq 6$, and by Proposition \ref{criterion b}(2), $\varphi_{-m}$ is birational for all $m\geq 46$; if $|- 14K_X|$ and $|-6K_X|$ are composed with the same pencil, then we may take $m_1=24$ and $\mu_0\leq \frac{14}{15}$ by Remark \ref{5.3}, and by Proposition \ref{criterion b}(2), $\varphi_{-m}$ is birational for all $m\geq 50$.

Combining all above cases, we complete the proof.
\end{proof}

\subsection{The case $P_{-1}>0$ and $1/30\leq -K_X^3< 0.21$}

\begin{lem}\label{lem 30 1}
Let $(X, Y, Z)$ be a Fano--Mori triple such that $\rho(Y)>1$. Assume that $-K_X^3\geq 1/30$ and $r_X\leq 660$. Then \begin{enumerate}
\item 
$P_{-18}\geq 21$;
\item  $\dim \overline{\varphi_{-m}(X)}>1$ for all $m\geq 35$.
 \end{enumerate}
\end{lem}

\begin{proof}
(1) 
By Proposition \ref{rr inequality}, take $t=18/8$ and use the fact $-K_X^3\geq 1/30$, we have $P_{-18}>20$.

For (2),
by Lemma \ref{lem l/d},
$$\frac{\lambda(M_X)}{-K_X^3}\leq \max\left\{1, \frac{3}{1/30}, \sqrt{\frac{2\times 660}{1/30}}\right\}<199.$$
Take $t=35/8$, then
$$-\frac{3}{4}+\sqrt{\frac{12}{t(-K_X^3)}+\frac{6\lambda(M_X)}{-K_X^3}+\frac{1}{16}}< 35.$$
By Proposition \ref{criterion 2}, $|-mK_X|$ is not composed with a pencil for $m\geq 35.$
\end{proof}

\begin{lem}\label{lem 13-24}
Let $(X, Y, Z)$ be a Fano--Mori triple such that $\rho(Y)>1$. Assume that $P_{-1}>0$ and $1/30\leq -K_X^3< 0.21$.
Assume that one of the following conditions holds:
\begin{enumerate}
\item $11\leq r_{\max}\leq 13$ and  $-K_X^3< 0.12$; or
\item $r_{\max}\geq 14$.
\end{enumerate}
Then $\varphi_{-m,X}$ is birational for all $m\geq 51$.
\end{lem}

\begin{proof}
We will use Chen--Chen's method to classify all possible baskets.

Consider $B_X=\{(b_i, r_i)\}$, we may always assume that $r_1=r_{\max}$.


By equality \eqref{P1K3}, 
\begin{align}
2P_{-1}{}&=-K_X^3+6-\sum_i\frac{b_i(r_i-b_i)}{r_i}\notag\\
{}&\leq-K_X^3+6-\frac{b_1(r_1-b_1)}{r_1}\label{b1r1}\\
{}&\leq-K_X^3+6-\frac{r_1-1}{r_1}\notag\\
{}&<6.\notag
\end{align}
Hence $P_{-1}\leq 2$.
\medskip

{\bf Case 1.} $P_{-1}=2.$

By \cite[Inequality (4.1)]{CC},
\begin{align*}
0.21>-K_X^3\geq \frac{1}{12}(-1-P_{-1}-P_{-2}+P_{-4}),
\end{align*}
which means that
\begin{align}\label{426'}
P_{-4}\leq P_{-2}+5.
\end{align}

By inequality \eqref{b1r1}, 
$$
 \frac{b_1(r_1-b_1)}{r_1}\leq 2-K_X^3,
$$
which implies that $b_1\leq 2$. If $b_1=1$, then $(1,r_1)\in B_X$, which means that $(1,r_1)\in B^{(0)}$. If $b_1=2$, then $r_1$ is an odd number, we may write $r_1=2s+1$ for some $s\geq 5$, and $(2,r_1)\in B_X$ implies that $\{(1,s), (1,s+1)\}\subset B^{(0)}$. 

Recall that $B^{(0)}=\{n_{1,r}^0\times(1,r)\}$, we have
\begin{align}
\sum_r \frac{n_{1,r}^0\times(r-1)}{r}= 2-K^3(B^{(0)})\leq 2-K_X^3<2.21. \label{B0221}
\end{align}
Note that either $n^0_{1,r_1}\geq 1$, or $n^0_{1,s}\geq 1$ and  $n^0_{1,s+1}\geq 1$ where $r_1=2s+1$. Also we have 
$$\sigma(B^{(0)})=\sigma(B_X)=10-5P_{-1}+P_{-2}=P_{-2}\geq 2P_{-1}-1=3.$$ 
If $\sigma\geq 4$, then 
$$
\sum_r \frac{n_{1,r}^0\times(r-1)}{r}\geq \begin{cases}
3\times\frac{1}{2}+\frac{10}{11}>2.21, & {\rm or }\\
2\times\frac{1}{2}+\frac{4}{5}+\frac{5}{6}>2.21,&
\end{cases} 
$$
a contradiction. Hence $\sigma= 3$.
By easy computation, we know that $B^{(0)}$ satisfying inequality \eqref{B0221} is one of the following:
{\scriptsize
\begin{longtable}{L}
\caption{}\label{tab19}\\
B^{(0)} \\
\hline
\endfirsthead
\{2\times(1,2),(1,r_1)\}\\
\{(1,2), (1,3),(1,r_1)\}\\
\{(1,2),(1,4),(1,r_1)\} \\
\{(1,2),(1,s),(1,s+1)\}, r_1=2s+1\\
\hline
\end{longtable} 
}
Note that the first case is absurd since it admits no packing and $-K^3(B^{(0)})<0$.
From Table \ref{tab19}, we may get all possible packings with $\gamma\geq 0$ and $r_{\max}=r_1\geq 11$, 
which are listed in Table \ref{tab20}. In Table \ref{tab20},
for each basket $B_X$, if $r_X\leq 287$ and $r_{\max}\leq 12$, then we apply Lemma \ref{lem 12}(2) (such baskets are marked with {\checkmark} in the last column);  if $-K_X^3$ does not satisfy the assumption in the lemma, we mark it with the symbol ``{$\times$}'' in the last column; otherwise we can compute $M_X$ and $\lambda(M_X)$, then find $n_1$ such that $P_{-n_1}>\lambda(M_X)n_1+1$ where $P_{-n_1}$ is computed by Reid's Riemann--Roch formula. Hence by Proposition \ref{criterion 1}, $\dim \overline{\varphi_{-m}(X)}>1$ for all $m\geq n_1$ since $P_{-1}>0$. By assumption, $P_{-1}= 2$, we may take $m_0=\nu_0=1$. Then by Proposition \ref{criterion b}(1), take $m_1=n_1$ and $\mu_0\leq m_0$, we get the integer $n_2$ such that $\varphi_{-m}$ is birational for all $m\geq n_2$. (For the value of $n_2$ with a ``$*$'' marked, we have applied Proposition \ref{criterion b}(3)).

{\scriptsize
\begin{longtable}{LLCCCCCCC}
\caption{}\label{tab20}\\
\hline
B_X & -K^3 & M_X & \lambda(M_X) & n_1 & m_0 & r_{\max} & n_2 \\
\hline
\endfirsthead
\multicolumn{3}{l}{{ {\bf \tablename\ \thetable{}} \textrm{-- continued from previous page}}}
 \\
\hline 
B_X & -K^3 & M_X & \lambda(M_X) & n_1 & m_0 & r_{\max} & n_2 \\ \hline 
\endhead

\hline \multicolumn{3}{r}{{\textrm{Continued on next page}}} \\ \hline
\endfoot

\hline \hline
\endlastfoot
\{(1,2), (1,3),(1,11)\}& &  &  &  &  & & \checkmark \\
\{(1,2), (1,3),(1,12)\}& &  &  &  &  & &  \checkmark\\
\{(1,2), (1,3),(1,13)\}&7/78 &7  & 7/2 & 14 &1  &13 &41*  \\
\{(1,2), (1,3),(1,14)\}&2/21 &  4& 3 &  13& 1 &14 & 42 \\
\{(1,2), (1,3),(1,15)\}& 1/10& 3 & 3 & 13 & 1 &15 & 42 \\
\{(1,2), (1,3),(1,16)\}&5/48 & 5 & 3 &13  & 1 &16 & 42 \\
\{(1,2), (1,3),(1,17)\}& 11/102& 11 & 4 & 15 & 1 & 17& 48 \\
\{(1,2), (1,3),(1,18)\}& 1/9&  2& 2 &  10&1  &18 & 33 \\
\{(1,2), (1,3),(1,19)\}&13/114 &13  & 13/3 & 15 &1  & 19& 48 \\
\{(2,5),(1,11)\}& &  &  &  &  & &  \checkmark\\
\{(2,5),(1,12)\}& &  &  &  &  & &  \checkmark\\
\{(2,5),(1,13)\}&8/65 &  &  &  &  & &  \times\\
\{(2,5),(1,14)\}& 9/70& 9 & 4 &  13&1  & 14& 42 \\
\{(2,5),(1,15)\}&2/15 & 2 & 2 & 9 & 1 & 15 & 30 \\
\{(2,5),(1,16)\}&11/80 & 11 & 4 &13  & 1 & 16 & 42 \\
\{(2,5),(1,17)\}&12/85 &12  & 4 &13  &  1&17 &  42\\
\{(2,5),(1,18)\}&13/90 & 13 & 13/3 & 13 & 1 & 18& 42 \\
\{(2,5),(1,19)\}& 14/95&  14& 14/3 & 14 &1  &19 &  45\\
\hline
\{(1,2), (1,4),(1,11)\}& &  &  &  &  & &  \checkmark\\
\{(1,2), (1,4),(1,12)\}& &  &  &  &  & & \checkmark \\
\{(1,2), (1,4),(1,13)\}&9/52 &  &  &  &  & & \times \\
\{(1,2), (1,4),(1,14)\}& 5/28& 5 & 3 & 10 & 1 & 14&  33\\
\{(1,2), (1,4),(1,15)\}& 11/60& 11 &4  & 11 & 1 & 15& 36 \\
\{(1,2), (1,4),(1,16)\}& 3/16&  3& 3 & 10 & 1 & 16& 33 \\
\{(1,2), (1,4),(1,17)\}& 13/68&13  & 13/3 &  12& 1 & 17& 39 \\
\{(1,2), (1,4),(1,18)\}& 7/36& 7 & 7/2 & 10 & 1 & 18& 33 \\
\hline
\{(1,2),(2,11)\} & &  &  &  &  & & \checkmark \\
\{(1,2),(2,13)\} & 5/26&  &  &  &  & & \times \\
\{(1,2),(2,15)\} &7/30 &  &  &  &  & &  \times\\
\{(1,2),(2,17)\} & 9/34&  &  &  &  & &  \times \\
\{(1,2),(2,19)\} & 11/38&  &  &  &  & &  \times \\
\{(1,2),(2,21)\} &13/42 &  &  &  &  & &  \times \\
\end{longtable} 
}
\medskip

{\bf Case 2.} $P_{-1}=1.$

By \cite[Inequality (4.1)]{CC},
\begin{align*}
0.21>-K_X^3\geq \frac{1}{12}(-1-P_{-1}-P_{-2}+P_{-4}),
\end{align*}
which means that
\begin{align}\label{424}
P_{-4}\leq P_{-2}+4.
\end{align}

The basket $B^{(0)}$ has datum
 \begin{numcases}{}
n_{1,2}^0=-1 +4P_{-2}-P_{-3};\notag\\
n_{1,3}^0=2 -2P_{-2}+3P_{-3}-P_{-4};\notag\\
n_{1,4}^0=4 -P_{-2}-2P_{-3}+P_{-4}-\sigma_5.\notag
\end{numcases}

By $n_{1,4}^0\geq 0$ and inequality \eqref{424}, we get $P_{-3}\leq 4$ and $P_{-3}=4$ only if $P_{-4}=P_{-2}+4$ and $\sigma_5=0$. 
By $n_{1,3}^0\geq 0$ and $P_{-4}\geq 2P_{-2}-1$, we have $P_{-2}\leq 3$.  Recall that we also have $n_{1,2}^0\geq 0$ and $P_{-3}\geq P_{-2}$ since $P_{-1}=1$.
Hence the possible values of $(P_{-2}, P_{-3})$ are $(3,4)$, $(2,4)$, $(3,3)$,  $(2,3)$, $(1,3)$,  $(2,2)$,  $(1,2)$, $(1,1)$.
\medskip

{\bf Subcase 2-i.} $(P_{-2}, P_{-3})=(3,4).$

In this case, $P_{-4}=7$ and $\sigma_5=0$. Hence $$B^{(0)}=\{7\times (1,2), (1,3) \},$$
and we may get all possible packings with $\gamma\geq 0$ and $r_{\max}\geq 11$, 
which are the following:
\begin{align*}
&\{3\times (1,2), (5,11)\},&&-K^3=5/22;\\
&\{2\times (1,2), (6,13)\},&&-K^3=3/13;\\
&\{(1,2), (7,15)\},&& -K^3=7/30;\\
&\{(8,17)\}, &&-K^3=4/17.
\end{align*}
Since $-K^3>0.21$, none of the baskets satisfy the assumption of the lemma.  
\medskip

{\bf Subcase 2-ii.} $(P_{-2}, P_{-3})=(2,4).$

In this case, $P_{-4}=6$ and $\sigma_5=0$. Hence $$B^{(0)}=\{3\times (1,2), 4\times (1,3) \},$$
and we may get all possible packings with $\gamma\geq 0$ and $r_{\max}\geq 11$, 
which are the following:
\begin{align*}
{}&\{2\times (1,2), (5,14)\},{}&&-K^3=3/14;\\
{}&\{ (1,2), (5,13), (1,3)\},{}&&-K^3=19/78;\\
{}&\{ (2,5), (5,13)\},{}&&-K^3=18/65;\\
{}&\{(5,12),2\times (1,3)\},{}&&-K^3=1/4;\\
{}&\{(7,18)\},{}&&-K^3=5/18.
\end{align*}
It is clear that, for each above basket, $-K^3>0.21$, which does not satisfy the assumption of the lemma.
\medskip

{\bf Subcase 2-iii.} $(P_{-2}, P_{-3})=(3,3).$

In this case, 
 \begin{numcases}{}
n_{1,2}^0=8;\notag\\
n_{1,3}^0=5-P_{-4};\notag\\
n_{1,4}^0=-5+P_{-4}-\sigma_5.\notag
\end{numcases}
Hence $P_{-4}=5$ and $\sigma_5=0$, and $B^{(0)}=\{8\times (1,2)\}.$
It is clear that $B_X=B^{(0)}$, contradicting to assumption of the local index.  
\medskip

{\bf Subcase 2-iv.} $(P_{-2}, P_{-3})=(2,3).$

In this case, 
 \begin{numcases}{}
n_{1,2}^0=4;\notag\\
n_{1,3}^0=7-P_{-4};\notag\\
n_{1,4}^0=-4+P_{-4}-\sigma_5.\notag
\end{numcases}
By inequality \eqref{424}, $P_{-4}\leq 6$. Hence $(P_{-4},\sigma_5)=(6,0)$, $(6,1)$, $(6,2)$, $(5,0)$, $(5,1)$, $(4,0)$.
Hence the corresponding $B^{(0)}$ is in the following list:
{\scriptsize
\begin{longtable}{L}
\caption{}\label{tab21}\\
B^{(0)} \\
\hline
\endfirsthead
\{4\times(1,2),(1,3), 2\times (1,4)\}\\
\{4\times(1,2),(1,3), (1,4), (1,s)\}, s\geq 5\\
\{4\times(1,2),(1,3), (1,s_1), (1,s_2)\}, 5\leq s_1\leq s_2\\
\{4\times(1,2),2\times (1,3), (1,4)\}\\
\{4\times(1,2),2\times (1,3), (1,s)\}, s\geq 5\\
\{4\times(1,2),3\times(1,3)\}\\
\hline
\end{longtable} 
}

Hence all possible packings $B_X$ of $B^{(0)}$ with $\gamma\geq 0$ and $r_{\max}\geq 11$ are dominated by one of the baskets $B'$  listed in Table \ref{tab31}. 
{\scriptsize
\begin{longtable}{c|L}
\caption{}\label{tab31}\\
No. &B' \\
\hline
\endfirsthead
1&\{(5,11), 2\times (1,4)\}\\
2& \{4\times(1,2),(3,11)\}\\
3&\{(5,11), (1,4), (1,s)\}, 5\leq s\leq 9 \\ 
4&\{4\times(1,2),(1,3), (1,4), (1,11)\}\\
5&\{(5,11), (1,s_1), (1,s_2)\}, 5\leq s_1\leq s_2\leq 7\\
6&\{4\times(1,2),(1,3), (2,r_1)\}, r_1=11,13,15\\
7&  \{(5,11), (1,3), (1,4)\}\\
8& \{(1,2), (5,12), (1,4)\}\\
9&\{4\times(1,2),2\times (1,3), (1,r_1)\}, r_1=11,12 \\
10&  \{(5,11), (1,3), (1,5)\}\\
11&  \{(5,11), (1,3), (1,s)\}, 6\leq s\leq 10\\
12&\{(1,2), (5,12), (1,s)\}, 5\leq s\leq 10\\
13& \{(7,17)\}\\
14& \{(5,11),2\times(1,3)\}\\
15& \{(1,2),(5,12), (1,3)\} \\
16& \{3\times(1,2),(4,11)\}\\
\hline
\end{longtable} 
}
In Table \ref{tab31}, for cases No. 1, 3--6, 9, 11--12, one has  $-K^3(B')>0.21$ by direct calculation. Hence $B_X$ can not be dominated by these baskets. 
For cases No. 2, 8, and 10, they are minimal and all satisfy $-K^3>0.12$, which must be excluded. 
For case No. 7, $B'$ has only one possible packing and both have $-K^3>0.12$ and $r_{\text{max}}=11$. Hence case No. 7 must be excluded. 
Hence all possible $B_X$ appear as packings of  $B'$ in No. 13--16, which are the followings: 
\begin{align*}
{}&\{(7,17)\},\\
{}&\{(5,11),2\times(1,3)\},\\
{}&\{(1,2),(5,12), (1,3)\}, \\
{}&\{2\times(1,2),(5,13)\}, \\
{}&\{3\times(1,2),(4,11)\}.
\end{align*}

If $B_X=\{(7,17)\}$, then $-K_X^3=2/17$, $M_X=\lambda(M_X)=2$. Note that $P_{-9}>9\lambda(M_X)+1$ where $P_{-9}$ is computed by Reid's Riemann--Roch formula. Hence by Proposition \ref{criterion 1}, $\dim \overline{\varphi_{-m}(X)}>1$ for all $m\geq 9$ since $P_{-1}>0$. By assumption, $P_{-2}= 2$, we may take $m_0=2$. Then by Proposition \ref{criterion b}(1), take $m_1=9$, $\mu_0\leq 2$, we get that $\varphi_{-m}$ is birational for all $m\geq 33$. 

If $B_X=\{2\times(1,2),(5,13)\}$, then $-K_X^3=1/13$, $M_X=\lambda(M_X)=2$. Note that $P_{-10}>10\lambda(M_X)+1$ where $P_{-10}$ is computed by Reid's Riemann--Roch formula. Hence by Proposition \ref{criterion 1}, $\dim \overline{\varphi_{-m}(X)}>1$ for all $m\geq 10$ since $P_{-1}>0$. By assumption, $P_{-2}= 2$, we may take $m_0=2$. Then by Proposition \ref{criterion b}(1), take $m_1=10$, $\mu_0\leq 2$, we get that $\varphi_{-m}$ is birational for all $m\geq 36$. 

For the remaining 3 baskets, we may apply Lemma \ref{lem 12}(2) since $r_{\text{max}}\leq 12$.
\medskip

{\bf Subcase 2-v.} $(P_{-2}, P_{-3})=(1,3).$

In this case, 
 \begin{numcases}{}
n_{1,2}^0=0;\notag\\
n_{1,3}^0=9-P_{-4};\notag\\
n_{1,4}^0=-3+P_{-4}-\sigma_5.\notag
\end{numcases}
By inequality \eqref{424}, $P_{-4}\leq 5$. Hence $(P_{-4},\sigma_5)=(5,0)$, $(5,1)$, $(5,2)$, $(4,0)$, $(4,1)$, $(3,0)$.
Hence the corresponding $B^{(0)}$ is in the following list:
{\scriptsize
\begin{longtable}{c|L}
\caption{}\label{tab23}\\
No. &B^{(0)} \\
\hline
\endfirsthead
1 &\{4\times (1,3), 2\times (1,4)\}\\
2& \{4\times(1,3),(1,4), (1,s)\}, s\geq 5\\
3&\{4\times(1,3),(1,s_1), (1,s_2)\}, 5\leq s_1\leq s_2\\
4&\{5\times(1,3),(1,4)\}\\
5&\{5\times(1,3),(1,s)\}, s\geq 5\\
6&\{6\times(1,3)\}\\
\hline
\end{longtable} 
}

Clearly cases No. 5 and 6  are absurd since they admit no packings with $r_{\max}\geq 11$. 

For case No. 2, $B_X$ with $\gamma\geq 0$ and $r_{\max}\geq 11$ is dominated by $B'=\{ (1,3), (4,13), (1,s)\}$ for some $ 5\leq s\leq 8$. But then $-K_X^3\geq -K^3(B')>0.21$, a contradiction.

For case No. 3,  $B_X$ with $\gamma\geq 0$ and $r_{\max}\geq 11$ is $B_X=\{4\times(1,3),(2,r_1)\}$ for $r_1=11, 13$. But then $-K_X^3>0.21$, a contradiction.

For  cases No. 1 and 4,
we may get all possible packings with $\gamma\geq 0$ and $r_{\max}\geq 11$, 
which are listed in Table \ref{tab24}. In Table \ref{tab24},
for each basket $B_X$, if $r_X\leq 287$ and $r_{\max}\leq 12$, then we apply Lemma \ref{lem 12}(2) (such baskets are marked with {\checkmark} in the last column);  if $-K_X^3$ does not satisfy the assumption in the lemma, we mark it with {$\times$} in the last column;  otherwise we can compute $M_X$ and $\lambda(M_X)$, then find $n_1$ such that $P_{-n_1}>\lambda(M_X)n_1+1$ where $P_{-n_1}$ is computed by Reid's Riemann--Roch formula. Hence by Proposition \ref{criterion 1}, $\dim \overline{\varphi_{-m}(X)}>1$ for all $m\geq n_1$ since $P_{-1}>0$. By assumption, $P_{-3}= 3$, we may take $m_0=3$. Then by Proposition \ref{criterion b}(1), take $m_1=n_1$, $\mu_0\leq m_0$, and $\nu_0=1$, we get the integer $n_2$ such that $\varphi_{-m}$ is birational for all $m\geq n_2$. (For the value of $n_2$ with a $*$ mark, we apply Proposition \ref{criterion b}(3)).

{\scriptsize
\begin{longtable}{LLCCCCCCC}
\caption{}\label{tab24}\\
\hline
B_X & -K^3 & M_X & \lambda(M_X) & n_1 & m_0 & r_{\max} & n_2 \\
\hline
\endfirsthead
\multicolumn{3}{l}{{ {\bf \tablename\ \thetable{}} \textrm{-- continued from previous page}}}
 \\
\hline 
B_X & -K^3 & M_X & \lambda(M_X) & n_1 & m_0 & r_{\max} & n_2 \\ \hline 
\endhead

\hline \multicolumn{3}{r}{{\textrm{Continued on next page}}} \\ \hline
\endfoot

\hline \hline
\endlastfoot
\{(5,16), (1,4)\}& 3/16& 3 & 3 &  9& 3 &16 &  36\\
\{ (1,3),  (5,17)\}&10/51 &  10& 4 & 10 & 3 &17 &  39\\
\{ (1,3), (4,13), (1,4)\}&29/156 &  &  &  &  & & \times \\
\{ (4,13), (2,7)\}&18/91 &  &  &  &  & &\times  \\
\{3\times (1,3),  (3,11)\}& &  &  &  &  & &  \checkmark\\
\{2\times(1,3),(4,13)\}&4/39 & 4 & 3 & 12 & 3 & 13& 41* \\
\{(1,3),(5,16)\}& 5/48& 5 & 3 & 12 & 3 &16 & 45 \\
\{(6,19)\}&2/19 & 2 & 2 &  10& 3 &19 & 39 
\end{longtable} 
}
\medskip

{\bf Subcase 2-vi.} $(P_{-2}, P_{-3})=(2,2).$

In this case, 
 \begin{numcases}{}
n_{1,2}^0=5;\notag\\
n_{1,3}^0=4-P_{-4};\notag\\
n_{1,4}^0=-2+P_{-4}-\sigma_5.\notag
\end{numcases}
Hence $(P_{-4},\sigma_5)=(4,0)$, $(4,1)$, $(4,2)$, $(3,0)$, $(3,1)$, $(2,0)$.
Hence the corresponding $B^{(0)}$ is in the following list:
{\scriptsize
\begin{longtable}{c|L}
\caption{}\label{tab25}\\
No. &B^{(0)} \\
\hline
\endfirsthead
1&\{5\times (1,2), 2\times (1,4)\}\\
2&\{5\times (1,2), (1,4), (1,s)\}, s\geq 5\\
3&\{5\times (1,2), (1,s_1), (1,s_2)\}, 5\leq s_1\leq s_2\\
4&\{5\times(1,2),(1,3),(1,4)\}\\
5&\{5\times(1,2), (1,3), (1,s)\}, s\geq 5\\
6&\{5\times(1,2), 2\times(1,3)\}\\
\hline
\end{longtable} 
}

Clearly No.1 is absurd since there is no further packing. 

Case No. 6 is also absurd since all its packing has $$-K^3\leq -K^3(\{(7,16)\})<0.$$
Similarly, case No. 4 is also absurd since all its packing has $-K^3<0$.

For case No. 2, the only possible baskets are $B_X=B^{(0)}=\{5\times (1,2), (1,4), (1,s)\}$ for $s=11,12$ and we may apply Lemma \ref{lem 12}(2).

For case No. 3, possible $B_X$ with $\gamma\geq 0$ and $r_{\max}\geq 11$ are 
\begin{align*}
{}&\{5\times (1,2), (1,5),(1,11)\},\\
{}&\{5\times (1,2), (2,11)\},\\
{}&\{5\times (1,2), (2,13)\},\\
{}&\{5\times (1,2), (2,15)\}.
\end{align*}
For the first two cases, we may apply Lemma \ref{lem 12}(2). The last two cases are absurd since $-K_X^3$ does not satisfy the assumption in the lemma.

For case No. 5, we may get all possible packings with $\gamma\geq 0$ and $r_{\max}\geq 11$, 
which are listed in Table \ref{tab26}. In Table \ref{tab26},
for each basket $B_X$, if $r_X\leq 287$ and $r_{\max}\leq 12$, then we apply Lemma \ref{lem 12}(2) (such baskets are marked with {\checkmark} in the last column); if $-K_X^3$ does not satisfy the assumption in the lemma, we mark it with {$\times$} in the last column; otherwise we can compute $M_X$ and $\lambda(M_X)$, then find $n_1$ such that $P_{-n_1}>\lambda(M_X)n_1+1$ where $P_{-n_1}$ is computed by Reid's Riemann--Roch formula. Hence by Proposition \ref{criterion 1}, $\dim \overline{\varphi_{-m}(X)}>1$ for all $m\geq n_1$ since $P_{-1}>0$. By assumption, $P_{-2}= 2$, we may take $m_0=2$. Then by Proposition \ref{criterion b}(3), take $m_1=n_1$, $\mu_0\leq m_0$, and $\nu_0=1$, we get the integer $n_2$ such that $\varphi_{-m}$ is birational for all $m\geq n_2$. 

{\scriptsize
\begin{longtable}{LLCCCCCCC}
\caption{}\label{tab26}\\
\hline
B_X & -K^3 & M_X & \lambda(M_X) & n_1 & m_0 & r_{\max} & n_2 \\
\hline
\endfirsthead
\multicolumn{3}{l}{{ {\bf \tablename\ \thetable{}} \textrm{-- continued from previous page}}}
 \\
\hline 
B_X & -K^3 & M_X & \lambda(M_X) & n_1 & m_0 & r_{\max} & n_2 \\ \hline 
\endhead

\hline \multicolumn{3}{r}{{\textrm{Continued on next page}}} \\ \hline
\endfoot

\hline \hline
\endlastfoot
\{5\times(1,2), (1,3), (1,11)\}& &  &  &  &  & & \checkmark \\
\{4\times(1,2), (2,5), (1,11)\}& &  &  &  &  & &\checkmark  \\
\{3\times(1,2), (3,7), (1,11)\}& &  &  &  &  & & \checkmark \\
\{2\times(1,2), (4,9), (1,11)\}& &  &  &  &  & & \checkmark \\
\{(1,2), (5,11), (1,11)\}& &  &  &  &  & & \checkmark \\
\{(6,13),(1,11)\}& 20/143&  &  &  &  & & \times \\
\{5\times(1,2), (1,3), (1,12)\}& &  &  &  &  & & \checkmark \\
\{4\times(1,2), (2,5), (1,12)\}& &  &  &  &  & & \checkmark \\
\{3\times(1,2), (3,7), (1,12)\}& &  &  &  &  & & \checkmark \\
\{2\times(1,2), (4,9), (1,12)\}& &  &  &  &  & & \checkmark \\
\{5\times(1,2), (1,3), (1,13)\}&7/78 &7  &7/2  & 15 &2  &13 & 43 \\
\{4\times(1,2), (2,5), (1,13)\}&8/65 &  &  &  &  & & \times \\
\{(1,2),(5,11),(1,5)\}& &  &  &  &  & &  \checkmark\\
\{(1,2),(5,11),(1,6)\}& &  &  &  &  & & \checkmark \\
\{(1,2),(5,11),(1,7)\}& &  &  &  &  & & \checkmark \\
\{(1,2),(5,11),(1,8)\}& &  &  &  &  & & \checkmark \\
\{(1,2),(5,11),(1,9)\}& &  &  &  &  & & \checkmark \\
\{(1,2),(5,11),(1,10)\}& &  &  &  &  & & \checkmark \\
\{(6,13),(1,5)\}&2/65 &  &  &  &  & & \times \\
\{(6,13),(1,6)\}&5/78 & 5 & 3 & 15 & 2 &13 & 43 \\
\{(6,13),(1,7)\}& 8/91& 8 & 4 &16  &2  & 13& 44\\
\{(6,13),(1,8)\}& 11/104& 11 &  4&  &  2&13 &  43\\
\{(6,13),(1,9)\}&14/117 & 14 &14/3  &15  &2  &13 &  43\\
\{(6,13),(1,10)\}&17/130 &  &  &  &  & & \times 
\end{longtable} 
}
\medskip

{\bf Subcase 2-vii.} $(P_{-2}, P_{-3})=(1,2).$

In this case, 
 \begin{numcases}{}
n_{1,2}^0=1;\notag\\
n_{1,3}^0=6-P_{-4};\notag\\
n_{1,4}^0=-1+P_{-4}-\sigma_5.\notag
\end{numcases}
By inequality \eqref{424} and $P_{-4}\geq P_{-3}$, $2\leq P_{-4}\leq 5$. 
Hence $(P_{-4},\sigma_5)=(5,0)$, $(5,1)$, $(5,2)$, $(5,3)$, $(5,4)$, $(4,0)$, $(4,1)$, $(4,2)$, $(4,3)$, $(3,0)$, $(3,1)$, $(3,2)$, $(2,0)$, $(2,1)$.
Hence the corresponding $B^{(0)}$ is in the following list:
{\scriptsize
\begin{longtable}{c|L}
\caption{}\label{tab27}\\
No. & B^{(0)} \\
\hline
\endfirsthead
1&\{(1,2), (1,3), 4\times (1,4)\}\\
2&\{(1,2), (1,3), 3\times (1,4), (1,s)\}, s\geq 5\\
3&\{(1,2), (1,3), 2\times (1,4), (1,s_1), (1,s_2)\}, 5\leq s_1\leq s_2\\
4&\{(1,2), (1,3), (1,4), (1,s_1), (1,s_2), (1,s_3)\}, 5\leq s_1\leq s_2\leq s_3\\
5&\{(1,2), (1,3), (1,s_1), (1,s_2), (1,s_3), (1,s_4)\}, 5\leq s_1\leq s_2\leq s_3\\
6&\{(1,2), 2\times(1,3), 3\times (1,4)\}\\
7&\{(1,2), 2\times(1,3), 2\times (1,4), (1,s)\}, s\geq 5\\
8&\{(1,2), 2\times(1,3),  (1,4), (1,s_1), (1,s_2)\}, 5\leq s_1\leq s_2\\
9&\{(1,2), 2\times(1,3),  (1,s_1), (1,s_2), (1,s_3)\}, 5\leq s_1\leq s_2\leq s_3\\
10&\{(1,2), 3\times(1,3), 2\times (1,4)\}\\
11&\{(1,2), 3\times(1,3),  (1,4), (1,s)\}, s\geq 5\\
12&\{(1,2), 3\times(1,3),  (1,s_1), (1,s_2)\}, 5\leq s_1\leq s_2\\
13&\{(1,2), 4\times(1,3),  (1,4)\}\\
14&\{(1,2), 4\times(1,3), (1,s)\}, s\geq 5\\
\hline
\end{longtable} 
}

For cases No. 2--5 and 9, it is easy to compute that $$-K_X^3\geq -K^3(B^{(0)})>0.21,$$ which is absurd.

For case No. 7, all possible packings $B_X$ with $\gamma\geq 0$ and $r_{\max}\geq 11$ are $\{(1,2), (1,3), (3,11) ,(1,s)\}$ or $\{(2,5), (3,11) ,(1,s)\}$ for some $5\leq s\leq 9$. But in this case, $r_{\max}=11$ and $-K^3>0.12$, which is absurd.

For case No. 10, all possible packings $B_X$ with $r_{\max}\geq 11$ are the following:
\begin{align*}
{}&\{(4,11), 2\times (1,4)\},\\
{}&\{(1,2), (5,17)\},\\
{}&\{(1,2), (4,13), (1,4)\},\\
{}&\{(1,2), 2\times(1,3),  (3,11)\}, \\
{}&\{(2,5), (1,3),  (3,11)\},\\
{}&\{(3,8),  (3,11)\}. 
\end{align*}
The second and third baskets have $-K^3<1/30$, which is absurd. For other baskets we may apply Lemma \ref{lem 12}(2).

For case No. 12, all possible packings $B_X$ with $\gamma\geq 0$ and $r_{\max}\geq 11$ are the following:
\begin{align*}
{}&\{(4,11),  (1,s_1), (1,s_2)\}, 5\leq s_1\leq s_2\leq 8;\\
{}&\{(1,2), 3\times(1,3),  (2,r_1)\}, r_1=11,13; \\
{}&\{(2,5), 2\times(1,3),  (2,r_1)\}, r_1=11,13; \\
{}&\{(3,8), (1,3),  (2,r_1)\}, r_1=11,13; \\
{}&\{(4,11),   (2,r_1)\}, r_1=11,13.
\end{align*}
Since $r_{\max}\leq 13$ and $-K^3>0.12$, all above baskets should be excluded.

For case No. 13, all packings dominate either $B_{\min}=\{(5,14), (1,4)\}$ or $B_{\min}=\{(1,2), (5,16)\}$. But then $-K^3_X\leq -K^3(B_{\min})<0$, which is absurd.

{}Finally, for cases No. 1, 6, 8, 11, 14, we may get all possible packings with $\gamma\geq 0$ and $r_{\max}\geq 11$, 
which are listed in Table \ref{tab28}. In Table \ref{tab28},
for each basket $B_X$, if $r_X\leq 287$ and $r_{\max}\leq 12$, then we apply Lemma \ref{lem 12}(2) (such baskets are marked with {\checkmark} in the last column);  if $-K_X^3$ does not satisfy the assumption in the lemma, we mark it with {$\times$} in the last column; otherwise we can compute $M_X$ and $\lambda(M_X)$, then find $n_1$ such that $P_{-n_1}>\lambda(M_X)n_1+1$ where $P_{-n_1}$ is computed by Reid's Riemann--Roch formula. Hence by Proposition \ref{criterion 1}, $\dim \overline{\varphi_{-m}(X)}>1$ for all $m\geq n_1$ since $P_{-1}>0$. By assumption, $P_{-3}= 2$, we may take $m_0=3$. Then by Proposition \ref{criterion b}(1), take $m_1=n_1$, $\mu_0\leq m_0$, and $\nu_0=1$, we get the integer $n_2$ such that $\varphi_{-m}$ is birational for all $m\geq n_2$. (For the value of $n_2$ with a $*$ mark, we apply Proposition \ref{criterion b}(3)).

{\scriptsize
\begin{longtable}{LLCCCCCCC}
\caption{}\label{tab28}\\
\hline\hline
B_X & -K^3 & M_X & \lambda(M_X) & n_1 & m_0 & r_{\max} & n_2 \\
\hline
\endfirsthead
\multicolumn{3}{l}{{ {\bf \tablename\ \thetable{}} \textrm{-- continued from previous page}}}
 \\
\hline 
B_X & -K^3 & M_X & \lambda(M_X) & n_1 & m_0 & r_{\max} & n_2 \\ \hline 
\endhead

\hline \multicolumn{3}{r}{{\textrm{Continued on next page}}} \\ \hline
\endfoot

\hline \hline
\endlastfoot
\{(1,2), (3,11), 2\times (1,4)\}& &  &  &  &  & & \checkmark \\
\{(1,2), (4,15), (1,4)\}& 11/60&11  & 4 & 11 & 3 &15 &42  \\
\{(1,2), (5,19)\}& 7/38&  7&  7/2& 10 & 3 &19 &  39\\
\hline
\{(1,2), (5,18)\}&1/9 & 2 &2  & 10 & 3 &18 & 39 \\
\{(1,2), (1,3),  (4,15)\}& 1/10&3  &  3& 12 &3  & 15& 45 \\
\{(2,5),  (4,15)\}& 2/15& 2 &2  &  9& 3 & 15& 36 \\
\{(1,2), (1,3), (3,11) ,(1,4)\}& &  &  &  &  & &\checkmark  \\
\{(2,5), (3,11) ,(1,4)\}& &  &  &  &  & & \checkmark \\
\{(1,2),(3,11) ,(2,7)\}& &  &  &  &  & & \checkmark \\
\hline
\{(1,2), (3,11), 2\times(1,5)\}& &  &  &  &  & & \checkmark \\
\{(1,2), 2\times(1,3),  (3,14)\}&4/21 & 8 &4  & 11 & 3 &14 & 42 \\
\{(1,2), (3,11), (1,5), (1,6)\}&52/165 &  &  &  &  & & \times \\
\{(1,2), 2\times(1,3),  (1,4), (2,11)\}& &  &  &  &  & & \checkmark \\
\{(2,5), (1,3),  (1,4), (2,11)\}&167/660 &  &  &  &  & & \times \\
\{(3,8),  (1,4), (2,11)\}& &  &  &  &  & & \checkmark \\
\{(1,2), (1,3),  (2,7), (2,11)\}&107/462 &  &  &  &  & & \times \\
\{(1,2), (3,11), (2,11)\}& &  &  &  &  & &  \checkmark\\
\{(2,5), (2,7), (2,11)\}& 102/385&  &  &  &  & & \times \\
\{(1,2), 2\times(1,3),  (1,4), (2,13)\}&43/156 &  &  &  &  & & \times \\
\{(1,2),(1,3), (2,7), (2,13)\}&157/546 &  &  &  &  & & \times \\
\hline
\{(4,11), (1,4), (1,5)\}& &  &  &  &  & &  \checkmark\\
\{(4,11), (2,9)\}& &  &  &  &  & & \checkmark \\
\{(1,2), (4,13), (1,5)\}& 9/130& 9 & 4 & 18 &  3& 13& 47* \\
\{(4,11), (1,4), (1,6)\}& &  &  &  &  & & \checkmark \\
\{(1,2), (4,13), (1,6)\}& 4/39&  8& 4 & 15 &  3&13 & 44* \\
\{(4,11), (1,4), (1,7)\}& 47/308&  &  &  &  & &  \times\\
\{(1,2), (4,13), (1,7)\}& 23/182&  &  &  &  & & \times \\
\{(4,11), (1,4), (1,8)\}& &  &  &  &  & &  \checkmark\\
\{(1,2), (4,13), (1,8)\}& 15/104&  &  &  &  & & \times \\
\{(4,11), (1,4), (1,9)\}&73/396 &  &  &  &  & & \times \\
\{(1,2), (4,13), (1,9)\}&37/234 &  &  &  &  & &  \times\\
\hline
\{(5,14),  (1,5)\}& 1/70&  &  &  &  & & \times \\
\{(4,11),(1,3),  (1,5)\}& &  &  &  &  & &  \checkmark\\ 
\{(5,14),  (1,6)\}& 1/21& 2 &2  &14  & 3 &14 & 45* \\
\{(4,11),(1,3),  (1,6)\}& &  &  &  &  & &\checkmark  \\
\{(5,14),  (1,7)\}& 1/14&1  & 1 &8  & 3 & 14&  33\\
\{(4,11),(1,3),  (1,7)\}& &  &  &  &  & & \checkmark \\
\{(5,14),  (1,8)\}&5/56 &  5&  3&14  &3  &14 & 45* \\
\{(4,11),(1,3),  (1,8)\}& &  &  &  &  & & \checkmark \\
\{(5,14),  (1,9)\}& 13/126& 13 & 13/3 & 16 & 3 &14 &  47*\\
\{(4,11),(1,3),  (1,9)\}& &  &  &  &  & & \checkmark \\
\{(5,14),  (1,10)\}&4/35 & 8 &  4& 14 & 3 &14 &  45*\\
\{(4,11),(1,3),  (1,10)\}& 37/330& 37 & 8 &20  & 3 &11 &  45*\\  
\{(1,2), 4\times(1,3), (1,11)\}& &  &  &  &  & & \checkmark \\
\{(2,5), 3\times(1,3), (1,11)\}& &  &  &  &  & &  \checkmark
\end{longtable} 
}
\medskip

{\bf Subcase 2-viii.} $(P_{-2}, P_{-3})=(1,1).$

In this case, 
 \begin{numcases}{}
n_{1,2}^0=2;\notag\\
n_{1,3}^0=3-P_{-4};\notag\\
n_{1,4}^0=1+P_{-4}-\sigma_5.\notag
\end{numcases}
Hence $(P_{-4},\sigma_5)=(3,0)$, $(3,1)$, $(3,2)$, $(3,3)$, $(3,4)$, $(2,0)$, $(2,1)$, $(2,2)$, $(2,3)$, $(1,0)$, $(1,1)$, $(1,2)$. Hence the corresponding $B^{(0)}$ is in the following list:
{\scriptsize
\begin{longtable}{L}
\caption{}\label{tab29}\\
B^{(0)} \\
\hline
\endfirsthead
\{2\times (1,2), 4\times (1,4)\}\\
\{2\times (1,2),  3\times (1,4), (1,s)\}, s\geq 5\\
\{2\times (1,2),  2\times (1,4), (1,s_1), (1,s_2)\}, 5\leq s_1\leq s_2\\
\{2\times (1,2),  (1,4), (1,s_1), (1,s_2), (1,s_3)\}, 5\leq s_1\leq s_2\leq s_3\\
\{2\times (1,2),  (1,s_1), (1,s_2), (1,s_3), (1,s_4)\}, 5\leq s_1\leq s_2\leq s_3\\
\{2\times (1,2), (1,3), 3\times (1,4)\}\\
\{2\times(1,2), (1,3), 2\times (1,4), (1,s)\}, s\geq 5\\
\{2\times(1,2), (1,3),  (1,4), (1,s_1), (1,s_2)\}, 5\leq s_1\leq s_2\\
\{2\times(1,2), (1,3),  (1,s_1), (1,s_2), (1,s_3)\}, 5\leq s_1\leq s_2\leq s_3\\
\{2\times(1,2), 2\times(1,3), 2\times (1,4)\}\\
\{2\times(1,2), 2\times(1,3),  (1,4), (1,s)\}, s\geq 5\\
\{2\times(1,2), 2\times(1,3),  (1,s_1), (1,s_2)\}, 5\leq s_1\leq s_2\\
\hline\end{longtable} 
}

Hence we may get all possible packings with $\gamma\geq 0$ and $r_{\max}\geq 11$, 
which are listed in Table \ref{tab30}. In Table \ref{tab30},
for each basket $B_X$, if $r_X\leq 287$ and $r_{\max}\leq 12$, then we apply Lemma \ref{lem 12}(2) (such baskets are marked with {\checkmark} in the last column); if $-K_X^3$ does not satisfy the assumption in the lemma, we mark it with {$\times$} in the last column; otherwise we can compute $M_X$ and $\lambda(M_X)$, then find $n_1$ such that $P_{-n_1}>\lambda(M_X)n_1+1$ where $P_{-n_1}$ is computed by Reid's Riemann--Roch formula. Hence by Proposition \ref{criterion 1}, $\dim \overline{\varphi_{-m}(X)}>1$ for all $m\geq n_1$ since $P_{-1}>0$. Again by Reid's Riemann--Roch formula, we may take $m_0$ such that $P_{-m_0}\geq 2$. Then by Proposition \ref{criterion b}(3), take $m_1=n_1$, $\mu_0\leq m_0$, and $\nu_0=1$, we get the integer $n_2$ such that $\varphi_{-m}$ is birational for all $m\geq n_2$. (For the value of $n_2$ with a $*$ mark, we apply Proposition \ref{criterion b}(1)).

{\scriptsize
\begin{longtable}{LLCCCCCCC}
\caption{}\label{tab30}\\
\hline
B_X & -K^3 & M_X & \lambda(M_X) & n_1 & m_0 & r_{\max} & n_2 \\
\hline
\endfirsthead
\multicolumn{3}{l}{{ {\bf \tablename\ \thetable{}} \textrm{-- continued from previous page}}}
 \\
\hline 
B_X & -K^3 & M_X & \lambda(M_X) & n_1 & m_0 & r_{\max} & n_2 \\ \hline 
\endhead

\hline \multicolumn{3}{r}{{\textrm{Continued on next page}}} \\ \hline
\endfoot

\hline \hline
\endlastfoot
\{2\times (1,2),   (1,4), (3,13)\}& 3/52 &3  &3  & 17 & 4 &13 & 47 \\
\{2\times (1,2),   (4,17)\}& 1/17& 2 &  2&  13& 4 & 17&  51\\
\hline
\{2\times (1,2),   (1,4), (3,14)\}& 3/28& 3 & 3 &12  &4  &14 & 44 \\
\{2\times (1,2),   (3,13), (1,5)\}&7/65 & 14 &  14/3& 16 &  4& 13& 46 \\
\{2\times (1,2),   (3,13), (1,6)\}&11/78 &  &  &  &  & &  \times\\
\{2\times (1,2),  2\times (1,4), (2,11)\}& &  &  &  &  & &  \checkmark\\
\{2\times (1,2),   (3,13), (1,7)\}& 15/91&  &  &  &  & & \times \\
\{2\times (1,2),   (3,13), (1,8)\}&19/104 &  &  &  &  & & \times \\
\{2\times (1,2),  2\times (1,4), (2,13)\}&5/26 &  &  &  &  & &  \times\\
\hline
\{2\times (1,2),  (3,14), (1,5)\}&11/70 & 11 & 4 & 12 &4  & 14&  44\\
\{2\times (1,2),  (4,19)\}& 3/19&  6& 3 &  10 & 4 &19 & 42* \\
\{2\times (1,2),  (3,14), (1,6)\}&4/21 & 8 & 4 & 11 & 4 &14 & 43 \\
\{2\times (1,2),  (1,4), (1,5), (2,11)\}& &  &  &  &  & &\checkmark  \\
\{2\times (1,2),  (2,9), (2,11)\}& &  &  &  &  & &\checkmark  \\
\{2\times (1,2),  (1,4), (3,16)\}& 3/16&  3& 3 & 10 & 4 &16 &  42*\\
\{2\times (1,2),  (3,14), (1,7)\}&3/14 &  &  &  &  & &  \times\\
\{2\times (1,2),  (1,4), (2,11),(1,6)\}& &  &  &  &  & & \checkmark \\
\{2\times (1,2),  (1,4), (3,17)\}& 15/68&  &  &  &  & & \times \\
\hline
\{2\times (1,2),  (1,5), (3,16)\}& 19/80&  &  &  &  & & \times \\
\{2\times (1,2),  (4,21)\}&5/21 &  &  &  &  & &\times  \\
\hline
\{2\times (1,2), (3,11),  (1,4)\}& &  &  &  &  & &  \checkmark\\
\{2\times (1,2), (4,15)\}& <0&  &  &  &  & &\times  \\
\hline
\{2\times (1,2), (3,11),  (1,5)\}& &  &  &  &  & & \checkmark \\
\{2\times(1,2), (1,3),  (3,13)\}& <0&  &  &  &  & &\times  \\
\{(1,2), (2,5),  (3,13)\}&1/130 &  &  &  &  & &  \times\\
\{(3,7),  (3,13)\}&2/91 &  &  &  &  & &\times  \\
\{2\times (1,2), (3,11),  (1,6)\}& &  &  &  &  & &  \checkmark\\
\{2\times (1,2), (3,11),  (1,7)\}& &  &  &  &  & &  \checkmark\\
\{2\times (1,2), (3,11),  (1,8)\}& &  &  &  &  & &  \checkmark\\
\{2\times (1,2), (3,11),  (1,9)\}& &  &  &  &  & & \checkmark \\
\{2\times (1,2), (3,11),  (1,10)\}& &  &  &  &  & &  \checkmark\\
\hline
\{2\times(1,2), (1,3),  (3,14)\}& 1/42&  &  &  &  & & \times \\
\{(1,2), (2,5),  (3,14)\}&2/35 &  4& 3 & 17 &  4&14 & 49 \\
\{(3,7),  (3,14)\}& 1/14& 1 & 1 &  9& 4 &14 & 39* \\
\{2\times(1,2), (1,3),  (1,4), (2,11)\}& &  &  &  &  & & \checkmark \\
\{2\times(1,2), (2,7), (2,11)\}& &  &  &  &  & &\checkmark  \\
\{(1,2), (2,5),  (1,4), (2,11)\}& &  &  &  &  & & \checkmark \\
\{(3,7),  (1,4), (2,11)\}&31/308 &  31&31/4  &21  & 4 &11 &  47\\
\{2\times(1,2), (1,3),  (1,4), (2,13)\}& 17/156& 17 & 17/3 & 17 &  4&13 & 47 \\
\{2\times(1,2), (2,7), (2,13)\}& 11/91&  &  &  &  & &  \times\\
\{(1,2), (2,5),  (1,4), (2,13)\}& 37/260&  &  &  &  & & \times \\
\{(3,7),  (1,4), (2,13)\}& 57/364&  &  &  &  & &\times  \\
\hline
\{2\times(1,2), (1,3),  (1,5), (2,11)\}& 17/165& 34& 8 &  21& 2 & 11& 47 \\
\{(1,2), (2,5),  (1,5), (2,11)\}& &  &  &  &  & &\checkmark  \\
\{(3,7),  (1,5), (2,11)\}& 58/385&  &  &  &  & &  \times\\
\{2\times(1,2), (1,3),  (3,16)\}& 5/48&5  &  3& 13 & 4 &16 & 49 \\
\{(1,2), (2,5),  (3,16)\}& 11/80&  11& 4 & 13 &4  &16 &  49\\
\{(3,7),  (3,16)\}& 17/112& 17 & 3 & 15 & 4 & 16&  51\\
\{2\times(1,2), (1,3),  (2,11), (1,6)\}& &  &  &  &  & & \checkmark \\
\{(1,2), (2,5),  (2,11), (1,6)\}&28/165 &  &  &  &  & & \times \\
\{(3,7),  (2,11), (1,6)\}&85/462 &  &  &  &  & & \times \\
\{2\times(1,2), (1,3),  (3,17)\}&7/51 & 14 &14/3  & 14 &4  &17 & 52? \\
\{(1,2), (2,5),  (3,17)\}& 29/170&  29& 29/4 & 16 & 4 &17 & 54? \\
\{(3,7),  (3,17)\}& 22/119&  22&6  & 14 &4  &17 & 52? \\
\{2\times(1,2), (1,3),  (2,11), (1,7)\}&37/231 &  &  &  &  & &\times  \\
\{2\times(1,2), (1,3),  (1,5), (2,13)\}&31/195 &  &  &  &  & &\times  \\
\hline
\{2\times(1,2), (1,3),  (3,11)\}& &  &  &  &  & &  \checkmark\\
\{(1,2), (2,5),  (3,11)\}& &  &  &  &  & & \checkmark \\
\{(3,7),  (3,11)\}& &  &  &  &  & & \checkmark \\
\hline
\{2\times(1,2), 2\times(1,3),  (1,4), (1,11)\}& &  &  &  &  & & \checkmark \\
\{(1,2), (2,5), (1,3),  (1,4), (1,11)\}& 17/660&  &  &  &  & & \times \\
\{2\times(1,2), 2\times(1,3),  (1,4), (1,12)\}& &  &  &  &  & & \checkmark \\
\hline
\{2\times(1,2), 2\times(1,3),  (2,11)\}& &  &  &  &  & & \checkmark \\
\{(1,2), (2,5), (1,3),  (2,11)\}&1/330 &  &  &  &  & &\times  \\
\{(3,7), (1,3),  (2,11)\}& &  &  &  &  & & \checkmark \\
\{(1,2), (3,8),  (2,11)\}& &  &  &  &  & & \checkmark \\
\{ 2\times(2,5), (2,11)\}& &  &  &  &  & & \checkmark \\
\{2\times(1,2), 2\times(1,3),  (2,13)\}&1/39 &  &  &  &  & &  \times\\
\{(1,2), (2,5), (1,3),  (2,13)\}& 23/390& 23 &  6& 24 &  5&13 &  55?\\
\{(3,7), (1,3),  (2,13)\}& 20/273& 20 & 6 & 22 & 5 & 13& 53 \\
\{(1,2), (3,8),  (2,13)\}& 7/104& 7 & 7/2 & 17 & 5 & 13& 48 \\
\{ 2\times(2,5), (2,13)\}& 6/65&  6& 3 & 14 & 5 &13 & 45 \\
\{2\times(1,2), 2\times(1,3),  (2,15)\}&1/15 & 2 & 2 & 14 & 6 & 15& 50 \\
\{(1,2), (2,5), (1,3),  (2,15)\}&1/10 & 3 & 3 & 14 &5  &15 & 49 
\end{longtable} 
}

There are 4 cases with ``?" marked in the $n_2$ column of Table \ref{tab30}, where the values of  $n_2$ are larger than what we expect. So we discuss them in details in the following.

If $B_X$ is among $\{2\times(1,2), (1,3),  (3,17)\}$, $\{(1,2), (2,5),  (3,17)\}$,  and $\{(3,7),  (3,17)\}$, we know that $\dim \overline{\varphi_{-m}(X)}>1$ for all $m\geq 16$ from the list. Note that $P_{-7}\geq 11$. Take $m_0=4$. If $|-7K_X|$ and $|-4K_X|$ are not composed with the same pencil, then we may take $m_1=7$ and $\mu_0\leq 4$, and by Proposition \ref{criterion b}(1), $\varphi_{-m}$ is birational for all $m\geq 33$; if $|-7K_X|$ and $|-4K_X|$ are composed with the same pencil, then we may take $m_1=16$ and $\mu_0\leq \frac{7}{10}$ by Remark \ref{5.3}, and by Proposition \ref{criterion b}(3), $\varphi_{-m}$ is birational for all $m\geq 50$.

If $B_X=\{(1,2), (2,5), (1,3),  (2,13)\}$, we know that $\dim \overline{\varphi_{-m}(X)}>1$ for all $m\geq 24$ from the list. Note that $P_{-10}\geq 13$. Take $m_0=5$. If $|-10K_X|$ and $|-5K_X|$ are not composed with the same pencil, then we may take $m_1=10$ and $\mu_0\leq 5$, and by Proposition \ref{criterion b}(3), $\varphi_{-m}$ is birational for all $m\geq 41$; if $|-10K_X|$ and $|-5K_X|$ are composed with the same pencil, then we may take $m_1=24$ and $\mu_0\leq \frac{10}{12}$ by Remark \ref{5.3}, and by Proposition \ref{criterion b}(3), $\varphi_{-m}$ is birational for all $m\geq 50$.

Combining all above cases, the proof is completed.
\end{proof}

\begin{thm}\label{thm P1>0 1/30<}
Let $(X, Y, Z)$ be a Fano--Mori triple such that $\rho(Y)>1$. Assume that $P_{-1}>0$ and $1/30\leq -K_X^3< 0.21$.
Then
$\varphi_{-m}$ is birational for all $m\geq 51$.
\end{thm}
\begin{proof}
If $r_X=840$, then we are done by Theorem \ref{thm r840}(2). Hence we may always assume that $r_X\leq 660$. Since $P_{-1}>0$, we may always take $\nu_0=1$.

We discuss by the value of $r_{\max}$. 
\medskip

{\bf Case 1.} $r_{\max}\leq 8$.

Recall that $P_{-8}\geq 2$ by Proposition \ref{facts}. We may take $m_0=8$.
Note that by Lemma \ref{lem 30 1}(1), $P_{-18}\geq 21$.
If $|-18K_X|$ and $|-8K_X|$ are not composed with the same pencil, we may take $m_1=18$ and $\mu_0\leq 8$. By Proposition \ref{criterion b}(3), $\varphi_{-m}$ is birational for all $m\geq 42$.
 If $|-18K_X|$ and $|-8K_X|$ are composed with the same pencil, we may take $m_0=35$ by Lemma \ref{lem 30 1}(2) and $\mu_0\leq \frac{18}{20}$ by Remark \ref{5.3}. By Proposition \ref{criterion b}(3), $\varphi_{-m}$ is birational for all $m\geq 51$.
\medskip

{\bf Case 2}. $r_{\max}=9$. 

In this case, by inequality \eqref{kwmt}, arguing as \cite[Proof of Proposition 2.4]{CJ16}, one can show that either $r_X\leq 504=9\times 8\times 7$ or $r_X=630$, moreover, in the latter case, the set of local indices is either $\{2,5,7,9\}$ or $\{2,2,5,7,9\}$. (We leave this to interested readers as an exercise!)

If $r_X\leq 504$, then by Lemma \ref{lem l/d},
$$\frac{\lambda(M_X)}{-K_X^3}\leq \max\left\{1, \frac{3}{1/30}, \sqrt{\frac{2\times 504}{1/30}}\right\}<174.$$
Take $t=11$, then
$$-\frac{3}{4}+\sqrt{\frac{12}{t(-K_X^3)}+\frac{6\lambda(M_X)}{-K_X^3}+\frac{1}{16}}< 33.$$
By Proposition \ref{criterion 2}, $|-mK_X|$ is not composed with a pencil for $m\geq 33.$
Recall that $P_{-8}\geq 2$ by Proposition \ref{facts}. We may take $m_0=8$.
Note that by Lemma \ref{lem 30 1}(1), $P_{-18}\geq 21$.
If $|-18K_X|$ and $|-8K_X|$ are not composed with the same pencil, we may take $m_1=18$ and $\mu_0\leq 8$. By Proposition \ref{criterion b}(3), $\varphi_{-m}$ is birational for all $m\geq 44$.
 If $|-18K_X|$ and $|-8K_X|$ are composed with the same pencil, we may take $m_1=33$ by Lemma \ref{lem 30 1}(2) and $\mu_0\leq \frac{18}{20}$ by Remark \ref{5.3}. By Proposition \ref{criterion b}(3), $\varphi_{-m}$ is birational for all $m\geq 51$.

Now we consider the case $r_X=630$. In this case, $B_X$ is either $\{(1,2),(a, 5),(b,7),(c,9)\}$ or $\{2\times(1,2),(a, 5),(b,7),(c,9)\}$ for some $a\in\{1,2\}$, $b\in\{1,2,3\}$, and $c\in\{1,2,4\}$. Note that by Theorem \ref{thm >=0.21}(1), we may assume that $1/30\leq -K_X^3< 0.12$.

First we consider $B_X=\{(1,2),(a, 5),(b,7),(c,9)\}$. If $P_{-1}\geq 2$, then by the equality \eqref{P1K3}, 
\begin{align*}
-K_X^3{}&=2P_{-1}+\frac{1}{2}+\frac{a(5-a)}{5}+\frac{b(7-b)}{7}+\frac{c(9-c)}{9}-6\\
{}&\geq \frac{1}{2}+\frac{4}{5}+\frac{6}{7}+\frac{8}{9}-2>1,
\end{align*}
a contradiction. Hence $P_{-1}=1$ and
\begin{align*}
-K_X^3{}&=\frac{1}{2}+\frac{a(5-a)}{5}+\frac{b(7-b)}{7}+\frac{c(9-c)}{9}-4.
\end{align*}
It is easy to check that $1/30\leq -K_X^3< 0.12$ if and only if $(a,b,c)=(2,1,2)$, that is, $B_X=\{(1,2),(2, 5),(1,7),(2,9)\}$, $-K_X^3= 71/630$, $M_X=71$, and $\lambda(M_X)=71/5$. Note that $P_{-25}>25\lambda(M_X)+1$ where $P_{-25}$ is computed by Reid's Riemann--Roch formula. Hence by Proposition \ref{criterion 1}, $\dim \overline{\varphi_{-m}(X)}>1$ for all $m\geq 25$ since $P_{-1}>0$. Again by Reid's Riemann--Roch formula, we may take $m_0=4$ since $P_{-4}= 2$. Then by Proposition \ref{criterion b}(3), take $m_1=25$, $\mu_0\leq 4$, and $\nu_0=1$, we get that $\varphi_{-m}$ is birational for all $m\geq 47$.

Then we consider $B_X=\{2\times(1,2),(a, 5),(b,7),(c,9)\}$. If $P_{-1}\geq 2$, then by the equality \eqref{P1K3}, 
\begin{align*}
-K_X^3{}&=2P_{-1}+2\times\frac{1}{2}+\frac{a(5-a)}{5}+\frac{b(7-b)}{7}+\frac{c(9-c)}{9}-6\\
{}&\geq 2\times\frac{1}{2}+\frac{4}{5}+\frac{6}{7}+\frac{8}{9}-2>1,
\end{align*}
a contradiction. Hence $P_{-1}=1$ and
\begin{align*}
-K_X^3{}&=1+\frac{a(5-a)}{5}+\frac{b(7-b)}{7}+\frac{c(9-c)}{9}-4.
\end{align*}
It is easy to check that $1/30\leq -K_X^3< 0.12$ if and only if $(a,b,c)=(1,2,1)$, that is, $B_X=\{2\times (1,2),(1, 5),(2,7),(1,9)\}$, $-K_X^3= 37/315$, $M_X=74$, and $\lambda(M_X)=12$. Note that $P_{-25}>25\lambda(M_X)+1$ where $P_{-25}$ is computed by Reid's Riemann--Roch formula. Hence by Proposition \ref{criterion 1}, $\dim \overline{\varphi_{-m}(X)}>1$ for all $m\geq 25$ since $P_{-1}>0$. Again by Reid's Riemann--Roch formula, we may take $m_0=4$ since $P_{-4}= 2$. Then by Proposition \ref{criterion b}(3), take $m_1=25$, $\mu_0\leq 4$, and $\nu_0=1$, we get that $\varphi_{-m}$ is birational for all $m\geq 47$.
\medskip

{\bf Case 3}. $r_{\max}=10$.

In this case, by inequality \eqref{kwmt}, arguing as \cite[Proof of Proposition 2.4]{CJ16}, one can show that either $r_X\leq 210=10\times 7\times 3$ or $r_X=420$, moreover, in the latter case, the set of local indices is $\{3,4,7,10\}$.
 (We leave this again as an exercise.)

If $r_X\leq 210$, then we may apply Lemma \ref{lem 12}(2). Hence we only consider the case $r_X=420$. Note that by Theorem \ref{thm >=0.21}(1), we may assume that $1/30\leq -K_X^3< 0.12$.

In this case, 
$B_X$ is  $\{(1,3),(1, 4),(a,7),(b,10)\}$ 
for some $a\in\{1,2,3\}$, $b\in\{1,3\}$. 
 If $P_{-1}\geq 2$, then by the equality \eqref{P1K3}, 
\begin{align*}
-K_X^3{}&=2P_{-1}+\frac{2}{3}+\frac{3}{4}+\frac{a(7-a)}{7}+\frac{b(10-b)}{10}-6\\
{}&\geq \frac{2}{3}+\frac{3}{4}+\frac{6}{7}+\frac{9}{10}-2>1,
\end{align*}
a contradiction. Hence $P_{-1}=1$ and
\begin{align*}
-K_X^3{}&=\frac{2}{3}+\frac{3}{4}+\frac{a(7-a)}{7}+\frac{b(10-b)}{10}-4.
\end{align*}
One can directly check that $1/30\leq -K_X^3< 0.12$ could never happen.
\medskip

{\bf Case 4}. $r_{\max}\geq 11$.

This case is proved by Lemma \ref{lem 13-24} and Theorem \ref{thm >=0.21}.

Combining all above discussions, the proof is completed.
\end{proof}

\subsection{The case $P_{-2}>0$ and $-K_X^3<1/30$}
\begin{thm}\label{thm <1/30}
Let $(X, Y, Z)$ be a Fano--Mori triple such that $\rho(Y)>1$. Assume that $P_{-2}>0$ and $-K_X^3<1/30$.
Then
\begin{enumerate}
\item $\dim \overline{\varphi_{-m}(X)}>1$ for all $m\geq 37$;
\item $\varphi_{-m}$ is birational for all $m\geq 52$.
\end{enumerate}
\end{thm}
\begin{proof}
We follow the argument of \cite[Proof of Theorem 4.4]{CC} to classify all possible geometric baskets with $-K^3<1/30$. 

If $P_{-1}=0$, by  \cite[Proof of Theorem 4.4, Case I]{CC}, since $P_{-2}>0$ and $-K^3<1/30$, all possible geometric baskets are dominated by $\{7\times (1,2), (3,7), (1,5)\}$, which we already treated as the last four cases in Table \ref{tab9} in the proof of Theorem \ref{thm P1=0 <0.21}. 

Now we consider $P_{-1}\geq 1$. 

By \cite[Proof of Theorem 4.4, Case IV]{CC}, if $P_{-1}\geq 3$, then $-K_X^3\geq 1/2$, a contradiction. Hence $P_{-1}=1$ or $2$.
\medskip

{\bf Case 1.} $P_{-1}=1$.  By \cite[Proof of Theorem 4.4, Subcase II-1, Subcase II-2]{CC}), one has $-K_X^3\geq 1/12$ when $P_{-2}\geq 3$.  Hence we have $P_{-2}\leq 2$. 
\medskip

{\bf Subcase 1-i.} $P_{-2}=2$. Note that by \cite[Proof of Theorem 4.4, Subcase II-3]{CC}, either 
$$
B^{(0)}=\{5\times(1,2), (1,3),(1,s)\} 
$$
for some $s\geq 7$, or $B_X$ is dominated by $\{3\times(1,2), (3,7),(1,5)\}$. 

In the former case, if $s\geq 8$, then 
$$
-K^3(B_X)\geq -K^3(B^{(0)})\geq \frac{1}{24},
$$
which is absurd. So $s=7$, and one-step packing has $-K^3>1/30$, hence $B_X=B^{(0)}=\{5\times(1,2), (1,3),(1,7)\}.$ In the latter case, all packings of $\{3\times(1,2), (3,7),(1,5)\}$ has $-K^3< 1/30$.
Hence we may get all possible packings, 
which are listed in Table \ref{tab10}.

{\scriptsize
\begin{longtable}{LLCCCCCCC}
\caption{}\label{tab10}\\
\hline
B_X & -K^3 & M_X & \lambda(M_X) & n_1 & m_0 & r_{\max} & n_2 \\
\hline
\endfirsthead
\multicolumn{3}{l}{{ {\bf \tablename\ \thetable{}} \textrm{-- continued from previous page}}}
 \\
\hline 
B_X & -K^3 & M_X & \lambda(M_X) & n_1 & m_0 & r_{\max} & n_2 \\ \hline 
\endhead

\hline \multicolumn{3}{r}{{\textrm{Continued on next page}}} \\ \hline
\endfoot

\hline \hline
\endlastfoot
\{5\times(1,2), (1,3),(1,7)\}& 1/42 &  &  &  &  & & \checkmark \\
\hline
 \{3\times(1,2), (3,7),(1,5)\}& 1/70 &  &  &  &  & & \checkmark \\
 \{2\times(1,2), (4,9),(1,5)\}& 1/45&  &  &  &  & & \checkmark \\
 \{(1,2), (5,11),(1,5)\}& 3/110 &  &  &  &  & & \checkmark \\
 \{ (6,13),(1,5)\}& 2/65 & 2 & 2 & 17 & 2 & 13 & 45 
\end{longtable} 
}

In Table \ref{tab10},
for each basket $B_X$, if $r_X\leq 165$ and $r_{\max}\leq 12$, then we apply Lemmas \ref{lem 12 pencil} and \ref{lem 12}(2) (such baskets are marked with {\checkmark} in the last column), otherwise we can compute $M_X$ and $\lambda(M_X)$, then find $n_1$ such that $P_{-n_1}>\lambda(M_X)n_1+1$ where $P_{-n_1}$ is computed by Reid's Riemann--Roch formula. Hence by Proposition \ref{criterion 1}, $\dim \overline{\varphi_{-m}(X)}>1$ for all $m\geq n_1$ since $P_{-1}>0$. By assumption, $P_{-2}= 2$, we may take $m_0=2$. Then by Proposition \ref{criterion b}(3), take $m_1=n_1$, $\mu_0\leq m_0$, and $\nu_0=1$, we get the integer $n_2$ such that $\varphi_{-m}$ is birational for all $m\geq n_2$. 
\medskip

{\bf Subcase 1-ii.} $P_{-2}=1$. By \cite[Proof of Theorem 4.4, Subcase II-4]{CC},  since $-K^3_X<1/30$, $(P_{-3}, P_{-4})$ can only take the values $(2,3)$, $(2,2)$, $(1,2)$ and $(1,1)$.

If $(P_{-3}, P_{-4})=(2,3)$, then $B_X$ is dominated by 
$$\{(1,2), 2\times(1,3), (2,7),(1,4)\}$$ by \cite[Proof of Theorem 4.4, Case II-4b]{CC}. 
Hence we may get all possible packings with $0<-K^3(B_X)<1/30$, 
which are listed in Table \ref{tab11}. In Table \ref{tab11},
for each basket $B_X$, if $r_X\leq 165$ and $r_{\max}\leq 12$, then we apply Lemmas \ref{lem 12 pencil} and \ref{lem 12}(2) (such baskets are marked with {\checkmark} in the last column), otherwise  
we can compute $M_X$ and $\lambda(M_X)$, then find $n_1$ such that $P_{-n_1}>\lambda(M_X)n_1+1$ where $P_{-n_1}$ is computed by Reid's Riemann--Roch formula. Hence by Proposition \ref{criterion 1}, $\dim \overline{\varphi_{-m}(X)}>1$ for all $m\geq n_1$ since $P_{-1}>0$. By assumption, $P_{-3}= 2$, we may take $m_0=3$. Then by Proposition \ref{criterion b}(3), take $m_1=n_1$, $\mu_0\leq m_0$, and $\nu_0=1$, we get the integer $n_2$ such that $\varphi_{-m}$ is birational for all $m\geq n_2$. (For the values of $n_2$ with a $*$ mark, we apply Proposition \ref{criterion b}(1)).

{\scriptsize
\begin{longtable}{LLCCCCCCC}
\caption{}\label{tab11}\\
\hline
B_X & -K^3 & M_X & \lambda(M_X) & n_1 & m_0 & r_{\max} & n_2 \\
\hline
\endfirsthead
\multicolumn{3}{l}{{ {\bf \tablename\ \thetable{}} \textrm{-- continued from previous page}}}
 \\
\hline 
B_X & -K^3 & M_X & \lambda(M_X) & n_1 & m_0 & r_{\max} & n_2 \\ \hline 
\endhead

\hline \multicolumn{3}{r}{{\textrm{Continued on next page}}} \\ \hline
\endfoot

\hline \hline
\endlastfoot

\{(1,2), 2\times(1,3), (2,7),(1,4)\}& 1/84 &  &  &  &  & & \checkmark \\
\{(1,2), (1,3), (3,10),(1,4)\}& 1/60 &  &  &  &  & & \checkmark \\
\{(1,2), (4,13),(1,4)\}& 1/52 & 1 & 1 & 13 & 3 & 13 & 42 \\
\{(1,2), 2\times(1,3), (3,11)\}&1/66 &  &  &  &  & & \checkmark \\
\{(1,2), (1,3), 2\times (2,7)\}& 1/42 &  &  &  &  & & \checkmark \\
\{(1,2), (3,10), (2,7)\}& 1/35 &  &  &  &  & & \checkmark \\
\{(1,2), (5,17)\}& 1/34 & 1 & 1 & 11 & 3 & 17 & 42* 
\end{longtable} 
}

If $(P_{-3}, P_{-4})=(2,2)$, then by \cite[Proof of Theorem 4.4, Case II-4d]{CC}, either 
$$
B^{(0)}=\{(1,2), 4\times(1,3),(1,s)\} 
$$
for some $s\geq 7$, or $B_X$ is dominated by $\{(3,8),2\times(1,3),(1,5)\}$. In the former case, if $s\geq 8$, then 
$$
-K^3(B_X)\geq -K^3(B^{(0)})\geq \frac{1}{24},
$$
which is absurd. So $s=7$, and since its one-step packing has $-K^3>1/30$, $B_X=B^{(0)}=\{(1,2), 4\times(1,3),(1,7)\} .$
Hence we may get all possible packings with $0<-K^3(B_X)<1/30$, 
which are listed in Table \ref{tab12}. In Table \ref{tab12},
for each basket $B_X$, if $r_X\leq 165$ and $r_{\max}\leq 12$, then we apply Lemmas \ref{lem 12 pencil} and \ref{lem 12}(2) (such baskets are marked with {\checkmark} in the last column), otherwise 
we can compute $M_X$ and $\lambda(M_X)$, then find $n_1$ such that $P_{-n_1}>\lambda(M_X)n_1+1$ where $P_{-n_1}$ is computed by Reid's Riemann--Roch formula. Hence by Proposition \ref{criterion 1}, $\dim \overline{\varphi_{-m}(X)}>1$ for all $m\geq n_1$ since $P_{-1}>0$. By assumption, $P_{-3}= 2$, we may take $m_0=3$. Then by Proposition \ref{criterion b}(3), take $m_1=n_1$, $\mu_0\leq m_0$, and $\nu_0=1$, we get the integer $n_2$ such that $\varphi_{-m}$ is birational for all $m\geq n_2$.

{\scriptsize
\begin{longtable}{LLCCCCCCC}
\caption{}\label{tab12}\\
\hline
B_X & -K^3 & M_X & \lambda(M_X) & n_1 & m_0 & r_{\max} & n_2 \\
\hline
\endfirsthead
\multicolumn{3}{l}{{ {\bf \tablename\ \thetable{}} \textrm{-- continued from previous page}}}
 \\
\hline 
B_X & -K^3 & M_X & \lambda(M_X) & n_1 & m_0 & r_{\max} & n_2 \\ \hline 
\endhead

\hline \multicolumn{3}{r}{{\textrm{Continued on next page}}} \\ \hline
\endfoot

\hline \hline
\endlastfoot

\{(1,2), 4\times(1,3),(1,7)\}&1/42 &  &  &  &  & & \checkmark \\
\hline
\{(3,8),2\times(1,3),(1,5)\}& 1/120 &  &  &  &  & & \checkmark \\
\{(4,11),(1,3),(1,5)\}& 2/165 &  &  &  &  & & \checkmark \\
\{(5,14),(1,5)\}& 1/70 & 1 & 1 & 16 & 3 & 14 & 47 
\end{longtable} 
}

If $(P_{-3}, P_{-4})=(1,2)$, then by \cite[Proof of Theorem 4.4, Case II-4e]{CC}, $1\leq \sigma_5\leq 3$.

If $\sigma_5\geq 2$, then $B^{(0)}$ is either $$
B^{(0)}=\{2\times (1,2), (1,3),(1,4),(1,s_1),(1,s_2)\} 
$$
for some $s_2\geq s_1\geq 5$ or $$
B^{(0)}=\{2\times (1,2), (1,3),(1,s_1),(1,s_2), (1,s_3)\} 
$$
for some $s_3\geq s_2\geq s_1\geq 5$. For the former case, if $s_2\geq 6$, then $$
-K^3(B_X)\geq -K^3(B^{(0)})\geq \frac{1}{20},
$$
hence $s_1=s_2=5$ and $$B^{(0)}=\{2\times (1,2), (1,3),(1,4),2\times (1,5)\}. 
$$ For the latter one, we always have 
$$
-K^3(B_X)\geq -K^3(B^{(0)})\geq \frac{1}{15}.
$$

If $\sigma_5= 1$, then by \cite[Proof of Theorem 4.4, Case II-4e]{CC}, either 
$$
B^{(0)}=\{2\times (1,2), (1,3),2\times (1,4),(1,s)\} 
$$
for some $s\geq 7$, or $B_X$ is dominated by one of the following baskets:
\begin{align*}
{}&\{2\times(1,2),(2,7), (1,4),(1,6)\},\\
{}&\{(3,7),2\times(1,4),(1,5)\},\\
{}&\{(1,2),(2,5),(1,4),(2,9)\}.
\end{align*} 
In the former case, if $s\geq 8$, then 
$$
-K^3(B_X)\geq -K^3(B^{(0)})\geq \frac{1}{24},
$$
which is absurd. So $s=7$, and since its one-step packing has $-K^3>1/30$, $B_X=B^{(0)}=\{2\times (1,2), (1,3),2\times (1,4),(1,7)\} .$
Hence we may get all possible packings with $0<-K^3(B_X)<1/30$, 
which are listed in Table \ref{tab13}. In Table \ref{tab13},
for each basket $B_X$, if $r_X\leq 165$ and $r_{\max}\leq 12$, then we apply Lemmas \ref{lem 12 pencil} and \ref{lem 12}(2)  (such baskets are marked with {\checkmark} in the last column), otherwise  
we can compute $M_X$ and $\lambda(M_X)$, then find $n_1$ such that $P_{-n_1}>\lambda(M_X)n_1+1$ where $P_{-n_1}$ is computed by Reid's Riemann--Roch formula. Hence by Proposition \ref{criterion 1}, $\dim \overline{\varphi_{-m}(X)}>1$ for all $m\geq n_1$ since $P_{-1}>0$. By assumption, $P_{-4}= 2$, we may take $m_0=4$. Then by Proposition \ref{criterion b}(3), take $m_1=n_1$, $\mu_0\leq m_0$, and $\nu_0=1$, we get the integer $n_2$ such that $\varphi_{-m}$ is birational for all $m\geq n_2$. 

{\scriptsize
\begin{longtable}{LLCCCCCCC}
\caption{}\label{tab13}\\
\hline
B_X & -K^3 & M_X & \lambda(M_X) & n_1 & m_0 & r_{\max} & n_2 \\
\hline
\endfirsthead
\multicolumn{3}{l}{{ {\bf \tablename\ \thetable{}} \textrm{-- continued from previous page}}}
 \\
\hline 
B_X & -K^3 & M_X & \lambda(M_X) & n_1 & m_0 & r_{\max} & n_2 \\ \hline 
\endhead

\hline \multicolumn{3}{r}{{\textrm{Continued on next page}}} \\ \hline
\endfoot

\hline \hline
\endlastfoot
\{2\times (1,2), (1,3),(1,4),2\times (1,5)\}& 1/60&  &  &  &  & & \checkmark \\
\{2\times (1,2), (2,7),2\times (1,5)\}& 1/35&  &  &  &  & & \checkmark \\
\{2\times (1,2), (1,3),(2,9),(1,5)\}& 1/45&  &  &  &  & & \checkmark \\
\{2\times (1,2), (1,3),(3,14)\}& 1/42& 1 & 1 & 14 & 4 & 14 & 46 \\
\hline
\{2\times (1,2), (1,3),2\times (1,4),(1,7)\}& 1/42 &  &  &  &  & & \checkmark \\
\hline
\{2\times(1,2),(2,7), (1,4),(1,6))\}&1/84&  &  &  &  & & \checkmark \\
\{2\times(1,2),(3,11),(1,6))\}&1/66 &  &  &  &  & & \checkmark \\
\hline
\{(3,7),2\times(1,4),(1,5)\}& 1/70 &  &  &  &  & & \checkmark \\
\{(1,2),(2,5),(1,4),(2,9)\}& 1/180& 1 & 1 & 26 & 4 & 9 & 48 \\
\{(3,7),(1,4),(2,9)\}& 5/252 & 5 & 3 & 28 & 4 & 9& 50 \\
\{(3,7),(3,13)\}& 2/91 & 2 & 2 & 21 & 4 & 13 & 51 \\
\{(1,2),(2,5),(3,13)\}& 1/130& 1 & 1 & 22 & 4 & 13 & 52?
\end{longtable} 
}

Note that there is one case with ``?" mark in the $n_2$ value column of Table \ref{tab13} where we get bigger $n_2$ value than we expect. So we discuss it in details. 
If $B_X=\{(1,2),(2,5),(3,13)\}$, we know that $\dim \overline{\varphi_{-m}(X)}>1$ for all $m\geq 22$ from the list. Note that $P_{-12}=7$. Take $m_0=4$. If $|-12K_X|$ and $|-4K_X|$ are not composed with the same pencil, then we may take $m_1=12$ and $\mu_0\leq 4$, and by Proposition \ref{criterion b}(3), $\varphi_{-m}$ is birational for all $m\geq 42$; if $|-12K_X|$ and $|-4K_X|$ are composed with the same pencil, then we may take $m_1=22$ and $\mu_0\leq 2$ by Remark \ref{5.3}, and by Proposition \ref{criterion b}(3), $\varphi_{-m}$ is birational for all $m\geq 50$.

If $(P_{-3}, P_{-4})=(1,1)$, then by \cite[Proof of Theorem 4.4, Case II-4f]{CC}, any $B_X$ satisfying $-K^3(B_X)<1/30$ and $\gamma(B_X)\geq 0$ is either dominated by one of the following baskets:
{\scriptsize
\begin{longtable}{L}
\caption{}\label{tab14}\\
B \\
\hline
\endfirsthead
\{2\times (1,2), (1,3),(2,7),(1,11)\}\\
\{(1,2), (2,5),(1,3),(1,4),(1,s)\}, s=9,10,11\\
\{(3,7),(1,3),(1,4),(1,8)\}\\
\{(1,2),(2,5),(2,7),(1,8)\}\\
\{2\times (2,5),(1,4),(1,s)\}, s=7,8\\
\{(1,2),(2,5),(1,3),(1,5),(1,7)\}\\
\{(3,7),(1,3),(1,5),(1,6)\}\\
\{(1,2),(3,8),(1,5),(1,6)\}\\
\{(1,2),(2,5),(1,3),(2,11)\}\\
\hline
\end{longtable} }
\noindent or its initial basket $$B^{(0)}=\{2\times(1,2),2\times(1,3),(1,s_1),(1,s_2)\}$$ for some $s_2\geq s_1\geq 5$ and $s_1+s_2\geq 13$. Note that in the latter case, by $-K^3(B^{(0)})<1/30,$ the possible values of $(s_1, s_2)$ are $(5,8)$, $(5,9)$, $(6,7)$.

Hence we may get all possible packings with $0<-K^3(B_X)<1/30$ with $\gamma\geq 0$, 
which are listed in Table \ref{tab15}. In Table \ref{tab15},
for each basket $B_X$, if $r_X\leq 165$ and $r_{\max}\leq 12$, then we apply Lemmas \ref{lem 12 pencil} and \ref{lem 12}(2)  (such baskets are marked with {\checkmark} in the last column), otherwise 
we can compute $M_X$ and $\lambda(M_X)$, then find $n_1$ such that $P_{-n_1}>\lambda(M_X)n_1+1$ where $P_{-n_1}$ is computed by Reid's Riemann--Roch formula. Hence by Proposition \ref{criterion 1}, $\dim \overline{\varphi_{-m}(X)}>1$ for all $m\geq n_1$ since $P_{-1}>0$. Again by Reid's Riemann--Roch formula, we may find $m_0$ such that $P_{-m_0}\geq 2$. Then by Proposition \ref{criterion b}(3), take $m_1=n_1$, $\mu_0\leq m_0$, and $\nu_0=1$, we get the integer $n_2$ such that $\varphi_{-m}$ is birational for all $m\geq n_2$. 

{\scriptsize
\begin{longtable}{LLCCCCCCC}
\caption{}\label{tab15}\\
\hline
B_X & -K^3 & M_X & \lambda(M_X) & n_1 & m_0 & r_{\max} & n_2 \\
\hline
\endfirsthead
\multicolumn{3}{l}{{ {\bf \tablename\ \thetable{}} \textrm{-- continued from previous page}}}
 \\
\hline 
B_X & -K^3 & M_X & \lambda(M_X) & n_1 & m_0 & r_{\max} & n_2 \\ \hline 
\endhead

\hline \multicolumn{3}{r}{{\textrm{Continued on next page}}} \\ \hline
\endfoot

\hline \hline
\endlastfoot

\{2\times (1,2), (1,3),(2,7),(1,11)\}& 1/231 & & & & & & ? \\
\{2\times (1,2), (3,10),(1,11)\}& 1/110 & & & & & & ? \\
\hline
\{(1,2), (2,5),(1,3),(1,4),(1,9)\}&1/180 & 1 & 1 & 30 & 8 & 9& 56? \\
\{(1,2), (2,5),(2,7),(1,9)\}& 11/630 & 11 & 4 & 36 & 7 & 9 & 61? \\
\{(3,7),(2,7),(1,9)\}& 2/63&  &  &  &  & & \checkmark \\
\{(3,7), (1,3),(1,4),(1,9)\}& 5/252 & 5 & 3 & 29 & 7 & 9 & 54? \\
\{(1,2), (3,8),(1,4),(1,9)\}& 1/72&  &  &  &  & & \checkmark \\
\hline
\{(1,2), (2,5),(1,3),(1,4),(1,10)\} & 1/60 &  &  &  &  & & \checkmark \\
\{(3,7),(1,3),(1,4),(1,10)\} &13/420 & 13 & 13/3 & 29 & 7 & 10 & 56? \\
\{(1,2),(3,8),(1,4),(1,10)\} & 1/40&  &  &  &  & & \checkmark \\
\{(1,2), (2,5),(2,7),(1,10)\} & 1/35&  &  &  &  & & \checkmark \\
\hline
\{(1,2), (2,5),(1,3),(1,4),(1,11)\} & 17/660 & 17 & 17/3 & 36 & 8 & 11& 66? \\
\hline
\{(3,7),(1,3),(1,4),(1,8)\}& 1/168 & 1 & 1 & 28 & 7 & 8 & 51 \\
\{(1,2),(2,5),(2,7),(1,8)\}&1/280 & 1 & 1 & 36 & 7 & 8& 59? \\
\{(3,7),(2,7),(1,8)\}& 1/56 &  &  &  &  & & \checkmark \\
\hline
\{2\times (2,5),(1,4),(1,7)\} & 1/140 &  &  &  &  & & \checkmark \\
\hline
\{2\times (2,5),(1,4),(1,8)\} & 1/40 &  &  &  &  & & \checkmark \\
\hline
\{(1,2),(2,5),(1,3),(1,5),(1,7)\}& 1/42 & 5 & 3 & 26 & 5 & 7 & 45 \\
\{(1,2),(3,8),(1,5),(1,7)\}& 9/280 & 9 & 4 & 26 & 5 & 8 & 47 \\
\hline
\{(3,7),(1,3),(1,5),(1,6)\}& 1/70 & 3 & 3 & 34 & 5 & 7 & 53? \\
\{(3,7),(1,3),(2,11)\}& 4/231& 4 & 3 & 31 & 5 & 11 & 58? \\
\{(1,2),(3,8),(1,5),(1,6)\}& 1/120&  &  &  &  & & \checkmark \\
\{(1,2),(3,8),(2,11)\}& 1/88 &  &  &  &  & & \checkmark \\
\{(1,2),(2,5),(1,3),(2,11)\}& 1/330 & 1 & 1 & 37 & 5 & 11 & 64? \\
\hline
\{2\times(1,2),2\times(1,3),(1,5),(1,8)\}&1/120 &  &  &  &  & & \checkmark \\
\hline
\{2\times(1,2),2\times(1,3),(1,5),(1,9)\}& 1/45 &  &  &  &  & & \checkmark \\
\hline
\{2\times(1,2),2\times(1,3),(1,6),(1,7)\} & 1/42 &  &  &  &  & & \checkmark \\
\{2\times(1,2),2\times(1,3),(2,13)\}&1/39 & 2 & 2 & 20 & 6 & 13& 52? 
\end{longtable} 
}
There are 12 cases with ``?" marked in the $n_2$ column of Table \ref{tab15}. The first two baskets are non-geometric since they have $P_{-1}=1$ but $P_{-5}=0$. For the rest $10$ baskets with ``?'' marked, we discuss them in more details as followings.

If $B_X=\{(1,2), (2,5),(1,3),(1,4),(1,9)\}$, then $P_{-8}=2$ and we dealt it in Theorem \ref{thm P8=2}.

If $B_X=\{(1,2), (2,5),(2,7),(1,9)\}$, we know that $\dim \overline{\varphi_{-m}(X)}>1$ for all $m\geq 36$ from the list. Note that $P_{-18}=22$. Take $m_0=7$. If $|-18K_X|$ and $|-7K_X|$ are not composed with the same pencil, then we may take $m_1=18$ and $\mu_0\leq 7$, and by Proposition \ref{criterion b}(3), $\varphi_{-m,X}$ is birational for all $m\geq 43$. Hence we may assume that $|-18K_X|$ and $|-7K_X|$ are composed with the same pencil, then we may take $\mu_0\leq \mu'_0= \frac{18}{21}$ by Remark \ref{5.3}. We go on studying this situation by considering $|-31K_X|$. We have $P_{-31}=96$. If $|-31K_X|$ is not composed with a pencil, then we may take $m_1=31$ and by Proposition \ref{criterion b}(3), $\varphi_{-m}$ is birational for $m\geq 49$; if $|-31K_X|$ is composed with a pencil, then by Proposition \ref{prop new}, keep the same notation as in Subsection \ref{b setting},
$$
3<\frac{96}{31}=\frac{P_{-31}-1}{31}\leq \max\{3, -K_X^3, 2N_0\}\leq 2N_0,
$$
where $N_0=r_X(\pi^*(K_X)^2\cdot S)$. This implies that $N_0\geq 2$, and by Proposition \ref{criterion b2}, $\varphi_{-m,X}$ is birational for $m\geq 51$.

If $B_X=\{(3,7), (1,3),(1,4),(1,9)\}$, we know that $\dim \overline{\varphi_{-m}(X)}>1$ for all $m\geq 29$ from the list. Note that $P_{-17}=20$. Take $m_0=7$. If $|- 17K_X|$ and $|-7K_X|$ are not composed with the same pencil, then we may take $m_1= 17$ and $\mu_0\leq 7$, and by Proposition \ref{criterion b}(3), $\varphi_{-m}$ is birational for all $m\geq 42$; if $|- 17K_X|$ and $|-7K_X|$ are composed with the same pencil, then we may take $m_1=29$ and $\mu_0\leq \frac{17}{19}$ by Remark \ref{5.3}, and by Proposition \ref{criterion b}(3), $\varphi_{-m}$ is birational for all $m\geq 47$.

If $B_X=\{(3,7),(1,3),(1,4),(1,10)\}$, we know that $\dim \overline{\varphi_{-m}(X)}>1$ for all $m\geq 29$ from the list. Note that $P_{-14}=17$. Take $m_0=7$. If $|- 14K_X|$ and $|-7K_X|$ are not composed with the same pencil, then we may take $m_1= 14$ and $\mu_0\leq 7$, and by Proposition \ref{criterion b}(3), $\varphi_{-m}$ is birational for all $m\geq 41$; if $|- 14K_X|$ and $|-7K_X|$ are composed with the same pencil, then we may take $m_1=29$ and $\mu_0\leq \frac{14}{16}$ by Remark \ref{5.3}, and by Proposition \ref{criterion b}(3), $\varphi_{-m}$ is birational for all $m\geq 49$.

If $B_X=\{(1,2), (2,5),(1,3),(1,4),(1,11)\}$, then $P_{-8}=2$ and we dealt it in Theorem \ref{thm P8=2}.

If $B_X=\{(1,2),(2,5),(2,7),(1,8)\}$, take $m_0=7$. Note that $P_{-22}=12$. If $|-22K_X|$ and $|-7K_X|$ are not composed with the same pencil, then we may take $m_1= 22$ and $\mu_0\leq 7$, and by Proposition \ref{criterion b}(3), $\varphi_{-m}$ is birational for all $m\geq 45$; if $|- 22K_X|$ and $|-7K_X|$ are composed with the same pencil, then we may take $\mu'_0= \frac{22}{P_{-22}-1}=2$ by Remark \ref{5.3}, and by Proposition \ref{criterion b2}, $\varphi_{-m}$ is birational for all $m\geq 49$.

If $B_X=\{(3,7),(1,3),(1,5),(1,6)\}$, take $m_0=5$. Then we may take $\mu'_0\leq m_0=5$ by Remark \ref{5.3}, and by Proposition \ref{criterion b2}, $\varphi_{-m}$ is birational for all $m\geq 45$.

If $B_X=\{(3,7),(1,3),(2,11)\}$, take $m_0=5$. Then we may take $\mu'_0\leq m_0=5$ by Remark \ref{5.3}, and by Proposition \ref{criterion b2}, $\varphi_{-m}$ is birational for all $m\geq 48$.

If $B_X=\{(1,2),(2,5),(1,3),(2,11)\}$, take $m_0=5$. Note that $P_{-24}=16$. If $|- 24K_X|$ and $|-5K_X|$ are not composed with the same pencil, then we may take $m_1= 24$ and $\mu_0\leq 5$, and by Proposition \ref{criterion b}(3), $\varphi_{-m}$ is birational for all $m\geq 51$; if $|- 24K_X|$ and $|-5K_X|$ are composed with the same pencil, then we may take $\mu'_0=\frac{24}{P_{-24}-1}= \frac{24}{15}$ by Remark \ref{5.3}, and by Proposition \ref{criterion b2}, $\varphi_{-m}$ is birational for all $m\geq 52$. Note that this is in fact the only case that $\varphi_{-51}$ might not be birational.

If $B_X=\{2\times(1,2),2\times(1,3),(2,13)\}$, we know that $\dim \overline{\varphi_{-m}(X)}>1$ for all $m\geq 20$ from the list. Note that $P_{-13}=14$. Take $m_0=6$. If $|- 13K_X|$ and $|-6K_X|$ are not composed with the same pencil, then we may take $m_1= 13$ and $\mu_0\leq 6$, and by Proposition \ref{criterion b}(3), $\varphi_{-m}$ is birational for all $m\geq 45$; if $|- 13K_X|$ and $|-6K_X|$ are composed with the same pencil, then we may take $m_1=20$ and $\mu_0\leq 1$ by Remark \ref{5.3}, and by Proposition \ref{criterion b}(3), $\varphi_{-m}$ is birational for all $m\geq 47$.
\medskip

{\bf Case 2.} $P_{-1}=2.$ In this case, by \cite[Proof of Theorem 4.4, Case III]{CC}, 
 $$B^{(0)}=\{(1,2),(1,3),(1,s)\}$$ for some $s\geq 7$. Note that only $$B_X=B^{(0)}=\{(1,2),(1,3),(1,7)\}$$ satisfies $-K^3(B_X)<1/30$. In this case, we may apply Lemmas \ref{lem 12 pencil} and \ref{lem 12}(2).
 
 Combining all above discussions, the proof is completed.
\end{proof}

\section*{\bf Acknowledgment}

This project started while the authors were enjoying the workshop ``Higher Dimensional Algebraic Geometry, Holomorphic Dynamics and Their Interactions" at Institute for Mathematical Sciences, National University of Singapore. The authors are grateful for the hospitality and support of IMS-NUS. Many ideas of this paper came out during the second author's visit to Fudan University in March 2017 and he would like to thank School of Mathematical Sciences and Shanghai Center for Mathematical Sciences for the support and hospitality. 


\end{document}